\documentclass[final,3p,times]{elsarticle}

\usepackage{booktabs} 
\usepackage{amssymb}
\usepackage{amsmath}
\usepackage{amsthm}
\usepackage{mathtools}
\usepackage[inline]{enumitem}
\usepackage[latin1]{inputenc}
\usepackage{xy}
\xyoption{all}
\usepackage{tikz}
\usepackage{xcolor}
\usepackage{hyperref}

 \newtheorem{theorem}{Theorem} 
 \newtheorem{lemma}[theorem]{Lemma}
 \newtheorem{corollary}[theorem]{Corollary}
 \newtheorem{proposition}[theorem]{Proposition} 
 \newdefinition{remark}[theorem]{Remark}
 \newdefinition{definition}[theorem]{Definition}
 \newdefinition{example}[theorem]{Example}
 \newdefinition{fact}[theorem]{Fact}
 \newdefinition{notation}[theorem]{Notation}
 \newdefinition{framework}{Framework}
\newcommand{\pos}[1]{{\langle}#1{\rangle}} 
\newcommand{\FOL}{{\mathsf{FOL}}}

\newcommand{\HDFOLR}{{\mathsf{HDFOLR}}}
\newcommand{\HFOLS}{{\mathsf{HFOLS}}}
\newcommand{\HDPL}{{\mathsf{HDPL}}}
\newcommand{\HFOLR}{{\mathsf{HFOLR}}}
\newcommand{\HPL}{{\mathsf{HPL}}}

\newcommand{\RFOHL}{{\mathsf{RFOHL}}}

\newcommand{\Sig}{\mathsf{Sig}}
\newcommand{\Mod}{\mathsf{Mod}}
\newcommand{\Sen}{\mathsf{Sen}}
\newcommand{\alt}{\enspace|\enspace}
\newcommand{\act}{\mathfrak{a}}
\newcommand{\ari}{{\mathsf{ar}}}
\newcommand{\nom}{{\mathsf{n}}}
\newcommand{\card}{{\mathsf{card}}}
\newcommand{\rigid}{{\mathtt{r}}}
\newcommand{\flexible}{{\mathtt{f}}}
\newcommand{\cons}{{\mathtt{c}}}
\newcommand{\loose}{{\mathtt{l}}}
\newcommand{\ext}{{\mathtt{e}}}
\newcommand{\any}{{\mathtt{any}}}

\newcommand{\Cat}{\mathbb{C}\mathsf{at}}

\newcommand{\Set}{\mathbb{S}\mathsf{et}}
\newcommand{\Union}{\textit{(Union)} }
\newcommand{\Monotonicity}{\textit{(Monotonicity)}}
\newcommand{\Transitivity}{\textit{(Transitivity)}}
\newcommand{\Translation}{\textit{(Translation)}}
\newcommand{\N}{{\mathbb{N}}}
\renewcommand{\P}{\mathcal{P}}
\renewcommand{\L}{\mathcal{L}}

\newcommand{\F}{\mathsf{F}}

\newcommand{\A}{\mathfrak{A}}

\newcommand{\dra}{\dashrightarrow}
\newcommand{\x}{\mathtt{x}}
\newcommand{\z}{\mathtt{z}}
\newcommand{\Prop}{\mathtt{Prop}}
\usepackage{scalerel}

\newcommand{\bbsemicolon}{%
  \scalerel*{%
    \hbox{\usefont{U}{bbold}{m}{n} ;}%
  }{;}%
}
\newcommand{\comp}{\mathbin{\bbsemicolon}}
\newcommand{\modelsm}{\mathrel{\models\!\!\!|}}
\newcommand{\red}{\!\upharpoonright\!}
\newcommand{\At}[1]{\mathsf{at}_{#1}\, }
\newcommand{\at}[1]{@_{#1}\,}

\newcommand{\store}[1]{{\downarrow}#1\,{\cdot}\,}
\newcommand{\Forall}[1]{\forall #1\,{\cdot}\,}
\newcommand{\Exists}[1]{\exists #1\,{\cdot}\,}
\newcommand{\inferrule}[2]{{\small$\displaystyle\frac{#1}{#2}$}}
\newcommand{\Space}{\hspace{4mm}}
\usepackage{calc}
\usepackage{longtable}

\usepackage{array,etoolbox}

\makeatletter

\newlength{\PS@lastparam}
\newlength{\PSlastparam}
\newcommand{\PSlp}{%
  \setlength{\PSlastparam}{\PS@lastparam}%
  \the\PSlastparam
}
\def\PS@sub@lastparam{}

\newcommand{\PS@numwidth}{99}
\newcommand{\PSnumwidth}[1]{%
  \renewcommand{\PS@numwidth}{#1}%
}

\newcommand{\PS@style}{\small}
\newcommand{\PS@numstyle}{\footnotesize}

\newlength{\PSindent}
\newlength{\PS@extraindent}
\newlength{\PSpre}
\newlength{\PSpost}
\newlength{\PS@Nwidth}
\newlength{\PS@Swidth}
\newlength{\PS@Ewidth}
\newlength{\PScolsep}
\newcommand{\PS@rownumber}{%
  \ifPS@subsubsteps
  \thePSsubstepc.%
  \the\numexpr\value{PSsubsubstepc}+1\relax
  \else
  \ifPS@substeps
  \thePSstepc.%
  \the\numexpr\value{PSsubstepc}+1\relax
  \else
  \the\numexpr\value{PSstepc}+1\relax
  \fi\fi
}
\newcommand{\PS@step}{%
  \ifPS@subsubsteps
  \refstepcounter{PSsubsubstepc}%
  \else
  \ifPS@substeps
  \refstepcounter{PSsubstepc}%
  \else
  \refstepcounter{PSstepc}
  \fi\fi%
}

\newif\ifPS@inprogress
\newif\ifPS@substeps
\newif\ifPS@subsubsteps
\newif\ifPS@continued
\newif\ifPS@subcontinued
\newcounter{PSc}
\newcounter{PSstepc}[PSc]
\newcounter{PSsubstepc}[PSstepc]
\renewcommand{\thePSsubstepc}{\thePSstepc.\arabic{PSsubstepc}}
\newcounter{PSsubsubstepc}[PSsubstepc]

\newenvironment{proofsteps}[1]{%
  \global\settowidth{\PS@lastparam}{\PS@style\hspace*{#1}}
  \ifPS@continued\else\refstepcounter{PSc}\fi
  \begingroup
  \setlength{\LTpre}{\PSpre}%
  \setlength{\LTpost}{\PSpost}%
  
  \setlength{\tabcolsep}{0pt}
  \noindent\PS@style
  \settowidth{\PS@Nwidth}{\PS@numstyle\PS@numwidth}%
  \setlength{\PS@Swidth}{#1}%
  \addtolength{\PS@Swidth}{-\PS@extraindent}%
  \setlength{\PS@Ewidth}{\linewidth}%
  \addtolength{\PS@Ewidth}{-\PSindent}%
  \addtolength{\PS@Ewidth}{-\PS@extraindent}%
  \addtolength{\PS@Ewidth}{-\PS@Nwidth}%
  \addtolength{\PS@Ewidth}{-\PScolsep}%
  \addtolength{\PS@Ewidth}{-\PS@Swidth}%
  \addtolength{\PS@Ewidth}{-\PScolsep}%
  \PS@inprogresstrue
  \longtable{%
    @{\hspace*{\PSindent}\hspace*{\PS@extraindent}\makebox[\PS@Nwidth][r]{\PS@rownumber}}%
    @{\hskip\PScolsep}>{\PS@step}p{\PS@Swidth}%
    @{\hskip\PScolsep}>{\footnotesize\raggedright\arraybackslash}p{\PS@Ewidth}%
  }%
}{%
  \ifPS@inprogress
  \addtocounter{table}{-1}%
  \endlongtable  
  \endgroup
  \PS@continuedfalse
  \PS@inprogressfalse
  \else\fi
}

\newcommand{\PSbreak}[1]{%
  \endproofsteps
  \par\medskip
  #1
  \medskip\par
  \PS@continuedtrue
  \proofsteps{\PS@lastparam}%
}

\newif\ifPS@sub@inprogress
\newenvironment{substeps}[1]{%
  \global\def\PS@sub@lastparam{#1}%
  \endproofsteps
  \settowidth{\PS@extraindent}{\PS@numstyle\PS@numwidth}%
  \addtolength{\PS@extraindent}{\PScolsep}%
  \addtolength{\PS@extraindent}{\widthof{\PS@numstyle.#1}}%
  \PS@substepstrue
  \PS@continuedtrue
  \PS@sub@inprogresstrue
  \proofsteps{\PS@lastparam}%
}{%
  \ifPS@sub@inprogress
    \endproofsteps
    \ifPS@laststep
      \PS@laststepfalse
      \PS@inprogressfalse
    \else
      \setlength{\PS@extraindent}{0pt}%
      \PS@substepsfalse
      \PS@continuedtrue
      \proofsteps{\PS@lastparam}%
    \fi
    \PS@sub@inprogressfalse
  \else
    \ifPS@laststep
      \PS@laststepfalse
      \PS@inprogressfalse
    \else
      \setlength{\PS@extraindent}{0pt}%
      \PS@substepsfalse
      \PS@continuedtrue
      \proofsteps{\PS@lastparam}%
    \fi
  \fi
}

\newif\ifPS@laststep
\newcommand{\laststep}{\global\PS@laststeptrue}
\newif\ifPS@lastsubstep
\newcommand{\lastsubstep}{\global\PS@lastsubsteptrue}

\newcommand{\adjustcol}[1]{%
  \global\advance\@colroom-#1%
}

\makeatother

\newcommand{\psqed}{%
  \vspace{-\baselineskip}\vspace{-1\smallskipamount}
}

\newcommand*{\pcformat}[1]{%
  [\;{\normalfont\itshape #1}\;]%
}

\newenvironment{proofcases}[1][]{%
  \description[font=\pcformat, leftmargin=\parindent, #1]%
}{\enddescription}

\journal{Annals of Pure and Applied Logic}

\begin{document}

\begin{frontmatter}



\title{Omitting Types Theorem in hybrid-dynamic first-order logic with rigid symbols}


\author[a1]{Daniel G\u{a}in\u{a}}
\ead{daniel@imi.kyushu-u.ac.jp} 

\author[a2]{Guillermo Badia}
\ead{guillebadia89@gmail.com}

\author[a3]{Tomasz Kowalski}
\ead{T.Kowalski@latrobe.edu.au}

\address[a1]{Institute of Mathematics for Industry, Kyushu University} 
\address[a2]{The University of Queensland \& Johannes Kepler University Linz} 
\address[a3]{La Trobe University}

\begin{abstract}
In the the present contribution, we prove an Omitting Types Theorem (OTT) for an arbitrary fragment of hybrid-dynamic first-order logic with rigid symbols (i.e. symbols with fixed interpretations across worlds) closed under \emph{negation} and \emph{retrieve}.
The logical framework can be regarded as a parameter and it is instantiated by some well-known hybrid and/or dynamic logics from the literature.
We develop a \emph{forcing} technique and then we study a \emph{forcing property} based on local satisfiability, which lead to a refined proof of the OTT. 
For uncountable signatures, the result requires compactness, while for countable signatures, compactness is not necessary.
We apply the OTT to obtain upwards and downwards  L\" owenheim-Skolem theorems for our logic, as well as a completeness theorem for its \emph{constructor-based} variant. The main result of this paper can easily be recast in the institutional model theory framework, giving it a higher level of generality.
\end{abstract}



\begin{keyword}
 Institution \sep hybrid logic \sep dynamic logic \sep forcing \sep Omitting Types Theorem



\end{keyword}

\end{frontmatter}


\section{Introduction}

\paragraph{Kripke semantics and hybrid-dynamic logics} 
Modal logics are formalisms for describing and reasoning about multi-graphs.
These structures appear naturally in many areas of research.
For example, in knowledge representation formalisms, role assertions describe relationships between individuals/objects grouped into classes determined by concepts.
Linguistic information can be represented by multi-graphs.
Other mathematical entities that can be viewed as multi-graphs are transition systems, derivation trees, semantic networks, etc.
Therefore, it is useful to think of a Kripke structure in the following way:
\begin{itemize}
\item a frame consisting of a set of nodes together with a family of (typed) edge sets, and 
\item a mapping from the set of nodes to a class of local models that gives meaning to the nodes.
\end{itemize}
However, modal logics have no mechanisms for referring to the individual nodes in such structures, which is necessary when they are used as representation formalisms.
Hybrid logics increase the expressive power of ordinary modal logics by adding an additional sort of symbols called \emph{nominals} such that each nominal is true relative to exactly one point.
The history of hybrid logics goes back to Arthur Prior's work~\cite{prior67}. 
Further developments can be found in works such as \cite{ArecesB01,ArecesBM01,ArecesBM03,brau11}.
The research on hybrid logics received an additional boost due to the recent interest in the logical foundations of the \emph{reconfiguration paradigm}.
Dynamic logics provide a powerful language for describing programs and reason about their correctness.
Logics of programs have the roots in the work in the late 1960s of computer scientists interested in assigning meaning to programming languages and finding a rigorous standard for proofs about the programs.
There is a significant body of research on this topic;
\cite{passay91} and \cite{HarelKT01} are two prominent examples among many others.
In the present contribution, we consider a logical system endowed with features from both hybrid and dynamic logics, which is built on top of many-sorted first-order logic with equality. 
Despite its complexity, it displays a certain simplicity due to its modular construction, which is a reminiscent of the hybridization of institutions from~\cite{martins}.

\paragraph{Applications of hybrid-dynamic logics}
The application domain of the work reported in this contribution refers to a broad range of reconfigurable systems
whose states or configurations can be presented explicitly, based on some kind of context-independent data types,
and for which we distinguish the computations performed at the local/configuration level from the dynamic evolution of the configurations.
This suggests a two-layered approach to the design and analysis of reconfigurable systems, involving:
\begin{itemize}
\item \emph{a local view}, which amounts to describing the structural properties of configurations, and
\item \emph{a global view}, which corresponds to a specialized language for specifying and reasoning about the way system configurations evolve.
\end{itemize}
Since configurations can be represented by local models and the dynamic evolution of configurations can be depicted by the accessibility relations of the Kripke structures, hybrid-dynamic logics and their fragments are acknowledged as suitable for describing and reasoning about systems with reconfigurable features.
In addition, it is well-known (see~e.g.,~\cite{Blackburn00}) that hybrid logics specialize to temporal logics~\cite{DBLP:journals/jacm/HalpernS91}, description logics~\cite{baader2017} and feature logics~\cite{ROUNDS1997475}.
Therefore, the area of applications of the present work is
rather large and it involves knowledge representation, computational linguistics, artificial intelligence, biomedical informatics, semantic networks and ontologies.
We recommend \cite{Blackburn00} for more information on this topic.

\paragraph{Omitting Types Theorem (OTT)}
In this paper we focus on obtaining an OTT for hybrid-dynamic first-order logic with rigid symbols and sufficiently expressive fragments.  
Observe that an OTT  for the full logic would not necessarily have given us the property for its fragments. 
For this reason, we work within an arbitrary fragment of hybrid-dynamic first-order logic with rigid symbols, which can be viewed as a parameter.   
Thus the generality of our proofs is an important feature, since the parameter is instantiated by many concrete hybrid and/or dynamic logical systems which appear in the literature.
We provide a version of OTT for countable languages without any restrictions and a version for uncountable languages provided that the fragment in question is compact. 
We show that compactness is necessary at least for one fragment of the underlying logic. 
This situation is similar to that described in a  theorem by Lindstr\"om for first-order logic with only relational symbols~\cite{lind78}. 
The OTT for countable first-order languages is a result originally from~\cite{gr1961}. 
The extension of the OTT to uncountable languages is from \cite{chang1964}. 
One of the best known applications of the OTT is a simple proof of the completeness of $\omega$-logic 
(a more complex proof without using the OTT can be found in~\cite{orey1956}).
In the present contribution, we develop this idea further to provide one important application of OTT to computer science, which is described briefly in the following paragraph.

Formal  methods practitioners are often interested in properties that are true of a restricted class of models whose elements are reachable by some constructor operations~\cite{Bidoit-COL,gai-cbl,gai-int}. 
For this reason, several algebraic specification languages incorporate features to express reachability and to deal with constructors like, for instance, Larch \cite{guttag-larch}, CASL \cite{astesiano-CASL} or CITP \cite{citp13,gai-apsec}.
This situation is similar to the one in classical model theory, where the models of $\omega$-logic are reachable by the constructors \emph{zero} and \emph{successor}.
In the present contribution, the completeness  of $\omega$-logic is generalized by replacing the signature of arithmetics with an arbitrary vocabulary for which we distinguish a set of constructor operators.
Then we apply OTT to obtained completeness of the logical system resulted from restricting the semantics of the underlying fragment of hybrid-dynamic first-order logic with rigid symbols to constructor-based Kripke structures.
 
\paragraph{Institutions} 
Our approach is rooted in institutional model theory \cite{ins}, which provides a unifying setting for studying logical systems using category theory.
The concept of institution formalizes the intuitive notion of logic, including syntax, semantics and the satisfaction relation between them.
The theory of institutions is one major approach in universal logic which promotes the development of logical properties at the most general level of abstraction.
However, to make the study available to a broader audience, the authors decided to present the results in a framework given by a concrete logical system, that is, hybrid-dynamic first-order logic with rigid symbols.
It should be obvious, at least for the experts in institutions, that the main result, OTT, can be easily cast in a more general framework such as the one provided by the definition of stratified institution~\cite{dia-ult-kripke}, similarly to the work reported in~\cite{gai-godel}. 
Therefore, the area of applications of our results covers a much broader range of hybrid-dynamic logics than the one mentioned in the present contribution.
 
\paragraph{Forcing} 
OTT is proved by applying a forcing technique, a method of constructing models based on consistency results.
Forcing was invented by Paul Cohen~\cite{cohen63,cohen64} in set theory to prove the independence of the continuum hypothesis from the other axioms of Zermelo-Fraenkel set theory. 
Robinson~\cite{rob71} developed an analogous forcing method in model theory.
In institutional model theory, forcing was introduced in \cite{gai-com} to prove a G\"{o}del Completeness Theorem. 
It was developed further for stratified institutions~\cite{gai-godel} to prove the completeness of a large class of hybrid logics.
The present contribution extends the forcing introduced in  \cite{gai-godel} to cover logics with both hybrid and dynamic features and studies a forcing property based on local satisfiability to deliver an Omitting Types Theorem.

\paragraph{Structure of the paper}
The article is arranged as follows: \S \ref{1} reviews the framework of
many-sorted first-order logic in the institutional setting. \S \ref{2}
introduces all the necessary preliminaries about hybrid dynamic first-order
logic  with rigid symbols, which expands the base system introduced  in  \S
\ref{1}.  \S\ref{3} presents some necessary technical notions for the arguments
that follow later, such as that of a reachable model and a language
fragment. \S\ref{4} develops the basics of the forcing technique in our present
context.  \S\ref{5} presents a semantic forcing property that is instrumental in
proving the main result of the paper. \S\ref{6} contains the proof of the main
result, an Omitting Types Theorem for both countable and uncountable
signatures. \S\ref{7} gives an application of the main result by establishing a
completeness theorem for the constructor-based variant of the logic. \S\ref{8}
establishes L\"owenheim-Skolem theorems (upwards and downwards) as consequences
of the OTT. \S\ref{9} shows that for a certain fragment of the logic in question
compactness is a necessary condition for the OTT for uncountable signatures to
hold.

\section{Many-sorted first-order logic (\textbf{FOL})}\label{1}
 In this section, we recall the definition of first-order logic as presented in institutional model theory \cite{ins}.
\paragraph{Signatures}
Signatures are of the form $(S,F,P)$, where $S$ is a set of sorts,
$F=\{F_{\ari\to s}\}_{(\ari,s)\in S^*\times S}$ is a \textup{(}$S^*\times S$
-indexed\textup{)}
set of operation symbols, and $P=\{P_\ari\}_{\ari\in S^*}$ is a
\textup{(}$S^*$-indexed\textup{)}
set of relation symbols.
If $\ari=\varepsilon$ then an element of $F_{\ari\to s}$ is called a
\emph{constant symbol}. Generally, $\ari$ ranges over arities, which are understood here as strings of sorts; in other words an arity gives the number of arguments together with their sorts.
We overload the notation and let $F$ and $P$ also denote $\biguplus_{(\ari,s)\in
  S^*\times S}F_{\ari\to s}$ and $\biguplus_{\ari\in S^*}P_\ari$, respectively.  
Therefore, we may write $\sigma\in F_{\ari\to s}$ or $(\sigma\colon\ari\to s)
\in F$; both have the same meaning, which is: $\sigma$ is an operation symbol
of type $\ari\to s$.  
Throughout this paper, we let $\Sigma$, $\Sigma'$ and $\Sigma_i$ to range over first-order signatures of the form $(S,F,P)$, $(S',F',P')$ and $(S_i,F_i,P_i)$, respectively. 

\paragraph{Signature morphisms}
A number of usual tricks, such as adding constants, but
also, importantly, quantification, are viewed as expansions of  
the signature, so moving between signatures is common.
To make such transitions smooth, a notion of a \emph{signature morphism} is introduced. 
A signature morphism $\varphi\colon \Sigma\to \Sigma'$ is a triple
$\chi=(\chi^{st},\chi^{op},\chi^{rl})$ of maps:
\begin{enumerate*}[label=(\alph*)]
\item~$\chi^{st}\colon S\to S'$,
\item~$\chi^{op}=\{\chi^{op}_{\ari\to s}\colon F_{\ari\to s}\to F'_{\chi^{st}(\ari)\to \chi^{st}(s)}\mid \ari\in S^*,s\in S\}$, and
\item~$\chi^{rl}=\{\chi^{rl}_{\ari}\colon P_{\ari}\to P'_{\chi^{st}(\ari)}\mid \ari\in S^*\}$.
\end{enumerate*}
When there is no danger of confusion, we may let $\chi$ denote either of
$\chi^{st}$, $\chi^{op}_{\ari\to s}$, $\chi^{rl}_{\ari}$. 
\begin{fact}
First-order signature morphisms form a category $\Sig^\FOL$ under the
componentwise composition as functions.
\end{fact}
\paragraph{Models} Given a signature $\Sigma$, a $\Sigma$-model is a triple
$$
\A=(\{\A_s\}_{s\in S},\{\A_\sigma\}_{(\ari,s)\in S^*\times S,\sigma\in F_{\ari\to
    s}},\{\A_\pi\}_{\ari\in S^*,\pi\in P_\ari})
$$
interpreting each sort $s$ as a non-empty set $\A_s$, each operation symbol $\sigma\in F_{\ari\to s}$ as a function $\A_\sigma\colon \A_\ari\to \A_s$ (where $\A_\ari$ stands for $\A_{s_1}\times\ldots\times \A_{s_n}$ if $\ari=s_1\ldots s_n$), and each relation symbol $\pi\in P_\ari$ as a relation $\A_\pi\subseteq \A_\ari$.  
Morphisms between models are the usual $\Sigma$-homomorphisms, i.e., $S$-sorted
functions that preserve the structure.
\begin{fact}
For any signature $\Sigma$, the $\Sigma$-homomorphisms form a category $\Mod^\FOL(\Sigma)$ under the obvious composition as many-sorted functions. 
\end{fact}
For any signature morphism $\chi\colon \Sigma\to \Sigma'$, 
the reduct functor $\_\red_\chi\colon\Mod(\Sigma')\to\Mod(\Sigma)$ is defined as follows: 
\begin{enumerate}
\item The reduct $\mathfrak{A}'\red_\chi$ of a $\Sigma'$-model $\mathfrak{A}'$ is a defined by
  $({\A'\red_\chi})_x=\A'_{\chi(x)}$ for each sort $s\in S$, operation symbol $x\in F$ or relation symbol $x\in P$.
 Note that, unlike the single-sorted case, the reduct functor modifies the universes of models. 
 For the universe of $\A'\red_\chi$ is $\{\A'_{\chi(s)}\}_{s\in S}$, which means that the sorts outside the image of $S$ are discarded. 
 Otherwise, the notion of reduct is standard. 

\item The reduct $h'\red_\chi$ of a homomorphism $h'$ is defined by
  $(h'\red_\chi)_s=h'_{\chi(s)}$ for all sorts $s\in S$. 
\end{enumerate}
\begin{fact}
$\Mod^\FOL$ becomes a functor $\Sig^\FOL\to \Cat^{op}$, with $\Mod^\FOL(\chi)(h') = {h'\red_\chi}$ for each signature morphism $\chi\colon\Sigma\to\Sigma'$ and each $\Sigma'$-homomorphism $h'$.
\end{fact}
\paragraph{Sentences} 
We assume a countably infinite set of variable names $\{v_i\mid i<\omega\}$.
A variable for a signature $\Sigma$ is a triple $\pos{v_i,s,\Sigma}$, where $v_i$ is a variable name, and $s$ is a sort in $\Sigma$.
Given a signature $\Sigma$, the $S$-sorted set of $\Sigma$-terms is denoted by $T_\Sigma$.
The set $\Sen^\FOL(\Sigma)$ of sentences over $\Sigma$ is given by the following grammar:
$$\gamma \Coloneqq  t= t' \mid \pi(t_1,\ldots,t_n)\mid \neg\gamma \mid \vee\Gamma \mid \Exists{X}\gamma'$$ 
where
\begin{enumerate*}[label=(\alph*)]
\item~$t= t'$ is an equation with $t,t'\in T_{\Sigma,s}$ and $s\in S$,
\item~$\pi(t_1,\ldots,t_n)$ is a relational atom with $\pi\in P_{s_1\ldots s_n}$, $t_i\in T_{\Sigma,s_i}$ and $s_i\in S$, 
\item~$\Gamma$ is a finite set of $\Sigma$-sentences,
\item~$X$ is a finite set of variables for $\Sigma$,
\item~$\gamma'$ is a $\Sigma[X]$-sentence, where $\Sigma[X]=(S,F[X],P)$, and $F[X]$ is the set of function symbols obtained by adding the variables in $X$ as constants to $F$.
\end{enumerate*}
\paragraph{Sentence translations} 
Quantification comes with some subtle issues related to the translation of quantified sentences along signature morphisms that require a closer look.
The translation of a variable $\pos{v_i,s,\Sigma}$ along a signature morphism $\chi\colon\Sigma\to\Sigma'$ is $\pos{v_i,\chi(s),\Sigma'}$.
Therefore, any signature morphism $\chi\colon\Sigma\to\Sigma'$ can be extended canonically to a function $\chi\colon\Sen^\FOL(\Sigma)\to\Sen^\FOL(\Sigma')$ that translates sentences symbolwise. 
$$\xymatrix{ 
 \Sigma[X] \ar[r]^{\chi'} & \Sigma'[X']\\
 \Sigma \ar@{^{(}->}[u] \ar[r]_{\chi} & \Sigma' \ar@{^{(}->}[u]
}$$
Notice that $\chi(\Exists{X}\gamma)=\Exists{X'}\chi'(\gamma)$, where 
$X'=\{\pos{v_i,\chi(s),\Sigma'} \mid \pos{v_i,s,\Sigma}\in X\}$ and
$\chi'\colon\Sigma[X]\to \Sigma'[X']$ is the extension of $\chi$ that maps each
variable $\pos{v_i,s,\Sigma}\in X$ to $\pos{v_i,\chi(s),\Sigma'}\in X'$ and
such that the diagram of signature morphisms above is commutative.
\begin{fact}
 $\Sen^\FOL$ is a functor $\Sig^\FOL \to \Set$.
\end{fact}
For the sake of simplicity, we will identify a variable only by its name and sort provided that there is no danger of confusion.
Using this convention, each inclusion $\iota\colon\Sigma\hookrightarrow\Sigma'$
is canonically extended to an inclusion of sentences
$\iota\colon\Sen^\FOL(\Sigma)\hookrightarrow\Sen^\FOL(\Sigma')$, which
corresponds to the approach of classical model theory.
\paragraph{Satisfaction relation} Satisfaction is the usual first-order
satisfaction and it is defined using the natural interpretations of ground terms
$t$ as elements $\A_t$ in models $\A$. 
For example, $\A\models t_1= t_2$ iff $\A_{t_1}=\A_{t_2}$.
\paragraph{Non-void signatures}
 A first-order signature $\Sigma$ is called \emph{non-void} if all sorts in $\Sigma$ are inhabited by terms, that is $T_{\Sigma,s}\neq\emptyset$ for all sorts $s$ in $\Sigma$.  
If $\Sigma$ is a \emph{non-void} signature then the set of $\Sigma$-terms $T_\Sigma$ can be regarded as a first-order model which interprets 
\begin{enumerate*}[label=(\alph*)]
 \item any function symbol $(\sigma\colon\ari\to s)\in F$ as a function $T_{\Sigma,\sigma}\colon T_{\Sigma,\ari}\to T_{\Sigma,s}$ defined by $T_{\Sigma,\sigma}(t)=\sigma(t)$ for all $t\in T_{\Sigma,\ari}$, and
 \item any relation symbol as the empty set.
\end{enumerate*}
\paragraph{Notations}
For each first-order signature $\Sigma$, we denote by $\bot$ the $\Sigma$-sentence $\vee\emptyset$.
Obviously, $\bot$ is not satisfiable and $\chi(\bot)=\bot$ for all signature morphisms $\chi\colon\Sigma\to \Sigma'$.
Let $T$ and $\Gamma$ be two theories over $\Sigma$.
\begin{itemize}
\item  $\A\models T$ if $\A\models \varphi$ for all $\varphi\in T$, where $\A$ is any first-order $\Sigma$-structure.
\item $T\models \Gamma$ if for all first-order structures $\A$ over $\Sigma$, we have $\A\models T$ implies $\A\models\Gamma$.
\item $T\modelsm\Gamma$ if $T\models \Gamma$ and $\Gamma\models T$.
In this case, we say that $T$ and $\Gamma$ are semantically equivalent.
\end{itemize}

\section{Hybrid-dynamic first-order logic with rigid symbols (\textbf{HDFOLR})}\label{2}
 In this section, we present hybrid-dynamic first-order logic with rigid symbols, which is an extension of hybrid first-order logic with rigid symbols~\cite{gai-godel} with features of dynamic logics.
 Some preliminary attempts to the presentation of this logic framework can be found in \cite{gai-dbir}.
\paragraph{Signatures}
The signatures are of the form $\Delta=(\Sigma^\nom,\Sigma^\rigid\subseteq\Sigma)$, where
\begin{enumerate}
\item $\Sigma^\nom=(S^\nom,F^\nom,P^\nom)$ is a single-sorted first-order signature such that
$S^\nom=\{\any\}$ is a singleton,  
$F^\nom$ is a set of constants called \emph{nominals}, and 
$P^\nom$ is a set of binary relation symbols called \emph{modalities},
\item $\Sigma=(S,F,P)$ is a many-sorted first-order signature such that
$S$ is a set of sorts, 
$F$ is a $(S^*\times S)$-indexed set of function symbols, and 
$P$ is a $S^*$-indexed set of relation symbols, and
\item $\Sigma^\rigid=(S^\rigid,F^\rigid,P^\rigid)$ is a many-sorted first-order
subsignature of \emph{rigid} symbols.
\end{enumerate}
Throughout this paper, we let $\Delta$ and $\Delta_i$ 
range over $\HDFOLR$ signatures of the form $(\Sigma^\nom,\Sigma^\rigid\subseteq\Sigma)$ and $(\Sigma_i^\nom,\Sigma_i^\rigid\subseteq\Sigma_i)$, respectively.
 \paragraph{Signature morphisms}
A \emph{signature morphism} $\chi \colon \Delta \to \Delta_1$ consists of a pair of first-order signature morphisms
$\chi^{\nom} \colon \Sigma^{\nom} \to \Sigma_1^{\nom}$ and $\chi \colon \Sigma \to \Sigma_1$ such that $\chi(\Sigma^{\rigid}) \subseteq \Sigma_1^{\rigid}$.
\begin{fact}
$\HDFOLR$ signature morphisms form a category $\Sig^\HDFOLR$ under the component-wise composition as first-order signature morphisms.
\end{fact}

\paragraph{Kripke structures}
  For every signature $\Delta$, the class of Kripke structures over $\Delta$
  consists of pairs $(W,M)$, where
  
\begin{enumerate}
\item $W$ is a first-order structure over $\Sigma^\nom$, called a frame,
with the universe $|W|$ consisting of a non-empty set of possible worlds, and
  
\item $M\colon|W|\to |\Mod^\FOL(\Sigma)|$ is a mapping from the universe of $W$ to the class of first-order $\Sigma$-structures such that the following sharing condition holds:  ${M_{w_1}\red_{\Sigma^\rigid}}={M_{w_2}\red_{\Sigma^\rigid}}$ for all possible worlds $w_1,w_2\in |W|$.
\end{enumerate}
\paragraph{Kripke homomorphisms}
A \emph{morphism} $h \colon (W, M) \to (W', M')$ is also a pair $(W\stackrel{h}\to W', \{M_{w}\stackrel{h_{w}}\to M'_{h(w)}\}_{w \in |W|})$
consisting of first-order homomorphisms such that $h_{w_{1}, s} = h_{w_{2}, s}$ for all possible worlds $w_{1}, w_{2} \in |W|$ and all rigid sorts $s \in S^{\rigid}$.
\begin{fact}
For any signature $\Delta$, the $\Delta$-homomorphisms form a category $\Mod^\HDFOLR(\Delta)$ under the component-wise composition as first-order homomorphisms.
\end{fact}
Every signature morphism $\chi \colon \Delta \to \Delta'$ induces appropriate \emph{reductions of models}, as follows:
every $\Delta'$-model $(W', M')$ is reduced to a $\Delta$-model $(W', M') \red_{\chi}$ that interprets every symbol $x$ in $\Delta$ as $(W', M')_{\chi(x)}$.
When $\chi$ is an inclusion, we usually denote $(W', M') \red_\chi$ by $(W', M') \red_\Delta$ -- in this case, the model reduct simply forgets the interpretation of those symbols in $\Delta'$ that do not belong to $\Delta$.
\begin{fact}
$\Mod^\HDFOLR$ becomes a functor $\Sig^\HDFOLR\to \Cat^{op}$, with $\Mod^\HDFOLR(\chi)(W,M) = (W,M)\red_\chi$ for each signature morphism $\chi\colon\Delta\to\Delta'$ and each Kripke structure $(W,M)$ over $\Delta'$.
\end{fact}
\paragraph{Actions}
As in dynamic logic, $\HDFOLR$ supports structured actions obtained from atoms using sequential composition, union, and iteration.
The set $A^{\nom}$ of \emph{actions} over $\Sigma^{\nom}$ is defined in an inductive fashion, according to the grammar:
$$\act \Coloneqq \lambda \mid \act \comp \act \mid \act \cup \act \mid \act^{*}$$
where $\lambda \in  P^\nom$ is a binary relation on nominals.
Given a natural number $m > 0$, we denote by $\act^{m}$ the composition $\act \comp \dotsb \comp \act$ (where the action $\act$ occurs $m$ times).
Actions are interpreted in Kripke structures as \emph{accessibility relations} between possible worlds.
This is done by extending the interpretation of binary modalities (from $ P^\nom$):
$W_{\act_{1} \comp \act_{2}} = W_{\act_{1}} \comp W_{\act_{2}}$ (diagrammatic composition of relations), 
$W_{\act_{1} \cup \act_{2}} = W_{\act_{1}}\cup W_{\act_{2}}$ (union), and
$W_{\act^{*}} = (W_{\act})^{*}$ (reflexive \textit{\&} transitive closure).
  
  \paragraph{Hybrid terms} 
For any signature $\Delta$, we make the following notational conventions:
\begin{enumerate}  
\item $S^\ext\coloneqq S^\rigid\cup\{\any\}$ the extended set of rigid sorts, where $\any$ is the sort of nominals,
\item $S^\flexible \coloneqq S \setminus S^{\rigid}$ the subset of flexible sorts,  
\item $F^\flexible\coloneqq F\setminus F^\rigid$ the subset of flexible function symbols, 
where $F\setminus F^\rigid=\{F_{\ari\to s}\setminus F^\rigid_{\ari\to s}\}_{(\ari,s)\in S^*\times S}$,
\item $P^\flexible\coloneqq P\setminus  P^\rigid$ the subset of flexible relation symbols, where $P\setminus P^\rigid=\{P_\ari\setminus P^\rigid_\ari\}_{\ari\in S^*}$. 
\end{enumerate}
The \emph{rigidification} of $\Sigma$ with respect to $ F^\nom$ is the signature $@\Sigma=(@S,@F,@P)$, where 
\begin{enumerate}
\item $@S\coloneqq \{\at{k} s \mid k\in F^\nom \mbox{ and } s\in S\}$,
\item $@F\coloneqq\{\at{k}\sigma\colon \at{k}\ari \to \at{k} s \mid k\in F^\nom \mbox{ and } (\sigma\colon \ari\to s) \in F \}$,
\footnote{$\at{k} (s_1\ldots s_n) \coloneqq \at{k} s_1\ldots\at{k} s_n$ for all arities $s_1\ldots s_n$.} and
\item $@P\coloneqq \{\at{k} \pi\colon \at{k} \ari \mid k\in F^\nom \mbox{ and }(\pi\colon\ari)\in P\}$.
\end{enumerate}
It should be noted that $\at{k}$ is used polymorphically. Here it is a device
from metalanguage that creates new symbols out of existing ones. Later on $\at{k}$ 
will also be used as a sentence-building operator. The context always decides 
which of these uses are intended. 
Since the rigid symbols have the same interpretation across the worlds, we define $\at{k} x\coloneqq x$ for all nominals $k\in F^\nom$ and all symbols $x$ in $\Sigma^\rigid$.
 The set of \emph{rigid $\Delta$-terms} is $T_{@\Sigma}$, while the set of \emph{open $\Delta$-terms} is $T_\Sigma$. 
 The set of \emph{hybrid $\Delta$-terms} is $T_{\overline\Sigma}$, where 
 $\overline\Sigma=(\overline{S},\overline{F},\overline{P})$, 
 $\overline{S}=S\cup @S^\flexible$,
 $\overline{F}=F\cup @F^\flexible$, and
 $\overline{P}=P\cup @P^\flexible$.
\begin{remark}
The set of hybrid terms include both open and rigid terms, that is, $T_\Sigma\subseteq T_{\overline\Sigma}$ and $T_{@\Sigma}\subseteq T_{\overline\Sigma}$.
\end{remark}
The interpretation of the hybrid terms into Kripke structures is defined as follows:
for any $\Delta$-model $(W,M)$, and any possible world $w\in|W|$, 
\begin{enumerate}
\item $M_{w,\sigma(t)} = (M_{w,\sigma})(M_{w,t})$, where $(\sigma\colon\ari\to s)\in F$, and $t\in T_{\overline\Sigma,\ari}$,
\footnote{$M_{w,(t_1,\ldots,t_2)}\coloneqq M_{w,t_1},\ldots,M_{w,{t_n}}$ for all tuples of hybrid terms $(t_1,\ldots,t_n)$.}
\item $M_{w,(\at{k} \sigma)(t)} = (M_{w',\sigma}) (M_{w,t})$, where $(\at{k} \sigma\colon\at{k} \ari\to\at{k} s)\in @F^\flexible$, $t\in T_{\overline\Sigma,\at{k}\ari}$ and $w'=W_k$.
\end{enumerate}
%
%
\paragraph{Sentences}
The simplest sentences defined over a signature $\Delta$,
usually referred to as atomic, are given by:
$$\rho \Coloneqq k \alt t_{1} = t_{2} \alt \pi(t)$$
where 
\begin{enumerate*}[label=(\alph*)]
\item $k,k' \in  F^\nom$ are nominals,
\item $t_i \in T_{\overline\Sigma,s}$ are hybrid terms, $s\in \overline{S}$ is a hybrid sort,
\item $\pi\in\overline{P}_\ari$, $\ari\in (\overline{S})^*$ and $t\in T_{\overline\Sigma,\ari}$.
\end{enumerate*}
We refer to these sentences, in order, as \emph{nominal sentences}, \emph{hybrid equations} and \emph{hybrid relations}, respectively.
The set $\Sen^\HDFOLR(\Delta)$ of \emph{full sentences} over $\Delta$ are given by the following grammar:
\[
  \gamma \Coloneqq
  \rho \alt
  \at{k} \gamma \alt
  \lnot \gamma \alt
  \textstyle\vee \Gamma \alt
  \store{z} \gamma' \alt
  \Exists{X} \gamma'' \alt
  \pos{\act} \gamma 
\]
where 
\begin{enumerate*}[label=(\alph*)]
\item $\rho$ is a nominal sentence or a hybrid equation or a hybrid relation,
\item $k \in  F^\nom $ is a nominal,
\item $\act \in A^{\nom}$ is an action,
\item $\Gamma$ is a finite set of sentences over $\Delta$,
\item $z$ is a nominal variable for $\Delta$,
\item $\gamma'$ is a sentence over the signature $\Delta[z]$ obtained by adding $z$ as a new constant to $ F^\nom$,
\item $X$ is a set of variables for $\Delta$ of sorts from the extended set $S^\ext$ of rigid sorts, and
\item$\gamma''$ is a a sentence over the signature $\Delta[X]$ obtained by adding the variables in $X$ as new constants to $ F^\nom$ and $F^{\rigid}$.
\end{enumerate*}
Other than the first kind of sentences (\emph{atoms}), we refer to the sentence-building operators, in order, as 
\emph{retrieve}, 
\emph{negation}, 
\emph{disjunction}, 
\emph{store}, 
\emph{existential quantification} and 
\emph{possibility}, 
respectively. Notice that \emph{possibility} is parameterized by actions. 
 
 \paragraph{Sentence translations}
Every signature morphism $\chi\colon\Delta\to\Delta'$ induces \emph{translations of sentences}, as follows:
each $\Delta$-sentence $\gamma$ is translated to a $\Delta'$-sentence $\chi(\gamma)$ by replacing (in an inductive manner) the symbols in $\Delta$ with symbols from $\Delta'$ according to $\chi$.
\begin{fact}
$\Sen^\HDFOLR$ is a functor $\Sig^\HDFOLR \to \Set$.
\end{fact}
\paragraph{Local satisfaction relation}
Given a $\Delta$-model $(W, M)$ and a world $w \in |W|$, we define the \emph{satisfaction of $\Delta$-sentences at $w$} by structural induction as follows:
\begin{enumerate}

\item \emph{For atomic sentences}:
  \begin{itemize}

  \item $(W, M) \models^{w} k $ iff $W_k = w$ for all nominals $k$;

  \item $(W, M) \models^{w} t_{1} = t_{2}$ iff $M_{w, t_1} = M_{w,t_2} $ for all hybrid equations $t_{1} = t_{2}$;

  \item $(W, M) \models^{w} \pi(t)$ iff $M_{w,t} \in M_{w, \pi}$ for all hybrid relations $\pi(t)$.

  \end{itemize}

\item \emph{For full sentences}:
\begin{itemize}
  \item $(W, M) \models^{w} \at{k} \gamma$ iff $(W, M) \models^{w'} \gamma$, where $w' = W_{k}$;

  \item $(W, M) \models^{w} \neg \gamma$ iff $(W, M) \not\models^{w} \gamma$;

  \item $(W, M) \models^{w} \vee \Gamma$ iff $(W, M) \models^{w} \gamma$ for some $\gamma \in \Gamma$; 

  \item $(W, M) \models^{w} \store{z}{\gamma}$ iff $(W^{z \leftarrow w}, M) \models^{w} \gamma$,
  
where $(W^{z \leftarrow w}, M)$ is the unique $\Delta[z]$-expansion of $(W, M)$ that interprets the variable $z$ as $w$;
\footnote{An expansion of $(W, M)$ to $\Delta[X]$ is a Kripke structure $(W', M')$ over $\Delta[X]$ that interprets all symbols in $\Delta$ in the same way as $(W, M)$.}    

\item $(W, M) \models^w \Exists{X}{\gamma}$ iff $(W', M') \models^w \gamma$ for some expansion $(W', M')$ of $(W, M)$ to the signature $\Delta[X]$;
\footnotemark[\thefootnote]

\item $(W, M) \models^{w} \pos{\act} \gamma$ iff $(W, M) \models^{w'} \gamma$ for some $w' \in |W|$ such that $(w, w') \in W_{\act}$.

\end{itemize}
  
\end{enumerate}
The following \emph{satisfaction condition} can be proved by induction on the structure of $\Delta$-sentences.
The proof is essentially identical to those developed for several other variants of hybrid logic presented in the literature (see, e.g.~\cite{dia-qvh}).
\begin{proposition}[Local satisfaction condition for signature morphisms] \label{prop:sat-cond}
For every signature morphism $\chi \colon \Delta \to \Delta'$, 
$\Delta'$-model $(W', M')$, 
possible world $w' \in |W'|$, and 
$\Delta$-sentence $\gamma$, 
we have
$(W', M') \models^{w} \chi(\gamma)$ iff $(W', M') \red_{\chi} \models^{w} \gamma$.
\footnote{By the definition of reducts, $(W', M')$ and $(W', M') \red_{\chi}$ have the same possible worlds.}
\end{proposition}
\paragraph{Non-void signatures}
A signature $\Delta=(\Sigma^\nom,\Sigma^\rigid \subseteq \Sigma)$ is called \emph{non-void} if both $\Sigma^\nom$  and  $\Sigma$ are non-void first-order signatures.
Notice that for any non-void signature, 
the set of nominals is not empty, that is, $ F^\nom\neq \emptyset$, and  
the set of hybrid terms of any sort is not empty, that is, $T_{\overline\Sigma,s}\neq \emptyset$ for all sorts $s\in S$. 
 
 \begin{lemma} \label{HFOLR-init}
  If $\Delta=(\Sigma^\nom,\Sigma^\rigid\subseteq\Sigma)$ is non-void then there exists an initial model of terms $(W^\Delta,M^\Delta)$ defined as follows:
\begin{enumerate*}[label=(\arabic*)]
\item $W^\Delta= F^\nom$, and 
   
\item $M^\Delta\colon  F^\nom \to |\Mod^\FOL(\Sigma)|$, where for all $k\in F^\nom$, $M^\Delta_k$ is a first-order structure such that
\end{enumerate*}
\begin{enumerate}[label=(\alph*)]
\item $M^\Delta_{k,s}= T_{@\Sigma,@_k s}$ for all sorts $s\in S$,
\item $M^\Delta_{k,\sigma}\colon T_{@\Sigma,\at{k} \ari} \to T_{@\Sigma,@_k s}$ is defined by $M^\Delta_{k,\sigma}(t)=(\at{k} \sigma)(t)$ for all function symbols $(\sigma\colon\ari\to s)\in F$ and all tuples of hybrid terms $t\in T_{@\Sigma,\at{k} \ari}$, and
\item $M^\Delta_{k,\pi}$ is the empty set for all relation symbols $(\pi\colon\ari)\in P$.
\end{enumerate}
\end{lemma}
The proof of Lemma~\ref{HFOLR-init} is based on the unique interpretation of
terms into models, and it is straightforward. 
We leave it as an exercise for the reader.
\paragraph{Notations}
Take a signature $\Delta$,
a Kripke structure $(W,M)\in|\Mod^\HDFOLR(\Delta)|$,
a sentence $\varphi\in\Sen^\HDFOLR(\Delta)$, and
two theories $T,\Gamma\subseteq\Sen^\HDFOLR(\Delta)$.
\begin{itemize}
\item We say that $(W,M)$ (globally) satisfies $\varphi$, in symbols, $(W,M)\models \varphi$, if $(W,M)\models^w \varphi$ for all $w\in|W|$.
\item We say that $(W,M)$ satisfies $\Gamma$, in symbols, $(W,M)\models \Gamma$, if $(W,M)\models \gamma$ for all $\gamma\in\Gamma$.
\item We say that $T$ (globally) satisfies $\Gamma$, in symbols, $T\models \Gamma$, 

if $(V,N)\models T$ implies $(V,N)\models \Gamma$ for all $(V,N)\in|\Mod^\HDFOLR(\Delta)|$.
\footnote{Notice that the semantics of  $\varphi\models \gamma$ is different from the standard one, where $\varphi\models \gamma$ is interpreted locally, that is, $(V,N)\models^w\varphi$ implies $(V,N)\models^w\gamma$ for all Kripke structures $(V,N)$ and all possible worlds $w$ in $V$.}
\item We say that $T$ is semantically equivalent to $\Gamma$, in symbols, $T\modelsm \Gamma$, 
if $T\models \Gamma$ and $\Gamma\models T$.
\end{itemize}
\begin{lemma} \label{lemma:h-prop}
Let $\Delta$ be a signature.
\begin{enumerate}
  
\item For all sentences $\varphi$ over $\Delta$,
all nominal variables $z$ for $\Delta$,
all  $(W,M)\in|\Mod^\HDFOLR(\Delta)|$ and all $w\in|W|$,

$(W,M)\models^w \Forall{z}\at{z}\varphi$ iff 
  $(W,M)\models \Forall{z}\at{z}\varphi$ iff
$(W,M)\models\varphi$.

\item For all sentences $\varphi$ and $\gamma$ over $\Delta$, all nominal
  variables $z$ for  $\Delta$, and all nominals $k$ in $\Delta$, we have

$\varphi\modelsm \Forall{z}\at{z}\varphi \modelsm @_k\Forall{z}\at{z}\varphi$,
while $\varphi\Rightarrow \gamma 
\modelsm (\Forall{z}\at{z}\varphi) \Rightarrow \gamma$ does not hold, in
general.
\item For all theories $T$ over $\Delta$, all sentences $\varphi$ and $\gamma$ over $\Delta$ and all nominals $k$ in $\Delta$,

$T\models\at{k} (\varphi \Rightarrow \gamma) $ iff $T\cup\{\at{k}\varphi \}\models\at{k}\gamma$.
\item For all theories $T$ over $\Delta$, all sentences $\varphi$ over $\Delta$ and all nominals $k$ in $\Delta$,

$T\models\at{k}\neg\varphi$ iff $T\cup\{\at{k} \varphi\}\models\bot$.
\item For all theories $T$ over $\Delta$, 
all nominal variables $x$ and $z$ for $\Delta$ and
all sentences $\psi$ over $\Delta[x]$,

$T\cup \{\psi\}$ is satisfiable over $\Delta[x]$ iff $T\cup\{\Exists{x}\Forall{z}\at{z}\psi\}$ is satisfiable over $\Delta$.
\footnote{If we take into consideration the third component of a variable, the correct statement is 
$T\cup \{\psi\}$ is satisfiable over $\Delta[x]$ iff $T\cup\{\Exists{x}\Forall{z}\at{z}\iota(\psi)\}$ is satisfiable over $\Delta$, where $\iota\colon\Delta[x]\hookrightarrow \Delta[x,z]$.}
\end{enumerate}
\end{lemma}
The proof of this lemma is straightforward and we leave it as an exercise for
the interested reader. Informally, the key is that in the sentence     
$\Forall{z}\at{z}\varphi$ the quantifier $\forall z$ binds the free variable $z$
in $\at{z}$, so $\Forall{z}\at{z}\varphi$ means 
`$\varphi$ holds at all worlds $w$'.

By using `storing and retrieving' intuition it is easy to define complex properties. 
For example, temporal until operator $U$ -- with the following semantics: $U(\varphi,\psi)$ is true at a state $w$ if there is a future state $w'$ where $\varphi$ holds, such that $\psi$ holds in all states between $w$ and $w'$ -- can be defined as follows:
$$U(\varphi, \psi) \coloneqq \store{x}\Diamond\store{y}(\varphi \wedge \at{x}\Box(\Diamond y \Rightarrow \psi))$$
The idea is to name the current state $x$ using $\downarrow$, and then by
$\Diamond$, we identify a successor state, which we call $y$, where $\varphi$ holds. 
Using $@$, the point of evaluation is changed to $x$, and then at all successors
of $x$ connected to $y$, $\psi$ holds. 
\section{Logical concepts}\label{3}
In this section, we recall some concepts necessary to prove our results.
\subsection{Substitutions}  \label{sec:subst}
Let $\Delta$ be a signature, 
$C_1$ and $C_2$ two sets of new constants for $\Delta$ of sorts in $S^\ext$, the extended set of rigid sorts.
A substitution $\theta : C_1 \to C_2$ over $\Delta$ is a mapping from $C_1$ to $|(W^{\Delta[C_2]},M^{\Delta[C_2]})|$, the carrier sets of the initial Kripke structure $ (W^{\Delta[C_2]},M^{\Delta[C_2]})$ over $\Delta[C_2]$ defined in Lemma~\ref{HFOLR-init}.

\begin{proposition}[Local satisfaction condition for substitutions~\cite{gai-her}]
A substitution $\theta : C_1\to C_2$ over $\Delta$ uniquely determines:
\begin{enumerate}
\item a sentence function $\theta\colon\Sen^\HDFOLR(\Delta[C_1])\to \Sen^\HDFOLR(\Delta[C_2])$, 
which preserves $\Delta$ and maps each constant $c\in C_1$ to a rigid term $\theta(c)$ over $\Delta[C_2]$, and

\item a reduct functor $\red_\theta\colon\Mod^\HDFOLR(\Delta[C_2])\to\Mod^\HDFOLR(\Delta[C_1])$,
which preserves the interpretation of $\Delta$ and interprets each $c\in C_1$ as $\theta(c)$,
\end{enumerate}
such that the following local satisfaction condition holds:
$$(W,M)\models^w \theta(\gamma)\text{ iff }(W,M)\red_\theta\models^w \gamma$$
for all $\Delta[C_1]$-sentences $\gamma$, all Kripke structures $(W,M)$ over $\Delta[C_2]$ and all possible worlds $w\in|W|$.
\end{proposition}
\subsection{Fragments} \label{sec:frag}
By restricting the signatures and/or the sentences of $\HDFOLR$, one can obtain well-known hybrid logics studied in the literature.
\begin{definition}[Fragment]\label{def:fragment}
A fragment $\mathcal{L}$ of $\HDFOLR$ is obtained by restricting the syntax of $\HDFOLR$, that is, 
$\Sig^\L$ is a subcategory of $\Sig^\HDFOLR$ and $\Sen^\L\colon\Sig^\L\to\Set$ is a subfunctor of $\Sen^\HDFOLR\colon\Sig^\HDFOLR\to\Set$, such that
\begin{enumerate}
\item for any signature $\Delta\in|\Sig^\L|$, any set $C$ of new nominals and any set $D$ of new rigid constants, we have $\Delta\hookrightarrow\Delta[D,C]\in\Sig^\L$,

\item for any substitution $\theta\colon\pos{C_1,D_1}\to\pos{C_2,D_2}$ over a signature $\Delta\in|\Sig^\L|$ and 
any sentence $\gamma\in\Sen^\L(\Delta[C_1,D_1])$, 
we have $\theta(\gamma)\in\Sen^\L(\Delta[C_2,D_2])$, and

\item $\L$ is closed under subsentence relation, that is,
\begin{itemize}
\item if $\pos{\act_1\comp\act_2}\gamma\in\Sen^\L(\Delta)$ then $\pos{\act_1}\gamma\in \Sen^\L(\Delta)$ and $\pos{\act_2}\gamma\in \Sen^\L(\Delta)$,
\item if $\pos{\act_1\cup\act_2}\gamma\in\Sen^\L(\Delta)$ then $\pos{\act_1}\gamma\in \Sen^\L(\Delta)$ and $\pos{\act_2}\gamma\in \Sen^\L(\Delta)$,
\item if $\pos{\act^*}\gamma\in\Sen^\L(\Delta)$ then $\pos{\act^n}\gamma\in \Sen^\L(\Delta)$ for some $n\in\omega$,
\item if $\pos{\act}\gamma\in \Sen^\L(\Delta)$ then $\gamma\in \Sen^\L(\Delta)$,
\item if $\neg\gamma\in \Sen^\L(\Delta)$ then $\gamma\in \Sen^\L(\Delta)$,
\item if $\vee\Gamma\in \Sen^\L(\Delta)$ then $\gamma\in \Sen^\L(\Delta)$ for all $\gamma\in\Gamma$,
\item if $\at{k}\gamma\in \Sen^\L(\Delta)$ then $\gamma\in \Sen^\L(\Delta)$,
\item if $\store{z}\gamma\in \Sen^\L(\Delta)$ then $\gamma\in \Sen^\L(\Delta[z])$, and
\item if $\Exists{X}\gamma\in \Sen^\L(\Delta)$ then $\gamma\in \Sen^\L(\Delta[X])$.
\end{itemize} 
\end{enumerate}
\end{definition}
According to Definition~\ref{def:fragment}, a fragment $\L$ of $\HDFOLR$ has the same models as $\HDFOLR$.
By the closure under the subsentence relation, the sentences of $\L$ are
constructed from some atomic sentences by applying Boolean connectives,
possibility over action relations, retrieve, store or existential quantifiers,
if these sentence building operators are available in $\L$. 
It does not imply 
that $\L$ is closed under any of
these operators.

\begin{example} [Hybrid First-Order Logic with Rigid symbols ($\HFOLR$) \cite{gai-godel}] \label{ex:HFOLR}
This is the hybrid variant of $\HDFOLR$ obtained by discarding structured actions and allowing possibility over binary modalities.
According to \cite{gai-godel}, $\HFOLR$ is compact.
\end{example}
\begin{example}[Hybrid-Dynamic Propositional Logic ($\HDPL$)] \label{ex:HDPL}
This is the dynamic variant of the most common form of multi-modal hybrid logic (e.g. \cite{ArecesB01}).
$\HDPL$ is obtained from $\HDFOLR$ by restricting the signatures $\Delta=(\Sigma^\nom,\Sigma^\rigid\subseteq\Sigma)$ such that the set of sorts in $\Sigma$ is empty, and the set of sentences is  given by the following grammar:
$$\gamma \Coloneqq  \rho \alt k \alt \at{k}\gamma \alt \neg \gamma \alt \vee \Gamma \alt \pos{\act}\gamma $$
where
\begin{enumerate*}[label=(\alph*)]
\item $\rho$ is a propositional symbol,
\item $k \in  F^\nom $ is a nominal,
\item $\act \in A^{\nom}$ is an action, and
\item $\Gamma$ is a finite set of sentences over $\Delta$.
\end{enumerate*}
Notice that if $\Sigma=(S,F,P)$ and $S=\emptyset$ then $P$ contains only propositional symbols.
$\HPL$ is the fragment of $\HDPL$ obtained by discarding structured actions.
\end{example}
\begin{example}[Rigid First-Order Hybrid Logic ($\RFOHL$) \cite{DBLP:conf/wollic/BlackburnMMH19}] \label{ex:RFOHL}
This logic is obtained from $\HFOLR$ by restricting the signatures $\Delta=(\Sigma^\nom,\Sigma^\rigid\subseteq\Sigma)$ such that 
\begin{enumerate*}[label=(\alph*)]
\item $\Sigma^\nom$ has only one binary modality, 
\item $\Sigma$ is single-sorted, 
\item there are no rigid function symbols except variables (regarded here as special constants), and 
\item there are no rigid relation symbols.
\end{enumerate*}
\end{example}
All examples of logics given above are fragments of $\HDFOLR$.
In the following, we give an example of logic which is obtained from $\HDFOLR$ by some syntactic restrictions and it is not a fragment according to Definition~\ref{def:fragment}.
\begin{example}[Hybrid First-Order Logic with user-defined Sharing ($\HFOLS$)] \label{ex:HFOLS}
This logic has the same signatures and Kripke structure as $\HFOLR$.
The sentences are obtained from atoms constructed with open terms only, that is,
if $\Delta=(\Sigma^\nom,\Sigma^\rigid\subseteq\Sigma)$, 
all (ground) equations over $\Delta$ are of the form $t_1=t_2$, where $t_1,t_2\in T_\Sigma$, and
all (ground) relation over $\Delta$ are of the form $\pi(t)$, where $(\pi:\ari)\in P$ and $t\in T_{\Sigma,\ari}$.
Variants of $\HFOLS$ have been used in works such as \cite{martins,dia-msc,dia-qvh}.
\end{example}
$\HFOLS$ is not a fragment of $\HDFOLR$ in the sense of Definition~\ref{def:fragment}, as it is not closed under substitutions.
Retrieve is applied only to sentences and not to function or relation symbols.
However, according to \cite{gai-godel}, that is no loss of expressivity as $\HFOLS$ has the same expressive power as $\HFOLR$.
\begin{lemma} \label{lemma:HFOLR-HFOLS}
For each signature $\Delta$ and each sentence $\gamma\in\Sen^\HFOLR(\Delta)$ there exists a sentence $\gamma'\in \Sen^\HFOLS(\Delta)$ such that $(W,M)\models^w\gamma$ iff $(W,M)\models^w\gamma'$ for all Kripke structures $(W,M)$ over $\Delta$ and all possible worlds in $W$.
\end{lemma}
\begin{proof}
By using \cite[Lemma 2.20]{gai-godel} which shows that for any atomic sentence in $\HFOLR$ there exists a sentence in $\HFOLS$ which is satisfied by the same class of Kripke structures. 
\end{proof}
The forcing technique and the Omitting Types Theorem are not applicable to $\HFOLS$ even if it has the same expressivity power as $\HFOLR$.
This is due to the absence of a proper support for the substitutions described in Section~\ref{sec:subst}.
By Lemma~\ref{lemma:HFOLR-HFOLS}, the results can be borrowed from $\HFOLR$ to $\HFOLS$.
It is worth noting that $\HFOLS$ can be extended with features of dynamic logics such that the dynamic variant of $\HFOLS$ matches the expressivity of $\HDFOLR$ by the same arguments used in the proof of Lemma~\ref{lemma:HFOLR-HFOLS} .
\subsection{Reachable models}
In this section, we give a category-based description of the models which consist of elements that are denotations of terms.
The concept of reachable model appeared in institutional model-theory in \cite{Petria07}, and it has been used successfully in several abstract developments such as proof-theoretic results~\cite{gai-com,gai-cbl,gai-bir} as well as model-theoretic results~\cite{gai-int,gai-ott,gai-init,gai-dls,gai-her,tutu-iilp}.

\begin{definition}
A Kripke structure $(W,M)$ over a signature $\Delta=(\Sigma^\nom,\Sigma^\rigid\subseteq\Sigma)$ is \emph{reachable} if for 
each set of new nominals $C$,
each set of new rigid constants $D$, and 
any expansion $(W',M')$ of $(W,M)$ to $\Delta[C]$, 
there exists a substitution  $\theta\colon \pos{C,D}\to\pos{\emptyset,\emptyset}$ over $\Delta$ such that 
$(W,M)\red_\theta=(W',M')$.
\end{definition}

\begin{proposition} [Reachable Kripke structures \cite{gai-her}] \label{reach-HFOLR}
A Kripke structure is reachable iff 
\begin{enumerate}
\item its set of states consists of denotations of nominals, and 
\item its carrier sets for the rigid sorts consist of denotations of rigid terms.
\end{enumerate}
\end{proposition}

By Proposition~\ref{reach-HFOLR}, a model $(W,M)$ is reachable iff the unique homomorphism from the initial Kripke structure $h\colon (W^\Delta,M^\Delta)\to (W,M)$ is surjective, that is, $h\colon W^\Delta\to W$ is surjective and $h_w\colon M^\Delta_w\to M_{h(w)}$ is surjective for all possible worlds $w\in|W^\Delta|$. 

\subsection{Basic sentences}
In this section, we recall an important property of certain simple sentences of hybrid logics, which play the role analogous to atomic sentences of first-order logic.

\begin{definition}[Basic set of sentences~\cite{dia-ult}]
\label{def:basic}
A set of sentences $B$ over a signature $\Delta=(\Sigma^\nom,\Sigma^\rigid\subseteq\Sigma)$ is \emph{basic} if there exists a Kripke structure $(W^B,M^B)$ such that 
$$(W,M)\models B \mbox{ iff there exists a homomorphism } h:(W^B,M^B)\to (W,M)$$
for all Kripke structures $(W,M)$. 
We say that $(W^B,M^B)$ is a \emph{basic model} of~$B$. 
If in addition the homomorphism $h$ is unique then the set $B$ is called \emph{epi-basic}.
\end{definition}
According to \cite{dia-ult,dia-book}, in first-order logic, any set of atomic
sentences is basic. 
One important property of basic sentences is the preservation of their satisfaction along  homomorphisms: 
given a set of basic sentences $B$ and a homomorphism $h\colon M\to N$, if $M\models B$ then $N\models B$.
In hybrid logics, this property does not hold, in general.
The following example is from \cite{gai-godel}.

\begin{example} \label{ex:locally-basic}
 Consider the following $\HPL$ signature $\Delta=(\Sigma^\nom,\Prop)$ such that 
 $F^\nom=\{k\}$, $P^\nom=\{\lambda\colon\any~\any\}$ and $\Prop=\{\rho\}$.
Let $h\colon (W,M)\hookrightarrow (W',M')$ be the inclusion homomorphism defined by:
 \begin{enumerate}
 \item $|W|=\{k\}$, $W_\lambda=\{(k,k)\}$ , $\rho$ is true in $M_k$, and
 
 \item $|W'|=\{k,w\}$, $W'_\lambda=\{(k,k)\}$, $\rho$ is true in $M'_k$, $\rho$ is not true in $M'_w$.
 \end{enumerate}
\end{example}

Example~\ref{ex:locally-basic} points out a significant difference between
ordinary logics and  hybrid (or, more generally, modal) logics.
 Note that $(W,M)\models^\HPL k$, $(W,M)\models^\HPL \langle\lambda\rangle k$ and $(W,M)\models^\HPL \rho$.
 Since $(W',M')\not\models^w k$, $(W',M')\not\models^w \langle\lambda\rangle k$ and $(W',M')\not\models^w \rho$ we have  
 $(W',M')\not\models^\HPL k$, $(W',M')\not\models^\HPL \langle\lambda\rangle k$ and $(W',M')\not\models^\HPL \rho$.
Thus, homomorphisms do not preserve satisfaction of atomic sentences.
Hence, atomic sentences are not basic in $\HPL$ (the same example
works for any modal logic). Note however that local satisfaction (satisfiaction
at a world) is preserved, and in hybrid logic the retrieve operator ($@$) lifts
local satisfaction to global. This motivates the next definition.  

\begin{definition}[Locally basic set of sentences~\cite{gai-godel}]
A set of sentences $\Gamma$ over a signature $\Delta$ is \emph{locally (epi-)basic} if $@\Gamma\coloneqq \{\at{k}\gamma \mid k\in F^\nom \text{ and } \gamma\in \Gamma\}$ is (epi-)basic.
\end{definition}
Notice that $@\Gamma$ is semantically equivalent to $@@\Gamma$.
We denote by $\Sen^\HDFOLR_0(\Delta)$ the set of all \emph{extended atomic} sentences.
\begin{enumerate}
\item nominals $k\in F^\nom$,
\item nominal relations $\pos{\lambda}k$, where $\lambda \in  P^\nom$ is a binary modality and $k\in F^\nom$,
\item hybrid equations $t_1=t_2$, where $t_1,t_2\in T_{\overline{\Sigma}}$, and
\item hybrid relations $\pi(t)$, where $\pi\in\overline{P}_\ari$, $t\in (T_{\overline{\Sigma}})_\ari$ and $\ari\in (\overline{S})^*$. 
\end{enumerate}

We denote by $\Sen^\HDFOLR_b(\Delta)$ the set of all sentences obtained from an
extended atomic sentence by applying retrieve ($@$) at most once.

\begin{proposition} [Locally basic set of sentences~\cite{gai-her,gai-godel}] \label{prop:loc-basic}
Given a signature $\Delta$, every set of sentences $B\subseteq \Sen^\HDFOLR_b(\Delta)$ is locally basic. 
Moreover, if $\Delta$ is non-void, then $B$ is locally epi-basic and its basic model $(W^B,M^B)$ is reachable.
\end{proposition}

\begin{definition}[Rigidification] \label{def:rigidification}
For any signature $\Delta=(\Sigma^\nom,\Sigma^\rigid\subseteq\Sigma)$, the \emph{rigidification function} $\At{k}\_\colon T_{\overline\Sigma}\to T_{@\Sigma}$, where $k\in F^\nom$, is recursively defined by: \smallskip

\begin{minipage}{\textwidth}
\begin{itemize}
\item $\At{k} \sigma(t)\coloneqq 
\left\{\begin{array}{l l}
 (@_k \sigma)(\At{k} t) & \mbox{ if }(\sigma\colon\ari\to s)\in F^\flexible,\\
\sigma(\At{k} t)       & \mbox{ if }(\sigma\colon\ari\to s)\in F^\rigid\cup @F^\flexible.
\end{array} \right.$ 
\end{itemize}
\end{minipage}    
\smallskip

Its extension $\At{k}\_\colon\Sen^\HFOLR(\Delta)\to\Sen^\HFOLR(\Delta)$ is recursively defined by: 
\smallskip

\begin{tabular}{l l}
\begin{minipage}{0.5\textwidth}
\begin{itemize}
\item $\At{k} k'\coloneqq \at{k}k'$
\item $\At{k} \pos{\lambda}(k') \coloneqq \at{k} \pos{\lambda}(k')$
\item $\At{k}(t_1 = t_2) \coloneqq (\At{k}t_1 =\At{k}t_2)$
\item $\At{k} \pi(t)\coloneqq
\left\{\begin{array}{l l}
(@_k \pi)(\At{k} t) & \mbox{ if } \pi\in P^\flexible\\
\pi(\At{k} t)         & \mbox{ if } \pi\in P^\rigid\cup@P^\flexible
\end{array} \right.$
\end{itemize}
\end{minipage}
&
\begin{minipage}{0.45\textwidth}   
\begin{itemize}
\item $\At{k} \neg \gamma\coloneqq\neg \At{k}\gamma $
\item $\At{k} \vee \Gamma \coloneqq\vee \At{k}\Gamma $
\item $\At{k} \at{k'}\gamma \coloneqq \At{k'} \gamma$
\item $\At{k} \Exists{X}\gamma \coloneqq \Exists{X}\At{k}\gamma$
\end{itemize}
\end{minipage}
\end{tabular}
\smallskip

Any sentence semantically equivalent to a sentence in the image of $\At{k}$ is called a \emph{rigid sentence}.
\end{definition}

The proof of the following lemma is straightforward and we leave it as an exercise for the readers.
\begin{lemma} \label{lemma:rigidification}
Any sentence $\at{k}\gamma$ is semantically equivalent to $\At{k}\gamma$. 
Hence, $\at{k}\gamma$ is rigid.
\end{lemma}
\section{Forcing}\label{4}
Forcing is a method of constructing models satisfying some properties forced by some conditions.
In this section, we generalize the forcing relation for hybrid logics defined in \cite{gai-godel} to hybrid-dynamic first-order logic with rigid symbols.
It is worth mentioning that the present developments can be cast in the framework of stratified institutions following the ideas presented in \cite{gai-godel}.
\begin{framework}
The results in this paper will be developed in a fragment $\mathcal{L}$ of $\HDFOLR$ that is semantically closed under negation and retrieve.
\footnote{$\L$ is semantically closed under negation whenever for all $\L$-sentences $\gamma$ there exists another $\L$-sentence $\varphi$ such that we have: $(W,M)\models^w \varphi $ iff $(W,M)\not \models^w\gamma$ for all Kripke structures $(W,M)$ and all possible worlds $w\in|W|$.
When there is no danger of confusion, we denote $\varphi$ by $\neg\gamma$.
Similarly, one can define the semantic closer of $\L$ under any sentence building operator.} 
We make the following notational conventions:
\begin{itemize}
\item We let $\Sen_0^\L$ to denote the subfunctor of $\Sen^\L$ which maps each signature $\Delta$ to the set of extended atomic sentences of $\L$ over the signature $\Delta$. 
This means that $\Sen_0^\L(\Delta)=\Sen^\HDFOLR(\Delta)\cap \Sen^\HDFOLR_0(\Delta)$ for all signatures $\Delta$.

\item We let $\Sen_b^\L$ to denote the subfunctor of $\Sen^\L$ which maps each signature $\Delta$ to the set of basic sentences of $\L$ over the signature $\Delta$.
This means that $\Sen_b^\L(\Delta)=\Sen^\HDFOLR(\Delta)\cap \Sen^\HDFOLR_b(\Delta)$ for all signatures $\Delta$.
\end{itemize} 
Since $\L$ is the logic in which we develop our results, we drop the superscript $\L$ from the notations $\Sen^\L$, $\Sen_0^\L$ and $\Sen_b^\L$ if there is no danger of confusion.
\end{framework}
Examples of fragments can be found in Section~\ref{sec:frag}.
\begin{definition}[Forcing property] \label{def:forcing}
Given a signature $\Delta$, a forcing property over $\Delta$ is a triple $\mathbb{P}=\pos{P,\leq,f}$ such that:
 \begin{enumerate}
  \item  \label{def:forcing-1} $(P,\leq)$ is a partially ordered set with a least element $0$. 
  
  The elements of $p$ are traditionally called \emph{conditions}.
  
 \item  \label{def:forcing-2} $f\colon P\to \P(\Sen_b(\Delta))$ is a function,

 \item  \label{def:forcing-3} if $p\leq q$ then $f(p)\subseteq f(q)$, and

 \item  \label{def:forcing-4} if $f(p)\models \at{k} \gamma$ then $\at{k} \gamma\in f(q)$ for some $q\geq p$,
\end{enumerate}
where $p\in P$, $q\in P$, $k\in F^\nom$ and $\gamma\in\Sen_0(\Delta)$.
\end{definition}
As for ordinary first-order logics, a forcing property generates a forcing relation on the set of all sentences.

\begin{definition} [Forcing relation]  \label{def:forcing-relation}
Let $\mathbb{P}=\pos{P,\leq,f }$ be a forcing property over $\Delta$. 

The family of relations $\Vdash=\{\Vdash^k\}_{k\in F^\nom}$, 
where $\Vdash^k\subseteq P\times \Sen(\Delta)$, is inductively defined as follows:
\begin{enumerate}
\item \label{fr1} \emph{For $\gamma$ extended atomic:} $p\Vdash^k \gamma$ if $\at{k}\gamma\in f(p)$.
 
\item \label{fr2} \emph{For $\pos{\act_1\comp\act_2}k''$:}
 $p\Vdash^k \pos{\act_1\comp \act_2}k''$ if $p\Vdash^k \pos{\act_1}k'$ and $p\Vdash^{k'} \pos{\act_2}k''$ for some $k'\in F^\nom$.
 
\item \label{fr3} \emph{For $\pos{\act_1\cup\act_2}k''$:} 
  $p\Vdash^k \pos{\act_1\cup\act_2}k''$ if $p\Vdash^k \pos{\act_1}k''$ or $p\Vdash^k \pos{\act_2}k''$.
 
\item \label{fr4} \emph{For $\pos{\act^*} k''$:} 
  $p\Vdash^k\pos{\act^*}k''$ if $p\Vdash^k \pos{\act^n}k''$ for some $n\in \N$.

\item \label{fr8} \emph{For $\pos{\act}\gamma$ with $\gamma\not\in F^\nom$:} $p\Vdash^k\pos{\act}\gamma$ if $p\Vdash^k \pos{\act}k'$ and $p\Vdash^{k'} \gamma$ for some nominal $k'\in F^\nom$.

\item \label{fr5} \emph{For $\neg\gamma$:} $p\Vdash^k\neg \gamma$ if there is no $q\geq p$ such that $q\Vdash^k \gamma$.

 \item \label{fr6} \emph{For $\vee \Gamma$:}  $p\Vdash^k \vee \Gamma$ if $p\Vdash^k \gamma$ for some $\gamma\in\Gamma$.

 \item \label{fr7} \emph{For $\at{k'} \gamma$:} $p\Vdash^k \at{k'}\gamma$ if $p\Vdash^{k'} \gamma$.

 \item \label{fr9} \emph{For $\store{z} \gamma$:} $p\Vdash^k \store{z}\gamma$ if $p\Vdash^k \gamma(z\leftarrow k)$.

 \item \label{fr10} \emph{For $\Exists{X}\gamma$:} $p\Vdash^k \Exists{X}\gamma$ if $p\Vdash^k \theta(\gamma)$ for some substitution  $\theta\colon X\to \emptyset$ over $\Delta$.
\end{enumerate}
\end{definition}

The forcing relation defined in the present contribution consists of the forcing relation defined in \cite{gai-godel} plus the items \ref{fr2}---\ref{fr4} of Definition~\ref{def:forcing-relation}.
The notation $p\Vdash^k \gamma$ is read \textit{$p$ forces $\gamma$ at $k$}.
\begin{remark}
Notice that Definition~\ref{def:forcing-relation} does not rely on the fact that $\L$ is closed under disjunction or quantifiers.
For example, the last item from Definition~\ref{def:forcing-relation} should be interpreted as follows: if $\Exists{X}\gamma$ is a sentence in $\L$ and $p\Vdash^k\theta(\gamma)$ for some substitution $\theta:X\to\emptyset$ over $\Delta$ then $p\Vdash^k\Exists{X}\gamma$.
\end{remark}
In regard to the satisfaction relation, one may consider a global forcing relation: $p\Vdash \gamma$ iff $p\Vdash^k\gamma$ for all nominals $k$. 
This remark establishes a connection between the results in the present contribution and the results in \cite{gai-com} and \cite{gai-ott}, where there exists only a global forcing relation.
  
\begin{lemma} \label{lemma:forcing-property}
  Let $\mathbb{P}=\langle P,\leq,f \rangle$ be a forcing property as in Definition~\ref{def:forcing}. We have:
 \begin{enumerate}
  \item \label{p1} $p\Vdash^k \neg\neg \gamma$ iff for each $q\geq p$ there is $r\geq q$ such that $r\Vdash^k \gamma$.

  \item \label{p2} If $q\geq p$ and $p\Vdash^k \gamma$ then $q\Vdash^k \gamma$.

  \item \label{p3} If $p\Vdash^k \gamma$ then $p\Vdash^k \neg\neg \gamma$.

  \item \label{p4} We cannot have both $p\Vdash^k \gamma$ and $p\Vdash^k \neg \gamma$.
 \end{enumerate}
\end{lemma}
\begin{proof}
Notice that the statements~\ref{p1} and \ref{p3} are well-defined as $\L$ is semantically closed under negation.
 \begin{enumerate}
 \item $p\Vdash^k \neg\neg \gamma$ iff for each $q\geq p$ we have $q\not\Vdash^k \neg \gamma$ iff 
 
 for each $q\geq p$ there is $r\geq q$ such that $r\Vdash^k \gamma$.
 
 \item By induction on the structure of sentences:
 
\begin{proofcases}[itemsep=1.5ex]
\item[For $\gamma$ extended atomic]
The conclusion follows easily from $f(p)\subseteq f(q)$.
\item [For $\pos{\act_1\comp \act_2}k''$]  
$p\Vdash^k \pos{\act_1\comp \act_2}k''$ iff 
$p\Vdash^k \pos{\act_1}k'$ and $p\Vdash^{k'} \pos{\act_2}k''$ for some $k'\in F^\nom$.
By the induction hypothesis,  $q\Vdash^k\pos{ \act_1}k'$ and $q\Vdash^{k'} \pos{\act_2}k''$.
Hence, $q\Vdash^k \pos{\act_1\comp \act_2}k''$.
\item [For $\pos{\act_1\cup \act_2}k''$] 
$p\Vdash^k\pos{\act_1\cup\act_2}k''$ iff $p\Vdash^k \pos{\act_1}k''$ or $p\Vdash^k \pos{\act_2}k''$.
By the induction hypothesis, $q\Vdash^k \pos{\act_1}k''$ or $q\Vdash^k \pos{\act_2}k''$.
Hence, $q\Vdash^k\pos{\act_1\cup\act_2}k''$.
\item [For $\pos{\act^*}k''$]
$p\Vdash^k \pos{\act^*}k''$ iff there exists $n\in \N$ such that $p\Vdash^k \act^n \pos{k''}$.
By the induction hypothesis, $q\Vdash^k \pos{\act^n}k''$.
Hence, $q\Vdash^k \pos{\act^*}k''$.
\item [For $\pos{\act}\gamma$ with $\gamma\not\in  F^\nom$]
$p\Vdash^k \pos{\act}\gamma$ iff
$p\Vdash^k \pos{\act}k'$ and $p\Vdash^{k'}\gamma$.
By the induction hypothesis, $q\Vdash^k \pos{\act}k'$ and $q\Vdash^{k'}\gamma$.
Hence, $q\Vdash^k \pos{\act}\gamma$.
\item[For $\at{k'}\gamma$]
We have $p\Vdash^k \at{k'}\gamma$ iff $p\Vdash^{k'} \gamma$. 
By induction hypothesis, $q\Vdash^{k'}\gamma$. 
Hence, $q\Vdash^k \at{k'}\gamma$.
\item[For $\neg \gamma$]
We have $p\Vdash^k \neg \gamma$. 
This means $r\not\Vdash^k \gamma$ for all $r\geq p$. 
In particular, $r\not\Vdash^k \gamma$ for all $r\geq q$.
Hence, $q\Vdash^k \neg \gamma$.
\item[For $\vee \Gamma$]
$p\Vdash^k \gamma$ for some $\gamma\in \Gamma$. 
By induction hypothesis, $q\Vdash^k \gamma$ which implies $q\Vdash^k \vee \Gamma$.
\item[For $\store{z}\gamma$]
We have $p\Vdash^k \store{z}\gamma$ iff $p\Vdash^k \gamma(z\leftarrow k)$. 
By the induction hypothesis, $q\Vdash^k \gamma(z\leftarrow k)$, which implies $q\Vdash^k \store{z}\gamma$. 
\item[For $\Exists{X} \gamma$]
Since $p\Vdash^k \Exists{X} \gamma$ then $p\Vdash^k \theta(\gamma)$ for some substitution  $\theta\colon X\to \emptyset$ over $\Delta$. 
By the induction hypothesis, $q\Vdash^k \theta(\gamma)$. 
Hence, $q\Vdash^k \Exists{X}\gamma$.
\end{proofcases}
 
 \item It follows from \ref{p1} and \ref{p2}.
 
 \item By the reflexivity of $(P,\leq)$.
 \end{enumerate}
\psqed\end{proof}

Lemma~\ref{lemma:forcing-property} is a generalization of \cite[Lemma 4.4]{gai-godel} from hybrid logics to hybrid dynamic logics.
In the present contribution, since the proof of the second statement is by induction, we need to consider possibility over structured actions.

\begin{definition} [Generic set \cite{gai-godel}] \label{def:generic-set}
 Let $\mathbb{P}=\langle P,\leq,f\rangle$ be a forcing property over a signature $\Delta$. 
 
 A subset $G\subseteq P$ is generic if it has the following properties:
 \begin{enumerate}
 
  \item $r\in G$ if $r\leq p$ and $p\in G$;
  
  \item there exists $r\in G$ such that $r\geq p$ and $r\geq q$, for all $p,q\in G$;
  
  \item there exists $r\in G$ such that $r\Vdash^k \gamma$ or $r\Vdash^k \neg \gamma$, for all $\Delta$-sentences $\at{k} \gamma$.
 
 \end{enumerate}
 We write $G\Vdash^k \gamma$ whenever $p\Vdash^k \gamma$ for some $p\in G$. 
\qed\end{definition}

Note that $G$ in Definition~\ref{def:generic-set} is well-defined, as $\L$ is semantically closed under negation.
The following lemma ensures the existence of generic sets. 
The result is based on the assumption that signatures consist of a countable number of symbols.

\begin{lemma}[Existence of generic sets \cite{gai-godel}]\label{lemma:gsl} 
\label{lemma:gs-exist}
Let $\mathbb{P}=\langle P,\leq,f\rangle$ be a forcing property over a signature~$\Delta$.
If $\Sen(\Delta)$ is countable then every $p$ belongs to a generic set.
\end{lemma}

For the semantic forcing property defined in the next section it is possible to construct generic sets even if the underlying signature consists of an uncountable number of symbols.
Notice that the definition of forcing relation and the definition of generic set are based on syntactic compounds. 
The following definition gives the semantics/meaning to these concepts.

\begin{definition} [Generic model \cite{gai-godel}]\label{def:generic-model}
Let $\mathbb{P}=\langle P,\leq,f\rangle$ be a forcing property over a signature $\Delta$.
\begin{itemize}
\item $(W,M)$ is a model for a generic set $G\subseteq P$ when $(W,M)\models \at{k} \gamma \mbox{ iff } G\Vdash^k \gamma$, for all $\Delta$-sentences $\at{k} \gamma$.

\item $(W,M)$ is a model for $p\in P$ if there is a generic set $G\subseteq P$ such that $p\in G$ and $(W,M)$ is a model for $G$.
\end{itemize}
\end{definition}

The models $(W,M)$ from Definition~\ref{def:generic-model} are called, traditionally, \emph{generic models}. 
The following result ensures the existence of generic models.

\begin{theorem}[Generic Model Theorem]\label{GMT}
Let $\mathbb{P}=\langle P,\leq,f\rangle$ be a forcing property over $\Delta$.
Then each generic set $G$ of $\mathbb{P}$ has a generic Kripke structure $(W,M)$. 
If in addition $\Delta$ is non-void, $(W,M)$ is reachable.
\end{theorem}
\begin{proof} 
 Let $G$ be a generic set.
 We define $T=\{\at{k}\gamma\in\Sen(\Delta)  \mid G\Vdash^k \gamma \}$ and $B=T\cap\Sen_b (\Delta)$.
 By Proposition~\ref{prop:loc-basic}, $B$ is basic, and there exists a basic model $(W^B,M^B)$ for $B$ that is reachable.
 We show that $(W^B,M^B)\models \at{k} \gamma$ iff $G\Vdash^k \gamma$, for all $\Delta$-sentences $\at{k} \gamma$.

\begin{proofcases}[itemsep=1.5ex]
\item[For $\gamma$ extended atomic]
Assume that $(W^B,M^B)\models \at{k} \gamma$. 

\begin{proofsteps}{28em}
  \label{ps:gmt-basic-1}$B$ and $\{\at{k} \gamma\}$ are basic & 
  by  Proposition~\ref{prop:loc-basic}\\
  
  there exists an arrow $(W^{\at{k}\gamma}, M^{\{\at{k} \gamma\}})\to (W^B,M^B)$ & 
  since $\{\at{k} \gamma \}$ is basic and $(W^B,M^B)\models \at{k} \gamma$\\
  
  $B\models \at{k} \gamma$ & 
  since both $B$ and $\{\at{k} \gamma \}$ are basic \\
  
   there exists $B_f\subseteq B$ finite such that $B_f\models \at{k} \gamma$ & 
  since $\HDFOLR_b$ is compact\\
  
  $B_f=\{\at{k_1}\gamma_1,\dots, @_{k_n}\gamma_n\}$ for some $\gamma_i\in \Sen_0 (\Delta)$ and some $k_i\in F^\nom$ &
  by the definition of $B$\\
  
  for all $i\in \{1,\ldots,n\}$, there exists $p_i\in G$ such that $p_i\Vdash^{k_i}\gamma_i$ & 
  by the definition of $B$ \\
  
  there exists $p\in G$ such that $p\geq p_i$ for all $i\in \{1,\ldots,n\}$ &
  since $G$ is generic\\
  
  $B_f\subseteq f(p)$ & 
  since $B_f\subseteq \Sen_b(\Delta)$\\
  
  \label{gmt-b-9} $q\Vdash^k \gamma$ or $q\Vdash^k\neg\gamma$ for some $q\in G$  & 
  since $G$ is generic\\
  
  \label{gmt-b-10} suppose towards a contradiction that $q\Vdash^k \neg\gamma$ & 
  \substeps{9}
  $r\geq p$ and $r\geq q$ for some $r\in G$  & 
  since $G$ is generic \\
  
  $r\Vdash ^k \neg\gamma$ & 
  by Lemma~\ref{lemma:forcing-property}~(\ref{p2}), 
  since $r\geq q$ and $q\Vdash^k \neg\gamma$ \\
  
  $B_f\subseteq f(r)$ & 
  since $B_f\subseteq f(p)$  and  $r\geq p$ \\
  
  there exists $s\geq r$ such that $\at{k} \gamma\in f(s)$ & 
  since $B_f\models \at{k} \gamma$, we have $f(r)\models \at{k} \gamma$ \\
  $s\Vdash^k \gamma$ & 
  by Definition~\ref{def:forcing-relation} \\
  $s\Vdash^k \neg\gamma$ & 
  by Lemma~\ref{lemma:forcing-property}~(\ref{p2}) \\
  contradiction & 
  by Lemma~\ref{lemma:forcing-property}~(\ref{p4}) 
  \endsubsteps
  \label{ps:gmt-b-10} $q\Vdash^k \gamma$ & 
  by \ref{gmt-b-9} and \ref{gmt-b-10}\\
  $G\Vdash^k \gamma$ & since $q\in G$
\end{proofsteps}
If $G\Vdash^k \gamma$ then by the definition of $B$, we have $\at{k} \gamma\in B$, which implies $B\models \at{k} \gamma$; 
hence, $(W^B,M^B)\models \at{k} \gamma$.
 \item [For $\pos{\act_1\comp \act_2}k''$]  
  Assume that $(W^B,M^B)\models \at{k} \pos{\act_1\comp \act_2}k''$.
  \begin{proofsteps}{24em}
   $(W^B_k,W^B_{k''})\in W^B_{(\act_1\comp\act_2)}$ & by definition \\
   
   \label{ps:gm-act2} $(W^B_k,w)\in W^B_{\act_1}$ and  $(w,W^B_{k''})\in W^B_{\act_2}$ for some $w\in|W^B|$ & 
   since $\act_1\comp\act_2$ is the composition of the relations $\act_1$ and $\act_2$\\
   
   \label{ps:gm-act3} $w=W^B_{k'}$ for some nominal $k'\in F^\nom$ & 
   since $(W^B,M^B)$ is reachable \\
   
   $(W^B_k,W^B_{k'})\in W^B_{\act_1}$ and  $(W^B_{k'},W^B_{k''})\in W^B_{\act_2}$ &
   by  \ref{ps:gm-act2} and \ref{ps:gm-act3}\\

   $G\Vdash^k \pos{\act_1}k'$ and $G\Vdash^{k'}\pos{\act_2}k''$ &
   by the induction hypothesis \\
   
   \label{ps:gm-act6} $p\Vdash^k \pos{\act_1}k'$ for some $p\in G$ and \newline 
   $q\Vdash^{k'}\pos{\act_2}k''$ for some $q\in G$ & 
   by Definition~\ref{def:generic-set} \\
   
   \label{ps:gm-act7} $r\geq p$ and $r\geq q$ for some $r\in G$ & 
   since $G$ is generic\\
   
   $r\Vdash^k \pos{\act_1}k'$ and $r\Vdash^{k'}\pos{\act_2}k''$ &
   by Lemma~\ref{lemma:forcing-property} (\ref{p2}) applied to \ref{ps:gm-act6} and \ref{ps:gm-act7}\\
    
   $r\Vdash^k \pos{\act_1\comp \act_2}k''$ & 
   by Definition~\ref{def:forcing-relation} \\
   
   $G\Vdash^k \pos{\act_1\comp \act_2}k''$ &
   by Definition~\ref{def:generic-set}
  \end{proofsteps}
  
  Assume that $G\Vdash^k\pos{\act_1\comp \act_2}k''$.
  
  \begin{proofsteps}{24em}
   $p\Vdash^k\pos{\act_1\comp \act_2}k''$ for some $p\in G$ & \\
  
   $p\Vdash^k \pos{\act_1}k'$ and $p\Vdash^{k'} \pos{\act_2}k''$  for some $k'\in F^\nom$ & 
   by Definition~\ref{def:forcing-relation}\\
   
   $(W^B,M^B)\models\at{k}\pos{\act_1}k'$ and $(W^B,M^B)\models\at{k'}\pos{\act_2}k''$ &
   by the induction hypothesis \\
   
   $(W^B,M^B)\models\at{k}\pos{\act_1\comp \act_2}k''$ & by the semantics of $\act_1\comp\act_2$
  \end{proofsteps}
 \item [For $\pos{\act_1\cup \act_2}k''$] The following are equivalent:
  \begin{proofsteps}{24em}
   
   $G\Vdash^k \pos{\act_1\cup\act_2}k''$ & \\
   
   $p\Vdash^k\pos{\act_1\cup\act_2}k''$ for some $p\in G$ &
   by Definition~\ref{def:forcing-relation} \\
   
   $p\Vdash^k\pos{\act_1}k''$ or $p\Vdash^k\pos{\act_2}k''$ & 
   by Definition~\ref{def:forcing-relation} \\
   
   $G\Vdash^k\pos{\act_1}k''$ or $G\Vdash^k\pos{\act_2}k''$ & 
   by Definition~\ref{def:forcing-relation} \\

   $(W^B,M^B)\models \at{k}\pos{\act_1}k''$ or $(W^B,M^B)\models \at{k}\pos{\act_2}k''$ &  
   by the induction hypothesis \\
   
   $(W^B,M^B)\models \at{k}\pos{\act_1\cup\act_2}k''$ & by the semantics of $\act_1\cup\act_2$
  \end{proofsteps} 
\item [For $\pos{\act^*}k''$]  
 The following are equivalent:
 \begin{proofsteps}{24em}
  $(W^B,M^B)\models\at{k}\pos{\act^*}k''$ & \\
  
  $(W^B,M^B)\models\at{k}\pos{\act^n}k''$ for some $n\in \N$ & by the semantics of $\act^*$\\
  
  $G\Vdash^k\pos{\act^n}k''$ for some $n\in \N$ & by the induction hypothesis \\
  
  $G\Vdash^k\pos{\act^*}k''$ & by Definition~\ref{def:forcing-relation}
 \end{proofsteps} 
 \item[For $\pos{\act}\gamma$ with $\gamma\not\in F^\nom$]
 The following are equivalent:
 \begin{proofsteps}{25em} 
  $(W^B,M^B)\models \at{k}\pos{\act}\gamma $ & \\
    
  $(W^B,M^B) \models^{w_1} \gamma$ for some $w_1\in|W^B|$ such that $(W^B_k,w_1)\in W^B_\act$ & 
  by the definition of $\models$\\
  
  $(W^B,M^B)\models \at{k}\pos{\act}k_1$ and $(W^B,M^B)\models \at{k_1}\gamma$ \newline 
  for some $k_1\in\F^\nom$ such that $W^B_{k_1}=w_1$ & 
  by Proposition~\ref{reach-HFOLR}, since $(W^B,M^B)$ is reachable \\ 
  
  $G\Vdash^k \pos{\act}k_1$ and $G\Vdash^{k_1} \gamma$ for some $k_1\in F^\nom$  &
  by the induction hypothesis \\
  
  $G\Vdash^k\pos{\act}\gamma$ & since $G$ is generic
 \end{proofsteps}
\item[For $\neg \gamma$]
The following are equivalent:
\begin{proofsteps}{24em}
  \label{ps: gmt-n-1} $(W^B,M^B)\models \at{k} \neg\gamma$ & 
  \\
  \label{gmt-n-2} $(W^B,M^B)\not \models \at{k}\gamma$ & 
  by the semantics of negation\\
  \label{gmt-n-3} $G\not \Vdash^k\gamma$ & 
  by the induction hypothesis \\
  \label{gmt-n-4} $p\not \Vdash^k\gamma$ for all $p\in G$ & 
  by the definition of $\Vdash$\\
  \label{gmt-n-5} $p\Vdash^k \neg\gamma$ for some $p\in G$ & 
  since $G$ is generic\\
  \label{gmt-n-6} $G\Vdash^k \neg\gamma$
\end{proofsteps}

\item[For $\vee \Gamma$]
The following are equivalent: 
\begin{proofsteps}{24em}
  $(W^B,M^B)\models \at{k} \vee \Gamma$ & 
  \\
  $(W^B,M^B)\models \at{k} \gamma$ for some $\gamma\in \Gamma$ &
  by the semantics of disjunction\\
  $G\Vdash^k \gamma$ for some $\gamma\in \Gamma$ & 
  by the induction hypothesis\\
  $G\Vdash^k \vee \Gamma$ & 
  by the definition of $\Vdash$
\end{proofsteps}
\item[For $\Exists{X}\gamma$]
Let $w=W^B_k$.
The following are equivalent:

\begin{proofsteps}{28em}
   \label{gmt-q-1} $(W^B,M^B)\models \at{k} \Exists{X}\gamma$ &  \\
   \label{gmt-q-2} $(W',M')\models^w \gamma$ for some expansion $(W',M')$ of $(W^B,M^B)$ to $\Delta[X]$ & 
   by the definition of $\models$\\
   
   \label{gmt-q-3} $(W^B,M^B)\models^w \theta(\gamma)$ for some substitution 
   $\theta\colon X\to\emptyset$ over $\Delta$ such that $(W^B,M^B)\red_\theta=(W',M')$ & since $(W^B,M^B)$ is reachable  \\
   
   \label{gmt-q-4} $G\Vdash^k\theta(\gamma)$ for some substitution $\theta\colon X\to\emptyset$ over $\Delta$ & 
   by the induction hypothesis\\
   \label{gmt-q-5} $G\Vdash^k\Exists{X}\gamma$ & 
   by the definition of $\Vdash$
\end{proofsteps}
\item[For $\store{z}\gamma$]
 This case is straightforward since $\at{k}\store{z}\gamma$ is semantically equivalent to $\at{k}\gamma(z\leftarrow k)$.
\item[For $\at{k'}\gamma$]
This case is straightforward since $\at{k} \at{k'}\gamma$ is semantically equivalent to $\at{k'} \gamma$.
\end{proofcases}
\psqed\end{proof}

Theorem~\ref{GMT} is a generalization of Generic Model Theorem for hybrid logics from \cite{gai-godel}.
The new cases from the present contribution correspond to structured actions, which include second, third and fourth cases.
 
\section{Semantic forcing property}\label{5}

We study a semantic forcing property, which will be used to prove the Omitting Types Theorem for a fragment $\L$ of $\HDFOLR$ semantically closed under negation and retrieve.

\begin{framework} \label{frame:context}
In this section, we arbitrarily fix
\begin{enumerate}
\item a signature $\Delta=(\Sigma^\nom,\Sigma^\rigid\subseteq\Sigma)$ of $\L$, 
 
\item a class $\mathcal{K}$  of Kripke structures over the signature $\Delta$, and

\item a sorted set $C=\{C_s\}_{s\in S^\ext}$ of new rigid constants for $\Delta$ such that $\card(C_s)=\alpha$ for all sorts $s\in S^\ext$, where $S^\ext=S^\rigid\cup\{\any\}$ is the extended set of rigid sorts and $\any$ is the sort of nominals.
\end{enumerate}
We let $\alpha$ denote the power of $\Delta$.
\end{framework}
If the set of sorts in $\Sigma$ is empty then $C$ consists only of nominals.

\begin{definition}\label{def:semantic-forcing}
The semantic forcing property $\mathbb{P}=(P,\leq,f)$ over the signature $\Delta[C]$ relative to the class of Kripke structures $\mathcal{K}$ is defined as follows:
\begin{enumerate}
\item $P=\{p\subseteq \Sen(\Delta[C])\mid \card(p)<\alpha \text{ and } (W,M)\models p \text{ for some }(W,M) \in|\Mod(\Delta[C])| \text{ s.t. }(W,M)\red_\Delta\in\mathcal{K}\}$, 

\item $\leq$ is the inclusion relation, and 

\item $f(p)=p\cap\Sen_b(\Delta[C])$ for all $p\in P$.
\end{enumerate}
\end{definition}

\begin{lemma} \label{lemma:semantic forcing}
 $\mathbb{P}=\langle P,\leq,f\rangle$ described in Definition~\ref{def:semantic-forcing} is a forcing property.
\end{lemma}
\begin{proof}
All conditions enumerated in Definition~\ref{def:forcing} obviously hold except the last one.
Assume that $f(p)\models \at{k}\gamma$, where $p\in P$ and $\at{k}\gamma\in\Sen_b(\Delta)$.   
Since $f(p)\subseteq p$, we have $p\models \at{k}\gamma$.
By Definition~\ref{def:semantic-forcing}, $(W,M)\models p$ for some $(W,M)\in|\Mod(\Delta[C])|$ such that $(W,M)\red_{\Delta}\in\mathcal{K}$.
Since $(W,M)\models p$ and $p\models \at{k}\gamma$, $(W,M)\models p\cup\{\at{k}\gamma \}$.
Hence, $q\coloneqq p\cup\{\at{k}\gamma\}\in P$ and $p\leq q$ 
\end{proof}

\begin{proposition}\label{prop:semantic-forcing}
$\mathbb{P}=\langle P, \leq, f \rangle $ described in Definition~\ref{def:semantic-forcing} has the following properties:

\begin{enumerate}[label=P\arabic*),ref=P\arabic*]
\item \label{sf1} If $p\in P$ and $\at{k} \pos{\act_1\comp\act_2}k''\in p$ then $p\cup\{\at{k}\pos{\act_1}k',\at{k'}\pos{\act_2}k''\}\in P$ for some nominal $k'\in C_\any$.
\item \label{sf2} 
If $p\in P$ and $\at{k}\pos{\act}\gamma\in p$ with $\gamma\not\in F^\nom \cup C_\any $ 
then $p\cup\{\at{k}\pos{\act}k',\at{k'}\gamma\}\in P$ for some nominal $k'\in C_\any $.
\item \label{sf3} If $p\in P$ and $\at{k}\vee \Gamma\in p$ then $p\cup\{\at{k} \gamma\}\in P$ for some $\gamma\in\Gamma$.
\item  \label{sf4} 
If $p\in P$ and $\at{k} \Exists{X}\gamma\in p$ then there exists an injective mapping
$f\colon X\to C$ such that $p\cup \{\at{k}\chi(\gamma)\}\in P$, where $\chi\colon \Delta[C,X]\to \Delta[C]$ is the unique extension of $f$ to a signature morphism which preserves $\Delta[C]$.
\end{enumerate}
\end{proposition}

\begin{proof} 
Let $p\in P$ be a condition.
By the definition of $\mathbb{P}$, we have that $p\subseteq \Sen(\Delta[C'])$  for some $C'\subset C$ with $\card(C'_s)<\alpha$ for all $s\in S^\ext$.

\begin{enumerate}[label=P\arabic*)]
\item Assume that $\at{k} \pos{\act_1\comp\act_2}k''\in p$.
Since $\card(C_\any)=\alpha$ and $\card(C'_\any)<\alpha$, there exists $k'\in C_\any\setminus C'_\any$.
We show that $p\cup\{\at{k}\pos{\act_1}k',\at{k'}\pos{\act_2}k''\}\in P$:
\begin{proofsteps}{30em}
$(W,M)\models p$ for some model $(W,M)$ over $\Delta[C]$ with $(W,M)\red_\Delta\in\mathcal{K}$ 
& by the definition of $\mathbb{P}$\\

$(W',M')\coloneqq(W,M)\red_{\Delta[C']}\models p$ & by the satisfaction condition \\

$(W'_k,w)\in W'_{\act_1}$ and $(w,W'_{k''})\in W'_{\act_2}$ for some $w\in |W'|$ & 
since $(W',M')\models\at{k}\pos{\act_1\comp\act_2}k''$\\

$(W'',M')\models \at{k}\pos{\act_1}k'$ and $(W'',M')\models \at{k'}\pos{\act_2}k''$,
where $(W'',M')$ is the unique expansion of $(W',M')$ to $\Delta[C',k']$ interpreting $k'$ as $w$ \\

$(V,N)\models p\cup\{\at{k}\pos{\act_1}k',\at{k'}\pos{\act_2}k''\}$, where $(V,N)$ is any expansion of $(W'',M')$ to $\Delta[C]$ 
& by the satisfaction condition, since $(W'',M')\models p\cup\{\at{k}\pos{\act_1}k',\at{k'}\pos{\act_2}k''\}$ \\

$p\cup\{\at{k}\pos{\act_1}k',\at{k'}\pos{\act_2}k''\}\in P$ 
& since $(V,N)\models p\cup\{\at{k}\pos{\act_1}k',\at{k'}\pos{\act_2}k''\}$ and $(V,N)\red_\Delta=(W',M')\red_\Delta\in\mathcal{K}$
\end{proofsteps}
\item Assume that $\at{k}\pos{\act}\gamma\in p$ with $\gamma\not\in F^\nom \cup C_\any $.
Since $\card(C_\any)=\alpha$ and $\card(C'_\any)<\alpha$, there exists $k'\in C_\any\setminus C'_\any$. 
We show that $p\cup\{\at{k}\pos{\act}k',\at{k'}\gamma\}\in P$:
\begin{proofsteps}{28em}
$(W,M)\models p$ for some model $(W,M)$ over $\Delta[C]$ with $(W,M)\red_\Delta\in\mathcal{K}$ 
& by the definition of $\mathbb{P}$\\
$(W',M')\coloneqq(W,M)\red_{\Delta[C']}\models p$ & by the satisfaction condition \\
$(W'_k,w)\in W'_{\act}$ and $(W',M')\models^w\gamma$ for some $w\in |W'|$ &
since $(W',M')\models \at{k}\pos{\act}\gamma$\\
$(W'',M')\models \at{k}\pos{\act} k'$ and $(W'',M')\models \at{k'}\gamma$,
where $(W'',M')$ is the unique expansion of $(W',M')$ to $\Delta[C,k']$ interpreting $k'$ as $w$
& by semantics\\
$(V,N)\models p\cup\{\at{k}\pos{\act} k', \at{k'}\gamma \}$, where $(V,N)$ is any expansion of $(W'',M')$ to $\Delta[C]$ 
&
by the satisfaction condition, since $(W'',M')\models p\cup\{\at{k}\pos{\act} k', \at{k'}\gamma \}$\\
$p\cup\{\at{k}\pos{\act} k', \at{k'}\gamma \}\in P$ 
& since $(V,N)\red_\Delta=(W',M')\red_\Delta\in\mathcal{K}$ 
\end{proofsteps}

\item Assume that $\at{k}\vee \Gamma\in p$.
There exists a Kripke structure $(W,M)$ over $\Delta[C]$ such that $(W,M)\models p$ and $(W,M)\red_\Delta\in\mathcal{K}$.
Since $(W,M)\models \at{k}\vee\Gamma$, we have $(W,M)\models \at{k}\gamma$ for some $\gamma\in\Gamma$.
Since $(W,M)\models p$, $(W,M)\models \at{k}\gamma$  and  $(W,M)\red_\Delta\in\mathcal{K}$, we obtain $p\cup\{\at{k}\gamma\}\in P$. 

\item Assume that $\at{k} \Exists{X} \gamma\in p$.
Since $\card(C'_s)<\alpha$ and $\card(C_s)=\alpha$ for all sorts $s\in S^\ext$, by the finiteness of $X$, there exists an injective mapping $f\colon X \to C\setminus C'$.
Let $C''\coloneqq C'\cup f(X)$.
Let $\chi'\colon \Delta[C',X]\to \Delta[C'']$ be the unique extension of $f$ to a signature morphism which preserves $\Delta[C']$.
Let $\chi\colon \Delta[C,X]\to \Delta[C]$ be the unique extension of $f$ to a signature morphism which preserves $\Delta[C]$.
Let $\iota\colon\Delta[C'']\hookrightarrow \Delta[C]$ and $\iota'\colon\Delta[C',X]\hookrightarrow \Delta[C,X]$ be inclusions.
Since $\chi$ and $\chi'$ agree on $X$ and they preserve the rest of the symbols, we have $\chi'\comp\iota=\iota'\comp \chi$.

$$\xymatrix{ & \Delta[C',X] \ar@{-->}[rr]^{\iota'} \ar@{-->}[d]^{\chi'} & & \Delta[C,X] \ar@{-->}[d]^{\chi}\\
               \Delta[C'] \ar@{-->}[ur]^{\subseteq} \ar[r]_{\subseteq} & \Delta[C''] \ar[r]_\iota & \Delta[C] \ar[ur]^{\subseteq} \ar[r]_{1_{\Delta[C]}} & \Delta[C]
}$$
We show that $p\cup \{\at{k}\chi(\gamma)\}\in P$:
\begin{proofsteps}{23em}
$(W,M)\models p$ for some Kripke structure $(W,M)$ over the signature $\Delta[C]$ such that $(W,M)\red_\Delta\in\mathcal{K}$ 
& by the definition of $\mathbb{P}$\\ 

$(W',M')\coloneqq (W,M)\red_{\Delta[C']}\models p$ & by the satisfaction condition \\

$(V',N')\models^w \gamma$ for some expansion $(V',N')$ of $(W',M')$ to the signature $\Delta[C',X]$, 
where $w= W'_k=V'_k$
& since $\at{k}\Exists{X}\gamma\in p$ and $(W',M')\models p$\\

\label{sfp-4} 
let $(V'',N'')$ be the unique $\chi'$-expansion of $(V',N')$  
& $(V'',N'')$ exists, as $\chi'$ is a bijection\\ 

\label{sfp-5} 
let $(V,N)$ be any expansion of $(V'',N'')$ to $\Delta[C]$ & \\

$(V,N)\red_\chi\red_{\iota'}=(V,N)\red_\iota\red_{\chi'}=(V'',N'')\red_{\chi'}=(V',N')$
& from \ref{sfp-4} and \ref{sfp-5}, since $\iota'\comp \chi=\chi'\comp \iota$\\

$(V,N)\red_\chi\models^w\gamma$ 

& by the local satisfaction condition, 

since $(V,N)\red_\chi\red_{\iota'}=(V',N')\models^w\gamma$\\

$(V,N)\models^w\chi(\gamma)$ & by the local satisfaction condition \\

\label{sf3-10} $(V,N) \models \at{k}\chi(\gamma)$ & since  $w=V'_k=(V\red_\chi\red_{\iota'})_k=V_k$\\

$(V,N)\models p$ & by the satisfaction condition, 

since $(V,N)\red_{\Delta[C']}=(W',M')\models p$\\

\label{sf3-12} $(V,N)\red_\Delta\in\mathcal{K}$ 

& since $(V,N)\red_{\Delta[C']}=(W',M')$ and $(W',M')\red_\Delta\in\mathcal{K}$\\

$p\cup \{\at{k}\chi(\gamma)\}\in P$ & from \ref{sf3-10}---\ref{sf3-12}
\end{proofsteps}
\end{enumerate}
\end{proof}

Proposition~\ref{prop:semantic-forcing} sets the basis for the following important result concerning semantic forcing properties, which says that all sentences of a given condition are forced eventually by some condition greater or equal than the initial one.

\begin{theorem}[Semantic Forcing Theorem] \label{th:semantic-forcing}
Let $\mathbb{P}=\langle P, \leq, f \rangle $ be the semantic forcing property described in Definition~\ref{def:semantic-forcing}.
For all $\Delta[C]$-sentences $\at{k}\gamma$ and conditions $p\in P$ we have:
$$ q\Vdash^k \gamma \mbox{ for some }q\geq p\mbox{ iff } p\cup\{\at{k}\gamma\}\in P.$$
\end{theorem}
\begin{proof}
 We proceed by induction on the structure of $\gamma$.
  \begin{proofcases}[itemsep=1.5ex] 
 \item[For $\gamma$ extended atomic]
 Assume that there is $q\geq p$ such that $q\Vdash^k\gamma$. 
 We show that $p\cup\{\at{k}\gamma\}\in P$:
 \begin{proofsteps}{20em}
   $\at{k}\gamma\in q$ & 
   by Definition~\ref{def:forcing-relation}\\
   
   $p\cup\{\at{k}\gamma\}\leq q$ & 
   since $q\geq p$ \\
   
   $(W,M)\models q$ for some Kripke structure $(W,M)$ over the signature $\Delta[C]$ such that $(W,M)\red_\Delta\in\mathcal{K}$ &
   since $q\in P$\\ 
   
   $p\cup\{\at{k}\gamma\}\in P$ & 
   since $(W,M)\models p\cup\{\at{k}\gamma\}$ and $(W,M)\red_\Delta\in\mathcal{K}$
 \end{proofsteps}
 
 Assume that $p\cup\{\at{k}\gamma\}\in P$. 
 Let $q=p\cup\{\at{k}\gamma\}$. 
 By Definition~\ref{def:forcing-relation}, $q\Vdash^k \gamma$.
 \item [For $\pos{\act_1\cup\act_2}k''$]
The following are equivalent:
 \begin{proofsteps}{26em}
  $q\Vdash^k \pos{\act_1\cup\act_2}k''$ for some $q\geq p$ & \\
  $q\Vdash^k\pos{\act_1}k''$ or $q\Vdash^k\pos{\act_2}k'' $ & by Definition~\ref{def:forcing-relation}\\
  
  $p\cup\{\at{k}\pos{\act_1}k''\}\in P$ or $p\cup\{\at{k}\pos{\act_2}k''\}\in P$ & 
  by the induction hypothesis \\
  
  $(W,M)\models p\cup\{\at{k}\pos{\act_1}k''\} $ or $(W,M)\models p\cup\{\at{k}\pos{\act_2}k''\} $
  for some Kripke structure $(W,M)$ over $\Delta[C]$ such that $(W,M)\red_\Delta\in\mathcal{K}$ & 
  by Definition~\ref{def:semantic-forcing}\\
  
  $(W,M)\models p\cup \{\at{k}\pos{\act_1\cup\act_2}k''\}$ for some Kripke structure $(W,M)$ over $\Delta[C]$ such that $(W,M)\red_\Delta\in\mathcal{K}$
  & by the semantics of $\act_1\cup\act_2$ \\
  
  $p\cup\{\at{k}\pos{\act_1\cup\act_2}k''\} \in P$ &
  since $(W,M)\models p\cup \{ \at{k}\pos{\act_1\cup\act_2}k''\}$ and $(W,M)\red_\Delta\in\mathcal{K}$
 \end{proofsteps}
 \item[For $\pos{\act_1\comp\act_2}k''$]
 Assume that $q\Vdash^k \pos{\act_1\comp\act_2}k''$ for some $q\geq p$.
 We show that $p\cup\{\pos{\act_1\comp\act_2}k''\}\in P$:
 
 \begin{proofsteps}{24em}
 
 $q\Vdash^k\pos{\act_1}k'$ and $q\Vdash^{k'}\pos{\act_2}k''$ for some $k'\in F^\nom \cup C_\any $ 
 & by Definition~\ref{def:forcing-relation}\\
 
 $q\cup \{\at{k} \pos{\act_1}k'\}\in P$ & by the induction hypothesis, since $q\leq q$\\
 
 $q\cup \{\at{k} \pos{\act_1}k'\}\Vdash^{k'} \pos{\act_2}k''$ & by Lemma~\ref{lemma:forcing-property} (\ref{p2}), since $q\Vdash^{k'}\pos{\act_2}k''$ and $q\leq q\cup \{ \at{k}\pos{\act_1} k'\}$ \\
 
 $p\cup \{\at{k} \pos{\act_1}k'\}\cup\{\at{k'}\pos{\act_2}k''\}\in P$ & 
 by the induction hypothesis, since $p\cup \{ \pos{\act_1}k'\}\leq q\cup \{ \pos{\act_1}k'\}$\\
 
 $p\cup\{\at{k} \pos{\act_1\comp\act_2}k''\}\in P$ & by the definition of $\mathbb{P}$
 \end{proofsteps}
Assume that $p\cup\{\at{k}\pos{\act_1\comp\act_2}k''\}\in P$. 
We show that $q\Vdash^k \pos{\act_1\comp\act_2}k''$ for some $q\geq p$:
\begin{proofsteps}{28em}
 $p\cup\{\at{k}\pos{\act_1\comp\act_2}k'', \at{k} \pos{\act_1}k', \at{k'}\pos{\act_2}k''\}\in P$ for some $k\in C_\any $ &  by Proposition~\ref{prop:semantic-forcing} (\ref{sf1})\\
 
 let $r\coloneqq p\cup\{\at{k}\pos{\act_1\comp\act_2}k'', \at{k} \pos{\act_1}k', \at{k'}\pos{\act_2}k''\}$ & \\
 
 $s\Vdash^k \pos{\act_1}k'$ for some $s\geq r$ 

 & by the induction hypothesis, 
 
 since $r\cup\{ \at{k} \pos{\act_1} k'\}=r\in P$ \\
 
 \label{comp:sfp-4} $q \Vdash^{k'} \pos{\act_2}k''$ for some $q\geq s$ 
 
 & by the induction hypothesis, 
 
 since $s\cup\{\at{k'}\pos{\act_2}k''\}=s \in P$ \\
 
 \label{comp:sfp-5} $q \Vdash^k \pos{\act_1}k'$ 
 
 & by Lemma~\ref{lemma:forcing-property} (\ref{p2}), since $s\Vdash^k \pos{\act_1}k'$ and $q\geq s$ \\
 
 $q\Vdash^k \pos{\act_1\comp\act_2}k''$ & from \ref{comp:sfp-4} and \ref{comp:sfp-5}
\end{proofsteps}

 \item[For $\pos{\act^*}k''$]
 The following are equivalent:
 
\begin{proofsteps}{24em}
$q\Vdash^k\pos{\act^*}k''$ for some $q\geq p$ & \\

$q\Vdash^k \pos{\act^n}k''$ for some $q\geq p$ and $n\in\N$ 
& by Definition~\ref{def:forcing-relation} \\
 
$p\cup\{\at{k} \pos{\act^n}k''\}\in P$ for some $n\in\N$ 
&  by the induction hypothesis\\
 
$p\cup\{\at{k} \pos{\act^*}k''\}\in P$ 
& by the semantics of $\act^*$ and the definition of $\mathbb{P}$
 \end{proofsteps}
\item[For $\pos{\act}\gamma$ with $\gamma\not\in{F^\nom \cup C_\any }$] 
Assume that $q\Vdash^k\pos{\act}\gamma$ for some $q\geq p$.
 We show that $p\cup\{\at{k} \pos{\act} \gamma\} \in P$:
 \begin{proofsteps}{24em}
 
  $q\Vdash^k \pos{\act}k'$ and $q\Vdash^{k'} \gamma$ for some nominal $k'$ & 
  from $q\Vdash^k \pos{\act}\gamma$, by Definition~\ref{def:forcing-relation} \\
  
  $q\cup\{\at{k'}\gamma\}\in P$ &
  from $q\leq q$ and $q\Vdash^{k'} \gamma$, by the induction hypothesis \\
  
  $\at{k}\pos{\act}k'\in q$ &
  from $q\Vdash^k \pos{\act}k'$, by Definition~\ref{def:forcing-relation} \\
   
  $(W,M)\models q\cup \{\at{k}\pos{\act}k', \at{k'}\gamma \}$ for some Kripke structure $(W,M)$ over $\Delta[C]$ such that $(W,M)\red_\Delta\in\mathcal{K}$ 
  & 
  since $q\cup\{\at{k'}\gamma\}\in P$ and $\at{k}\pos{\act}k'\in q$ \\
  
  $(W,M)\models q\cup \{\at{k}\pos{\act}\gamma \}$ for some Kripke structure $(W,M)$ over $\Delta[C]$ such that $(W,M)\red_\Delta\in\mathcal{K}$ 
  & since $\{\at{k}\pos{\act}k', \at{k'}\gamma \}\models \at{k}\pos{\act}\gamma$\\
  
  $q\cup \{\at{k}\pos{\act}\gamma \}\in P$ & by Definition~\ref{def:semantic-forcing}\\
 
  $p\cup \{\at{k}\pos{\act}\gamma \}\in P$ & since $p\subseteq q$
 \end{proofsteps}

Assume that $p\cup \{ \at{k}\pos{\act}\gamma\}\in P$.
We show that $q\Vdash^k \pos{\act}\gamma$ for some $q\geq p$:
 \begin{proofsteps}{28em}
   $(p\cup \{ \at{k}\pos{\act}\gamma\})\cup \{ \at{k} \pos{\act}k',\at{k'} \gamma \} \in P$ 
   for some nominal $k'\in F^\nom \cup C_\any $ & 
  by Proposition~\ref{prop:semantic-forcing}~(\ref{sf2}) \\
  
   \label{ps:cfp-pos2}  $q\Vdash^{k'} \gamma$ for some $q\geq p\cup \{ \at{k}\pos{\act} \gamma, \at{k} \pos{\act}k' \}$ & 
  by the induction hypothesis\\
  
   \label{ps:cfp-pos3} $q\Vdash^k \pos{\act}k'$ & since $\at{k} \pos{\act}k' \in f(q)$\\
  
   $q\Vdash^k \pos{\act} \gamma$ & by \ref{ps:cfp-pos3} and \ref{ps:cfp-pos2}
 \end{proofsteps}

\item[For $\neg\gamma$] By the induction hypothesis, for each $q\in P$ we have
 \begin{enumerate}[label=(S\arabic*), ref= S\arabic*]
 
 \item $r\Vdash^k \gamma$  for some $r\geq q$       iff  $q\cup\{ \at{k}\gamma\}\in P$, which is equivalent to
 
 \item $r\not\Vdash^k \gamma$  for all  $r\geq q$   iff  $q\cup \{ \at{k}\gamma\}\not \in P$, which is equivalent to

 \item \label{negIH} $q\Vdash^k \neg\gamma$        iff  $q\cup \{ \at{k}\gamma\}\not \in P$.
 
 \end{enumerate}

Assume that $q\Vdash^k \neg\gamma$  for some $q\geq p$. 
We show that $p\cup\{\at{k} \neg\gamma\}\in P$:
\begin{proofsteps}{27em}
  $q\cup\{ \at{k}\gamma\}\not\in P$ & by statement~\ref{negIH}\\
  
  $(W,M)\models q$ for some Kripke structure $(W,M)$ over $\Delta[C]$ such that $(W,M)\red_\Delta\in\mathcal{K}$ & by Definition~\ref{def:semantic-forcing}, since $q\in P$ \\
 
  $(W,M)\not \models \at{k}\gamma $ & since $q\cup\{\at{k}\gamma\}\not\in P$ \\

  $(W,M)\models \at{k}\neg\gamma $ & 
  by the semantics of $\neg$  \\  
   
 $q\cup\{\at{k}\neg\gamma\}\in P$ & since $(W,M)\models q\cup\{\at{k}\neg\gamma\}$ and $(W,M)\red_\Delta\in\mathcal{K}$ \\
 
 $p\cup\{\at{k}\neg\gamma\}\in P$ & since $p\cup\{\at{k}\neg\gamma\}\subseteq q\cup\{\at{k}\neg\gamma\}$
\end{proofsteps}
  
   Assume that $p\cup\{\at{k}\neg\gamma\}\in P$.  
 We show that $q\Vdash^k\neg\gamma$ for some $q\geq p$:
 \begin{proofsteps}{20em}
  let  $q=p\cup\{\at{k} \neg\gamma\}$ & \\ 
  $q\cup\{ \at{k}\gamma\}\not\in P$ & since $\at{k}\neg\gamma \in q$ \\
  $q\Vdash^k \neg\gamma$ & by statement~\ref{negIH} 
\end{proofsteps}
\item[For $\vee \Gamma$] Assume that there exists $q\geq p$ such that $q\Vdash^k \vee\Gamma$. 
We show that $p\cup\{\at{k} \vee\Gamma\}\in P$:

 \begin{proofsteps}{20em}
  $q\Vdash^k\gamma$ for some $\gamma\in\Gamma$ & 
  by Definition~\ref{def:forcing-relation}\\
  
  $p\cup\{\at{k}\gamma\}\in P$ & 
  by the induction hypothesis\\
  
  $(W,M)\models p\cup \{\at{k}\gamma\}$ for some Kripke structure $(W,M)$ over $\Delta[C]$ such that $(W,M)\red_\Delta\in\mathcal{K}$ & by Definition~\ref{def:semantic-forcing} \\  
  
  $(W,M)\models p\cup \{\at{k} \vee\Gamma \}$  for some Kripke structure $(W,M)$ over $\Delta[C]$ such that $(W,M)\red_\Delta\in\mathcal{K}$ & by the semantics of $\vee$\\
  
  $p\cup \{\at{k} \vee\Gamma \} \in P$ & by Definition~\ref{def:semantic-forcing}
  \end{proofsteps}

Assume that $p\cup \{\at{k}\vee\Gamma\}\in P$.
We show that $q\Vdash^k \vee\Gamma$ for some $q\geq p$:
 \begin{proofsteps}{20em}
  \label{ps:cfp-dis1} $(p\cup\{\at{k}\vee\Gamma\})\cup\{\at{k}\gamma\}\in P$ for some $\gamma\in \Gamma$ & 
  by Proposition~\ref{prop:semantic-forcing} (\ref{sf3}) \\
  \label{ps:cfp-dis2} $q\Vdash^k\gamma$ for some $q\geq p\cup \{\at{k} \vee\Gamma\}$ & 
  by the induction hypothesis\\
  \label{ps:cfp-dis3} $q\Vdash^k \vee\Gamma$ for some $q\geq p$ & by Definition~\ref{def:forcing-relation}
 \end{proofsteps}
\item[For $\Exists{X}\gamma$] 
  Assume that $q\Vdash^k \Exists{X}\gamma$ for some $q\geq p$.
  We show that $p\cup\{\at{k}\Exists{X}\gamma\}\in P$:
  \begin{proofsteps}{30em}
   $q\Vdash^k\theta(\gamma)$ for some substitution $\theta : X\to\emptyset$ &
  by Definition~\ref{def:forcing-relation} \\
  
   $p\cup\{\at{k}\theta(\gamma)\}\in P$ & 
   by the induction hypothesis \\ 
  
   $(W,M)\models p\cup \{ \at{k}\theta(\gamma) \}$ for some Kripke structure $(W,M)$ over $\Delta[C]$ such that $(W,M)\red_\Delta\in\mathcal{K}$ & 
   by Definition~\ref{def:semantic-forcing}\\
  
  $(W,M)\models p\cup \{ \at{k}\Exists{X}\gamma \}$  and $(W,M)\red_\Delta\in\mathcal{K}$ & 
   by semantics \\
   
  $p\cup\{\at{k}\Exists{X}\gamma\}\in P$ & by Definition~\ref{def:semantic-forcing}
 \end{proofsteps}

 We assume that $p\cup\{\at{k}\Exists{X}\gamma\}\in P$.
 We show that $q\Vdash^k \Exists{X}\gamma$ for some $q\geq p$:

\begin{proofsteps}{30em}
$(p\cup\{\at{k}\Exists{X}\gamma\})\cup \{\at{k}\chi(\gamma)\}\in P$ \newline
for some signature morphism $\chi\colon \Delta[C,X]\to\Delta[C]$ which preserves $\Delta[C]$ &  
by Proposition~\ref{prop:semantic-forcing} (\ref{sf4})\\
    
$q\Vdash^k \chi(\gamma)$ for some $q\geq p\cup\{\at{k}\Exists{X}\gamma\}$ & 
by the induction hypothesis \\
    
$q\Vdash^k\Exists{X}\gamma$ for some $q\geq p$ & by Definition~\ref{def:forcing-relation}
\end{proofsteps} 
\item[For $\store{z}\gamma$] This case is straightforward, as $\at{k}\store{z}\gamma$ is semantically equivalent to $\at{k}\gamma(z\leftarrow k)$.
\item[For $\at{k'}\gamma$]
This case is straightforward, as $\at{k}\at{k'}\gamma$ is semantically equivalent to $\at{k'}\gamma$.
\end{proofcases}
\psqed\end{proof}
 The following result is a corollary of Theorem~\ref{th:semantic-forcing}. 
 It shows that each generic set of a given semantic forcing property has a reachable model that satisfies all its conditions.
\begin{corollary} 
 \label{cor:semantic-forcing}
 Let $\mathbb{P}=\langle P, \leq, f \rangle $ be the semantic forcing property described in Definition~\ref{def:semantic-forcing}.

Then for each generic set $G$ we have:
 \begin{enumerate}[label=C\arabic*)]
 \item \label{cor:sfp1} $G\Vdash^k\gamma$ for all conditions $p\in G$, sentences $\gamma\in p$ and nominals $k\in F^\nom \cup C_\any $.

 \item There exists a generic structure $(W^G,M^G)$ for $G$ which is reachable and satisfies each condition $p\in G$.
 \end{enumerate}
\end{corollary}

\begin{proof}\

\begin{enumerate}[label=C\arabic*)]
 \item Suppose towards a contradiction that $G\not\Vdash^k\gamma$ for some  $p\in G$, $\gamma\in p$ and nominal $k\in F^\nom \cup C_\any$. 
Then:
\begin{proofsteps}{22em}
$q\Vdash^k \neg\gamma$ for some $q\in G$ & since $G$ is generic\\  
  
$r\geq p$ and $r\geq q$ for some $r \in G$ & since $G$ is generic \\
 
$\gamma\in r$ & since $\gamma\in p$ and $r\geq p$\\
  
$r\cup\{\at{k}\gamma\}\in P$ & 
since $r\models\at{k}\gamma$\\
  
$s\Vdash^k \gamma$ for some $s\geq r$ & by Theorem~\ref{th:semantic-forcing} \\
  
$s\Vdash^k\neg\gamma$ 
& by Lemma~\ref{lemma:forcing-property} (\ref{p2}) since $s\geq q$ and $q\Vdash^k\neg\gamma$ \\
  
contradiction & by Lemma~\ref{lemma:forcing-property} (\ref{p4}) since $s\Vdash^k \gamma$ and $s\Vdash^k\neg\gamma$ 
\end{proofsteps}
It follows that  $G\Vdash^k\gamma$ for all $p\in G$, $\gamma\in p$ and nominals $k$.
\item  By Theorem~\ref{GMT}, there exists a generic model $(W^G,M^G)$ for $G$ which is reachable. 
Let $p\in G$, $\gamma\in p$ and $w\in|W^G|$.
Since $(W^G,M^G)$ is reachable, $w$ is the denotation of some nominal $k\in F^\nom \cup C_\any$.
By the first part of the proof, $G\Vdash^k \gamma$.
Since $(W^G,M^G)$ is a model for $G$, $(W^G,M^G)\models \at{k}\gamma$.
Hence, $(W^G,M^G)\models^w\gamma$. As $w$ was arbitrary, we have
$(W^G,M^G)\models\gamma$.
\end{enumerate}
\psqed\end{proof}
\section{Omitting Types Theorem}\label{6}

Let $\Delta=(\Sigma^\nom,\Sigma^\rigid\subseteq \Sigma^\nom)$ be a countable signature.
Let $C = \{C_s\}_{s\in S^\ext}$ be a finite set of new constants of extended rigid sorts. 
We say that a Kripke structure $(W,M)$ over $\Delta$ realizes a set $\Gamma$ of sentences over $\Delta[C]$ iff
there exists an expansion $(V,N)$ of $(W,M)$ to $\Delta[C]$ such that $(V,N)\models \Gamma$.
We say that $(W,M)$ omits $\Gamma$ if $(W,M)$ does not realize $\Gamma$.
We say that a satisfiable set $T$ of sentences over $\Delta$ locally realizes $\Gamma$ if there exists a finite set $p$ of sentences over $\Delta[C]$ such that 
$T\cup p$ is satisfiable, and 
$T\cup p\models \Gamma$.
In the following we generalize these definitions to signatures of any power. 

\begin{definition} [Omitting Types semantically] \label{def:sem-OT}
Assume a signature $\Delta=(\Sigma^\nom,\Sigma^\rigid\subseteq \Sigma^\nom)$, and let $\alpha$ be the power of $\Delta$.
Let $X=\{X_s\}_{s\in S^\ext}$ be a sorted set of variables for $\Delta$ such that $\card(X_s)<\omega$ for all sorts $s\in S^\ext$.
\begin{itemize}
\item A Kripke structure $(W,M)$ over $\Delta$ realizes a type $\Gamma\subseteq \Sen(\Delta[X])$ if 
there exists an expansion $(V,N)$ of $(W,M)$ to $\Delta[X]$ such that $(V,N)\models\Gamma$.
\item A Kripke structure $(W,M)$ over $\Delta$ omits a set $\Gamma$ of $\Delta[X]$-sentences if $(W,M)$ does not realize $\Gamma$.
\end{itemize}
\end{definition}

Classically, $\Gamma$ from Definition~\ref{def:sem-OT} is called a type with free variables $X$.

\begin{definition}[Omitting Types syntactically] \label{def:OTS}
Let $\Delta$ be a signature, and let $\alpha$ be the power of $\Delta$.
Let $X=\{X_s\}_{s\in S^\ext}$ be a sorted set of variables for $\Delta$ such that $X_s$ is finite for all sorts $s\in S^\ext$.
A theory $T\subseteq\Sen(\Delta)$  \emph{$\alpha$-realizes} a type $\Gamma\subseteq\Sen(\Delta[X])$ if there exist 
\begin{itemize}
\item a sorted set $C=\{C_s\}_{s\in S^\ext}$ of new constants for $\Delta$ with $\card(C_s)<\alpha$ for all $s\in S^\ext$,

\item a substitution $\theta : X \to C$, and

\item a set of sentences $p$ over $\Delta[C]$ with $\card(p)<\alpha$,
\end{itemize}
such that $T\cup p$ is satisfiable and $T\cup p\models \theta(\Gamma)$.
We say that $T$ $\alpha$-omits $\Gamma$ if $T$ does not $\alpha$-realize $\Gamma$.
\end{definition}

Notice that the power of any signature is at least $\omega$. 
If $\alpha=\omega$, we say that $T$ locally omits $\Gamma$ instead of $T$ $\alpha$-omits $\Gamma$.
Definition~\ref{def:OTS} is similar to the definition of locally omitting types for first-order logic without equality from \cite{keisler01}.
Our results are applicable to fragments $\L$ without equality.
We give a couple of equivalent descriptions of the omitting types property which can be found in the literature.

\begin{lemma} \label{lemma:loc-ott} \

\begin{enumerate}[label=L\arabic*),ref=L\arabic*]
\item \label{lemma:loc-ott-1} Assume that $\L$ is semantically closed under equality.
\footnote{$\L$ is semantically closed under equality whenever
\begin{enumerate*}[label=(\alph*)]
\item~for any nominal $k$ there exists an $\L$-sentence $\varphi$ such that $(W,M)\models^w\varphi$ iff $w=W_k$ for all Kripke structures $(W,M)$ and all possible worlds $w\in|W|$, and
\item~for any open terms $t_1,t_2\in T_{\overline\Sigma}$  there exists an $\L$-sentence $\varphi$ such that $(W,M)\models^w\varphi$ iff $M_{w,t_1}=M_{w,t_2}$ for all Kripke structures $(W,M)$ and all possible worlds $w\in|W|$.
\end{enumerate*}}
Then $T$ $\alpha$-realizes $\Gamma$ as described in Definition~\ref{def:OTS} iff 
there exist
\begin{enumerate*}
\item a sorted set $C=\{C_s\}_{s\in S^\ext}$ of new constants for $\Delta[X]$ with $\card(C_s)<\alpha$ for all $s\in S^\ext$, and
\item a set of sentences $p$ over $\Delta[C,X]$ with $\card(p)<\alpha$,
\end{enumerate*} 
such that $T\cup p$ is satisfiable and $T\cup p\models \Gamma$.

\item \label{lemma:loc-ott-2} Assume that $\L$ is semantically closed under equality, conjunction and quantifiers.
Then $T$  locally realizes $\Gamma$ iff 
there exists a finite set of $\Delta[X]$-sentences $p$
such that $T\cup p$ is satisfiable and $T\cup p\models \Gamma$.

\item \label{lemma:loc-ott-3} Assume that $\L$ is compact and semantically closed under equality, conjunction and quantifiers.
Then $T$ $\alpha$-realizes $\Gamma$ iff 
there exists a set of $\Delta[X]$-sentences $p$ with $\card(p)<\alpha$ such that $T\cup p$ is satisfiable and $T\cup p\models \Gamma$.

\end{enumerate}
\end{lemma}
\begin{proof} 
The backward implication is straightforward for all cases.
Therefore, we will focus on the forward implication.
$$
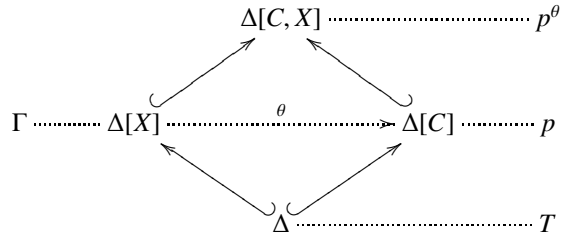
\xymatrix{
&  & \Delta[C,X]& & p^\theta \ar@{.}[ll]\\
\Gamma \ar@{.}[r]& \Delta[X] \ar@{^{(}->}[ur]\ar@{.>}[rr]^\theta && \Delta[C] \ar@{_{(}->}[ul] & \ar@{.}[l] p\\
& & \Delta \ar@{_{(}->}[ul] \ar@{^{(}->}[ur] & & T \ar@{.}[ll]
}$$ 

Let $\theta\colon X\to C$ be a substitution with $\card(C_s)<\alpha$ for all $s\in S^\ext$, and 
let $p$ be a set of sentences over $\Delta[C]$ with $\card(p)<\alpha$ such that $T\cup p$ is satisfiable and $T\cup p\models \theta(\Gamma)$.
Without loss of generality, we assume that $X\cap C=\emptyset$. 
Since in all three cases $\L$ is semantically closed under equality, there exists a set of sentences $p^\theta$ over $\Delta[C,X]$ semantically equivalent with $\{ x=\theta(x)\mid x\in X\}$
\footnote{Here $=$ is a shorthand from the metalanguage. In particular, for nominals
$x=\theta(x)$ means that $\at{x}\theta(x)$ for all $x\in X_\any$.}.
Since $T\cup p$ is satisfiable,
$T\cup p\cup p^\theta$ is satisfiable too. Now we consider three cases.  
\begin{enumerate}[label=L\arabic*)]
\item As $p\cup p^\theta$ is a set of sentences over $\Delta[C,X]$), we show
that $T\cup p\cup p^\theta\models \Gamma$: 
%
\begin{proofsteps}{24em}
let $(W,M)\in|\Mod(\Delta[C,X])|$ such that $(W,M)\models T\cup p \cup p^\theta$ & \\
$(W,M)\red_{\Delta[C]}\red_\theta=(W,M)\red_{\Delta[X]}$ & since $(W,M)\models p^\theta$ \\
$(W,M)\red_{\Delta[C]}\models T\cup p$ & by the satisfaction condition, since $(W,M)\models T\cup p$ \\
$(W,M)\red_{\Delta[C]}\models \theta(\Gamma)$ & since $T\cup p\models \theta(\Gamma)$ and $(W,M)\red_{\Delta[C]}\models T\cup p$ \\
$(W,M)\red_{\Delta[X]}=(W,M)\red_{\Delta[C]}\red_\theta\models\Gamma$ & by the satisfaction condition for substitutions\\
$(W,M)\models \Gamma$ & by the satisfaction condition\\
$T\cup p\cup p^\theta\models\Gamma$ & since $(W,M)$ was arbitrarily chosen
\end{proofsteps}
\item If $\alpha=\omega$, we show $T\cup\{\varphi\}\models \Gamma$ for
  a single sentence $\varphi$ over $\Delta[X]$:
\begin{proofsteps}{24em}
the sets $C$, $p$ and $p^\theta$ are finite & since their cardinals are strictly less than $\omega$\\
there exists a $\Delta[X]$-sentence $\varphi$ semantically equivalent with $\Exists{C}\Forall{z}\at{z}\wedge (p\cup p^\theta)$ & 
since $\L$ is semantically closed under conjunction and quantifiers\\
$T \cup \{\varphi\}$
is satisfiable over $\Delta[X]$ & since $T\cup p \cup p^\theta$ is satisfiable over $\Delta[C,X]$\\
%
$T\cup\{\varphi\}\models \Gamma$ &  since $T\cup p \cup p^\theta\models\Gamma$
\end{proofsteps}
\item If $\L$ is compact, we show that $T\cup p'\models\Gamma$ for
a set $p'$ of sentences over $\Delta[X]$ with $\card(p')<\alpha$:  
\begin{proofsteps}{27em}
for each $\gamma\in \Gamma$  there exists $p^\gamma\subseteq p\cup p^\theta$ finite such that $T\cup p^\gamma\models\gamma$ & 
by compactness, since $T\cup p \cup p^\theta\models \gamma$ for all $\gamma\in \Gamma$\\
let $C^\gamma$ be all constants from $C$ which occur in $p^\gamma$ for all $\gamma\in \Gamma$ & \\
there exists a set $p'$ of $\Delta[X]$-sentences semantically equivalent with
$\{\Exists{C^\gamma}\Forall{z}\at{z}\wedge p^\gamma \mid \gamma\in\Gamma\}$  & 
since $\L$ is semantically closed under conjunction, retrieve and quantifiers\\
$T\cup p'$
is satisfiable over $\Delta[C]$ &
since $T\cup p\cup p^\theta$ is satisfiable over $\Delta[C,X]$ \\
$T\cup p'\models\Gamma$  & 
since $T\cup p^\gamma\models \gamma$ for all $\gamma\in\Gamma$\\
$\card(\mathcal{P}_\omega(p\cup p^\theta))<\alpha$ &
since $\card(p)<\alpha$ and $\card(p^\theta)<\alpha$\\
$\card(\{p^\gamma\mid \gamma\in \Gamma\})<\alpha$ & 
since $\{p^\gamma\mid \gamma\in \Gamma\}\subseteq \mathcal{P}_\omega(p\cup p^\theta)$\\
$\card(p')<\alpha$ &  by its definition, $p'$ is in one-to-one
correspondence with $\{p^\gamma \mid \gamma\in\Gamma\}$
\end{proofsteps}
\end{enumerate}
\psqed\end{proof} 

The following result is needed for proving the Omitting Types Theorem.
\begin{lemma} \label{lemma:a-ott}
Assume that $T$ $\alpha$-omits $\Gamma$ as described in Definition~\ref{def:OTS}.
Then for any substitution $\theta : X \to C$ over $\Delta$ such that $\card(C_s)<\alpha$ for all $s\in S^\ext$, and 
any set of $\Delta[C]$-sentences $p$ such that $\card(p)<\alpha$ and $T\cup p$ is satisfiable,
there exists $\gamma\in \Gamma$ such that 
$T\cup p \cup \{\at{z}\neg\theta(\gamma)\}$ is satisfiable, where $z$ is a nominal variable for $\Delta[C]$.
\end{lemma}
\begin{proof}
Let $C=\{C_s\}_{s\in S^\ext}$ be a set of new constants for $\Delta$ with $\card(C_s)<\alpha$ for all $s\in S^\ext$.
Let $\theta : X \to C$ be a substitution over $\Delta$.
Let $p$ be a set of $\Delta[C]$-sentences such that $\card(p)<\alpha$ and $T\cup p$ satisfiable.
Since $T$ $\alpha$-omits $\Gamma$, we have $T\cup p\not \models \theta(\Gamma)$. 
There exists a Kripke structure $(W,M)$ over $\Delta[C]$ such that $(W,M)\models T\cup p$ and $(W,M)\not\models\theta(\Gamma)$.
It follows that $(W,M)\models^w \neg\theta(\gamma)$ for some possible world $w\in |W|$ and some sentence $\gamma\in \Gamma$.
Let $z$ be a new nominal for $\Delta[C]$, and let $(W^{z\leftarrow w},M)$ be the unique expansion of $(W,M)$ to $\Delta[z,C]$ which interprets $z$ as $w$.
Since $(W,M)\models^w\neg\theta(\gamma)$, we get $(W^{z\leftarrow w},M)\models \at{z}\neg\theta(\gamma)$. 
Hence, $T\cup p \cup\{\at{z}\neg\theta(\gamma)\}$ is satisfiable.
\end{proof}
\begin{definition}[Omitting Types Property] \label{def:OTP}
We say that $\L$ has $\alpha$-Omitting Types Property ($\alpha$-OTP), where $\alpha$ is an infinite cardinal, whenever
\begin{itemize}
\item for all signatures $\Delta$ of power at most $\alpha$,
\item all satisfiable theories $T\subseteq \Sen(\Delta)$, and
\item all families of types $\{\Gamma^i\subseteq \Sen(\Delta[X^i]) \mid i<\alpha \}$,

where $X^i=\{X^i_s\}_{s\in S^\ext}$ is a set of variables for $\Delta$ with $\card(X^i_s)<\omega$ for all $s\in S^\ext$,
\end{itemize}
such that $T$ $\alpha$-omits $\Gamma^i$ for all $i<\alpha$,
there exists a Kripke structure over $\Delta$ which satisfies $T$ and omits $\Gamma^i$ for all $i<\alpha$.
If for all signatures $\Delta\in |\Sig^\L|$ and all cardinals $\alpha$ equal or greater than the power of $\Delta$, 
$\L$ has $\alpha$-OTP then $\L$ has OTP. 
\end{definition}
All the ingredients for proving Omitting Types Theorem are in place.
\begin{theorem}[Extended Omitting Types Theorem]\label{th:eott}
Let $\alpha$ be an infinite cardinal.  
Assume that $\L$ is semantically closed under retrieve and negation, and
if $\alpha > \omega$ assume that $\L$ is compact.
Then $\L$ has $\alpha$-OTP.
\end{theorem}

\begin{proof}
Assume that $T$ $\alpha$-omits $\Gamma^i$ as described in Definition~\ref{def:OTS}.
Let $C=\{C_s\}_{s\in S^\ext}$ be a sorted set of new constants for $\Delta$ such that $\card(C_s)=\alpha$ for all $s\in S^\ext$.
Let $\mathbb{P}=(P,\leq,f)$ be the semantic forcing property described in Definition~\ref{def:semantic-forcing} with $\mathcal{K}=|\Mod(\Delta,T)|$. 
The proof is performed in four steps.
$$\xymatrix{ & \Delta[X^i] \ar@{^{(}->}[rr] \ar@{.>}[d]^{\theta'} & & \Delta[C,X^i] \ar@{.>}[d]^\theta \\
            \Delta \ar@{^{(}->}[ur] \ar@{^{(}->}[r] & \Delta[C'] \ar@{^{(}->}[r] & \Delta[C] \ar@{^{(}->}[ur]
            \ar[r]_{1_{\Delta[C]}}& \Delta[C]
}$$
\begin{enumerate}[label=S\arabic*)]
\item We show that for 
any condition $p\in P$,
any index $i<\alpha$, and 
any substitution $\theta : X^i \to \emptyset$ over $\Delta[C]$,
there exist a sentence $\gamma\in\Gamma^i$ and a nominal $c\in C$ such that $p\cup \{\at{c}\neg\theta(\gamma)\}\in P$:
\begin{proofsteps}{23em}
let $C^p$ be the set of all constants from $C$ which occur in $p$ & \\
there exists $c\in C_\any\setminus (\theta(X^i_\any) \cup C^p_\any)$ & 
since $\card(\theta(X^i_\any))<\omega$, $\card(C^p_\any)<\alpha$ and $\card(C_\any)=\alpha$\\
let $C'\coloneqq\theta(X^i)\cup C^p \cup\{c\}$  & \\
let $\theta' : X^i \to C'$ be the substitution over $\Delta$ defined by $\theta(x)=\theta'(x)$ for all $x\in X^i$ &\\
$T\cup p\cup\{\at{c}\neg\theta'(\gamma)\}$ is satisfiable for some $\gamma\in\Gamma^i$ 
& by Lemma~\ref{lemma:a-ott}, since $T$ $\alpha$-omits $\Gamma^i$\\
$p\cup\{\at{c}\neg\theta(\gamma)\}\in P$ &  
since $\at{c}\neg\theta(\gamma)=\at{c}\neg\theta'(\gamma)$,
we have $(W,M)\models T \cup p \cup\{\at{c}\neg\theta(\gamma)\}$ for some $(W,M)\in|\Mod(\Delta[C])|$

\end{proofsteps}
\item The cardinality of the set $\mathcal{S}^i$ of all substitutions $\theta : X^i \to \emptyset$ over $\Delta[C]$ is equal or less than $\alpha$.
It follows that the cardinality of $\mathcal{S}\coloneqq \bigcup_{i<\alpha}\mathcal{S}^i$ is equal or less than $\alpha$.
Let $\{\theta^j : X^{i_j} \to \emptyset\in\mathcal{S} \mid j < \alpha\}$ be an enumeration of $\mathcal{S}$.
Let $\{\at{k_j}\varphi_j\in \Sen(\Delta[C]) \mid j < \alpha\}$ be an enumeration of the $\Delta[C]$-sentences with retrieve as the top operator.
We define an increasing chain of conditions $p_0\leq p_1\leq\dots$ by induction on ordinals:
\begin{proofcases}[itemsep=1.5ex]
\item[$j=0$] $p_0\coloneqq\emptyset$.
\item[$j\Rightarrow j+1$] 
If $p_j \Vdash^{k_j}\neg\varphi_j$ then let $q\coloneqq p_j$ else let $q\geq p_j$ be a condition such that $q\Vdash^{k_j}\varphi_j$.
By the first part of the proof, 
there exist $\gamma\in\Gamma^{i_j}$ and $c\in C$ such that $q\cup\{ \at{c}\neg \theta^j(\gamma)\}\in P$.
Let $p_{j+1}\coloneqq q\cup\{ \at{c}\neg \theta^j(\gamma)\}$.
\item[$\beta<\alpha$ limit ordinal] $p_\beta\coloneqq\bigcup_{j<\beta}p_j$.
Since $\card(p_j)<\alpha$ for all $j<\beta$ and $\beta<\alpha$, we have $\card(p_\beta)<\alpha$.
Since $p_j\in P$ for all $j<\beta$, the set $T\cup p_j$ is satisfiable for all $j<\beta$.
By compactness\footnote{If there exists a limit ordinal $\beta<\alpha$ then $\alpha$ is not
  countable, so we assume $\L$ is compact.}, $(\bigcup_{j<\beta} p_j)\cup T$ is satisfiable too.
Hence, $p_\beta\in P$.

\end{proofcases}
The set $G=\{q\in P\mid q\leq p_{j+1} \text{ for some } j<\alpha\}$ is generic.
Let $k\in F^\nom \cup C_\any $ and $\psi\in T$.
Suppose towards a contradiction that $q\Vdash^k\neg \psi$ for some $q\in G$ then:
\begin{proofsteps}{28em}
$q\cup \{\at{k}\psi\}\in P$ & since $\psi\in T$ and $q\cup T$ is satisfiable\\
\label{ps:ott-2} $r\Vdash^k\psi$ for some $r\geq q$ & by Theorem~\ref{th:semantic-forcing}\\
\label{ps:ott-3} $r\Vdash^k\neg\psi$ & since $q\Vdash^k\neg\psi$ and $r\geq q$\\
contradiction & by Lemma~\ref{lemma:forcing-property} (\ref{p4}) from \ref{ps:ott-2} and \ref{ps:ott-3} 
\end{proofsteps}
Since $G$ is generic, $q\Vdash^k\psi$ for some $q\in G$.
\item By Theorem~\ref{GMT}, there exists a generic Kripke structure $(W,M)$ for $G$ that is reachable.
Let $(V,N)\coloneqq (W,M)\red_\Delta$.
We show that $(V,N)\models T$:
\begin{proofsteps}{21em}
let $w\in|W|$ and $\psi\in T$ & \\
$W_k=w$ for some $k\in  F^\nom \cup C_\any $ & since $(W,M)$ is reachable \\
$G\Vdash^k\psi$ & by the second part of the proof \\
$(W,M)\models \at{k}\psi$ & since $(W,M)$ is generic for $G$ \\
$(W,M)\models^w \psi$ & by the semantics of $@$, since $W_k=w$\\
$(W,M)\models T$ & since $w\in |W|$ and $\psi\in T$ were arbitrarily chosen \\
$(V,N)\models T$ & by the satisfaction condition, since $(W,M)\red_\Delta=(V,N)$
\end{proofsteps}
\item We show that $(V,N)$ omits $\Gamma^i$ for all $i<\alpha$:
\begin{proofsteps}{27.5em}
let $(V',N')$ be an arbitrary expansion of $(V,N)$ to $\Delta[X^i]$ & \\
there exists an expansion $(W',M')$ of $(W,M)$ to $\Delta[C,X^i]$ such that $(W',M')\red_{\Delta[X^i]}=(V',N')$ &
by interpreting $X^i$ as $(V',N')$ interprets $X^i$\\
there exists $\theta^j : X^i \to \emptyset \in\mathcal{S}$ such that $(W,M)\red_{\theta^j}=(W',M')$ 
& since $(W,M)$ is reachable\\ 
there exist $c\in C$ and $\gamma\in \Gamma^i$ such that $\at{c}\neg\theta^j(\gamma) \in p_{j+1}$ 
& by the construction of the chain $p_0\leq p_1\leq \dots$ \\
$(W,M)\models p_{j+1}$ & by Corollary~\ref{cor:semantic-forcing}, since $p_{j+1}\in G$  and $(W,M)$ is generic for $G$\\
$(W,M)\models^w\neg \theta^j(\gamma)$, where $w=W_{c}$ & since $\at{c}\neg\theta^j(\gamma)\in p_{j+1}$  \\
$(W',M')\models^w \neg \gamma$ & by the local satisfaction condition for $\theta^j$ \\
$(V',N')\models^w \neg \gamma$ & by the local satisfaction condition, since $(W',M')\red_{\Delta[X^i]}=(V',N')$\\
$(V',N')\not\models \Gamma^i$ & since $\gamma\in\Gamma^i$\\
$(V,N)$ omits $\Gamma^i$ & since $(V',N')$ is an arbitrary expansion of $(V,N)$
\end{proofsteps}
\end{enumerate}
We conclude that $(V,N)$ is a Kripke structure over $\Delta$ which satisfies $T$ and omits $\Gamma^i$ for all $i<\alpha$.
\end{proof}
Any fragment $\L$ of $\HDFOLR$ free of the Kleene operator is compact.
If, in addition, $\L$ is semantically closed under negation and retrieve, $\L$ is an instance of Theorem~\ref{th:eott}. 
In particular, any fragment presented in Examples~\ref{ex:HFOLR} --- \ref{ex:HFOLS} can be an instance of $\L$ from Theorem~\ref{th:eott}.
Omitting Types Theorem is obtained from Theorem~\ref{th:eott} by restricting the signatures $\Delta$ to countable ones.
By Lemma~\ref{lemma:loc-ott} (\ref{lemma:loc-ott-2}), Omitting Types Theorem is a corollary of Extended Omitting Types Theorem.

Notice that the forcing technique developed in the present contribution is not applicable to $\HFOLS$ as this logic lacks support for the substitutions described in Section~\ref{sec:subst}.
However, by Lemma~\ref{lemma:HFOLR-HFOLS}, OTP can be borrowed from $\HFOLR$ to $\HFOLS$.
\begin{theorem} \label{th:eott-HFOLS} 
$\HFOLS$ has $\alpha$-OTP for all infinite cardinals $\alpha$.
\end{theorem}
\begin{proof}
Recall that for all $\HFOLS$ signatures $\Delta$, we have:
\begin{itemize}
\item $\Sen^\HFOLS(\Delta)\subseteq \Sen^\HFOLR(\Delta)$, and
\item by Lemma~\ref{lemma:HFOLR-HFOLS}, for every sentence $\gamma\in\Sen^\HFOLR(\Delta)$ there exists a sentence $\gamma'\in\Sen^\HFOLS(\Delta)$ which is satisfied by the same class of Kripke structures as $\gamma$.
\end{itemize}
Assume that $T$ $\alpha$-omits $\Gamma^i$ as described in Definition~\ref{def:OTS}.
By the remarks above, $T$ $\alpha$-omits $\Gamma^i$ in $\HFOLR$ for all $i<\alpha$.
By Theorem~\ref{th:eott}, there exists a Kripke structure $(W,M)$ over $\Delta$, which satisfies $T$ and omits $\Gamma^i$ for all $i<\alpha$.
\end{proof}
It is worth noting that in general the
Omitting Types Property cannot be borrowed from a given
logic to its restrictions. 
If $T$ omits $\Gamma^i$ in a restriction
then $T$ might not omit $\Gamma^i$ in
the full underlying logic. 
This is the reason for developing Theorem~\ref{th:eott} in an arbitrary fragment
$\L$ of $\HDFOLR$. 
\section{Constructor-based completeness}\label{7}

Constructor-based completeness is a modern approach to the well-known $\omega$-completeness, which has applications in formal methods. 
We make the result independent of the arithmetic signature by working over an arbitrary vocabulary where we distinguish a set of constructors which determines a class of Kripke structures reachable by constructors.
Throughout this section we assume that the fragment $\L$ is semantically closed under equality, negation, retrieve, disjunction and quantifiers. 
An example of such fragment $\L$ is $\HDFOLR$ or $\HDPL$.

\subsection{Semantic restrictions}
Given a theory $T$ over a vocabulary $\Delta$, not all Kripke structures are of interest.
In many cases, formal methods practitioners are interested in the properties of a class Kripke structures that are reachable by a set of constructor operators.
Let $\Delta=(\Sigma^\nom,\Sigma^\rigid\subseteq\Sigma)$ be a signature and 
$\Sigma^{\cons}\subseteq\Sigma^\rigid$ a subset of constructor operators.
The constructors create a partition of the set of rigid sorts $S^\rigid$.
A \emph{constrained} sort is a rigid sort $s\in S^\rigid$ that has a constructor, that is, there exists a constructor $\sigma\colon w\to s$ in $\Sigma^\cons$.
A rigid sort that is not constrained it is called \emph{loose}.
We denote by $S^\cons$ the set of all constrained sorts, and by $S^\loose$ the set of all loose sorts.
Let $Y=\{Y_s\}_{s\in S^\loose}$ be a set of loose variables such that $Y_s$ is countably infinite for all $s\in S^\loose$.
A constructor-based Kripke structure is a Kripke structure $(W,M)$ such that 
\begin{itemize}
\item for all possible worlds $w\in|W|$ there exists a nominal $k\in F^\nom$ such that $w= W_k$, and
\item for all rigid sorts $s\in S^\rigid$, all possible worlds $w\in|W|$, and all elements $m\in M_{w,s}$ there exist 
an expansion $(W,N)$ of $(W,M)$ to $\Delta[Y]$, and 
a rigid term $t\in T_{@\Sigma^\cons}(Y)$ such that $m=N_{w,t}$.
\end{itemize}
\begin{example}
Let $\Delta$ be the signature defined as follows:
\begin{itemize}
\item 
$\Sigma^\nom=( F^\nom, P^\nom)$ such that $ F^\nom$ consists of all natural numbers, and
$ P^\nom$ has one element $\lambda$.

\item $\Sigma=(S,F,P)$, $S= \{Elt, List\}$, $F= \{empty:\to List, cons : Elt~List \to List, delete: List \to List \}$ and $P=\emptyset$.

\item $S^\rigid=S$ and $F^\rigid=\{empty:\to List, cons : Elt~List \to List\}$.

\item The set of constructors $F^\cons$ is $F^\rigid$. 
\end{itemize}

This means that $List$ is constrained while $Elt$ is loose.

We define a theory $T$ over $\Delta$, which deletes $n$ elements from a list in each possible world $n$:

\begin{itemize}
\item $\{ \at{n}\pos{\lambda} n+1 \mid n \geq 0 \} \cup \{ \neg \at{n} m \mid n \neq m\}$, 
\item $\{ \Forall{N}(\at{N} delete)(empty)=empty, \Forall{L}(\at{0}delete)(L)=L\}$, and
\item $\{ \Forall{E,L} (\at{n+1}delete) cons(E,L) = (\at{n}delete)(L) \mid n >0 \}$.
\end{itemize}
A constructor-based Kripke structure which satisfies $T$ has the set of possible worlds isomorphic with $\omega$.
Let $(W,M)$ be the constructor-based Kripke structure such that
(a)~$|W|=\omega$ and $W_\lambda$ is $<$, the natural order on $\omega$, and
(b)~for all possible worlds $n\in\omega$, 
the first-order-structure $M_n$ interprets 
$Elt$ as an arbitrary set, and
$List$ as the set of lists with elements from $M_{n,Elt}$, 
while the function $M_{n,delete}\colon M_{n,List}\to M_{n,List}$ deletes the first $n$ elements from the list given as argument.
Obviously, $(W,M)$ satisfies $T$.
\end{example}

By enhancing the syntax with a subset of rigid constructor operators and by
restricting the semantics to constructor-based Kripke structures, we obtain a
new logic $\L^\cons$ from $\L$. Note that 
restricting the semantics also changes
the relation $\models$, applied to theories:
$T\models\varphi$ now means that all restricted
models of $T$ are models of $\varphi$, so there may be non-restricted models of
$T$ which are not models of $\varphi$.

\subsection{Entailment systems}

Given a system of proof rules for $\L$  which is sound and complete, the goal
is to add some new proof rules such that the resulting proof system is sound and
complete for $\L^\cons$. 

\begin{definition}[Entailment relation] \label{def:entail} 
 An \emph{entailment relation} for $\L$ is a family of binary relations between sets of sentences indexed by signatures $\vdash=\{\vdash_\Delta\}_{\Delta\in|\Sig^\L|}$ with the following properties:
\begin{longtable}{l l}
\Monotonicity\inferrule{\Phi_1\subseteq \Phi_2}{\Phi_2 \vdash \Phi_1}
& \Transitivity\inferrule{\Phi_1\vdash \Phi_2 \Space \Phi_2\vdash \Phi_3}{\Phi_1\vdash \Phi_3}\\ \\
\Union\inferrule{\Phi_1\vdash \varphi_2 \text{ for all }\varphi_2\in\Phi_2}{\Phi_1\vdash\Phi_2} & 
\Translation\inferrule{\Phi_1 \vdash_\Sigma \Phi_2}{\chi(\Phi_1) \vdash \chi(\Phi_2)} where $\chi\colon\Delta\to\Delta'$ \\
\end{longtable}
\end{definition}

The entailment relation is sound (complete) if $\vdash\subseteq \models$ ($\models \subseteq \vdash$).
Examples of sound and complete entailment relations for $\HFOLR$ and $\HPL$ can be found in \cite{gai-godel}.

\begin{definition} [Constructor-based entailment relation]
Let $\vdash$ be an entailment relation for $\L$.
The entailment relation $\vdash^\cons$ for $\L^\cons$ is the least entailment
relation closed under the following proof rules: \newline
\begin{longtable}{l l l}
(R0)~\inferrule{\Phi\vdash \varphi}{\Phi\vdash^\cons \varphi} &
(R1)~\inferrule{\Phi\vdash^\cons \at{k_1} \varphi(k_2) \text{ for all }k_1, k_2\in F^\nom}{\Phi\vdash^\cons \Forall{x}\varphi(x)} &
(R2)~\inferrule{\Phi\vdash^\cons \at{k}\Forall{Y_t}\psi(t) \text{ for all } k\in F^\nom \text{ and } t\in T_{\Sigma^\cons}(Y)}{\Phi \vdash^\cons \Forall{y}\psi(y)} \\\
& & where 
$Y_t$ is the set of variables occurring in $t$\\
\end{longtable}
\end{definition}
According to \cite{dia-pro}, the entailment relation $\vdash^\cons$ is well defined.
We say that a theory $T$ in $\L$ is \emph{semantically closed under (R1)} if $T\models \at{k_1}\varphi(k_2)$ for all $k_1,k_2\in F^\nom$ implies $T\models \Forall{x}\varphi(x)$.
Similarly, we define the closure under (R2), that is, $T\models \at{k}\Forall{Y_t}\psi(t)$ for all $k\in F^\nom$ and $t\in T_{@\Sigma}(Y)$ implies $T\models \Forall{y}\psi(y)$.
It is not difficult to check that $\vdash^\cons$ is sound for $\L^\cons$ provided that $\vdash$ is sound for $\L$.
Completeness is much more difficult to establish in general,
but it can be done with the help of the OTP. 
\begin{theorem}[Constructor-based completeness]
The entailment relation $\vdash^\cons$ is complete for $\L^\cons$ if $\vdash$ is complete for $\L$ and $\L$ has OTP.
\end{theorem}
\begin{proof} 
Let $\Delta=(\Sigma^\nom,\Sigma^\rigid\subseteq\Sigma)$ be a signature and $T$ a
theory over $\Delta$ in $\L$. 
Let $\Sigma^\cons\subseteq \Sigma^\rigid$ be a set of constructors, and $Y$ a
set of loose variables. 
We perform the proof in two steps.

\begin{enumerate}[label=(S\arabic*)]

\item We show that if $T$ is satisfiable in $\L$ and semantically
  closed under (R1) and (R2) then $T$ is satisfiable in $\L^\cons$. 
Let $\Gamma^\nom\coloneqq\{\neg  \at{k} x \mid k\in F^\nom \}$ be a type in one nominal variable $x$, and
let $\Gamma^\rigid\coloneqq\{\Forall{Y_t}\neg t= y\mid t\in T_{\Sigma^\cons}(Y)\}$ be a type in one constrained variable $y$.
Any Kripke structure over $\Delta$ which omits $\Gamma^\nom$ and $\Gamma^\rigid$ is reachable by the constructors in $\Sigma^\cons$. 
Firstly, we show that $T$ locally omits $\Gamma^\nom$:
\begin{proofsteps}{24em} 
let $\rho(x)$ be a $\Delta[x]$-sentence such that $T\cup \{\rho(x)\}$ is satisfiable & \\
$T \cup \{ \Exists{x}\Forall{z}\at{z}\rho(x) \}$ is satisfiable & by semantics since $T\cup \{\rho(x)\}$ is satisfiable \\
$T\not\models \Forall{x} \neg \Forall{z}\at{z}\rho(x)$ & 
since $(W,M)\models T \cup \{\Forall{z}\at{z}\rho(x)\}$ for some Kripke structure $(W,M)$ over $\Delta[x]$ \\
$T\not\models \at{k_1} \neg \Forall{z}\at{z} \rho(k_2)$ for some nominals $k_1,k_2\in F^\nom$ 
& since $T$ is semantically closed under (R1)\\
$T\cup\{ \at{k_1} \Forall{z}\at{z}\rho(k_2)\}$ is satisfiable & by semantics of
negation and retrieve \\ 
$T\cup \{ \at{k_1}\Forall{z}\at{z}\rho(x) \} \cup \{ \at{k_2} x \}$ is satisfiable 
& by the semantics of nominals\\
$T\cup \{\rho(x) \} \cup \{ \at{k_2} x \}$ is satisfiable 
& by Lemma~\ref{lemma:h-prop} \\
$T$ locally omits $\Gamma^\nom$ & since $\rho(x)$ was arbitrarily chosen
\end{proofsteps}

Secondly, we show that $T$ locally omits $\Gamma^\rigid$:
\begin{proofsteps}{24em}
let $\rho(y)$ be a $\Delta[y]$-sentence such that $T\cup\{\rho(y)\}$ is satisfiable & \\
$T\cup\{\Exists{y}\Forall{z}\at{z}\rho(y)\}$ is satisfiable & 
by Lemma~\ref{lemma:h-prop}\\
$T\not\models \Forall{y}\neg\Forall{z}\at{z}\rho(y)$ 
& since $(W,M)\models T \cup \{\Forall{z}\at{z}\rho(y)\}$ for some Kripke structure $(W,M)$ over $\Delta[y]$ \\
$T\not\models\at{k}\Forall{Y_t}\neg\Forall{z}\at{z}\rho(t)$ for some $k\in
F^\nom$ and
$t\in T_{@\Sigma^\cons}(Y)$
& since $T$ is semantically closed under (R2)\\
$T\not\models \at{k}\neg\Forall{z}\at{z}\rho(t)$
over $\Delta[Y_t]$ & by semantics of quantifiers\\
$T\cup\{  \at{k}\Forall{z}\at{z}\rho(t)\}$ is satisfiable over $\Delta[Y_t]$ 
& by semantics of negation and retrieve \\
$T\cup\{ \rho(t)\} $ is satisfiable over $\Delta[Y_t]$ & 
by Lemma~\ref{lemma:h-prop}\\
$T\cup \{\rho(y)\}\cup \{\Exists{Y_t} t= y\}$ is satisfiable & 
since $(W,M)\models T\cup\{\rho(t) \}$ for some Kripke structure $(W,M)$ over $\Delta[Y_t]$ \\
$T$ locally omits $\Gamma^\rigid$ & since $\rho(y)$ was arbitrarily chosen
\end{proofsteps}

By Theorem~\ref{th:eott}, there exists a Kripke structure $(W,M)$ which satisfies $T$ and omits $\Gamma^\nom$ and $\Gamma^\rigid$.
By the definition of $\Gamma^\nom$ and $\Gamma^\rigid$, $(W,M)$ is a constructor-based Kripke structure.

\item Next we assume
$T$ is consistent in $\L^\cons$ and show that $T$ is satisfiable in $\L^\cons$.
Let $T'\coloneqq \{\varphi \in \Sen(\Delta)\mid T\vdash^\cons \varphi\}$.
We have that $T$ is consistent in $\L^\cons$ iff  $T'$ is consistent in $\L$:
\begin{itemize}
\item[] For the forward
implication, suppose towards a contradiction that $T'$ is not consistent in
$\L$, that is, $T'\vdash\bot$; 
By (R0), $T'\vdash^\cons\bot$;
by \Union, $T\vdash^\cons T'$;
by \Transitivity, $T\vdash^\cons\bot$, 
which is a contradiction with the consistency of $T$ in $\L^\cons$.

For the backward implication, suppose towards a contradiction that $T\vdash^\cons\bot$;
we have $\bot\in T'$, and by \Monotonicity, 
$T'\vdash \bot$, which is a contradiction with the consistency of $T'$ in $\L$. 
\end{itemize}
Assume that $T$ is consistent in $\L^\cons$.
It follows that $T'$ is consistent in $\L$.
By the completeness of $\vdash$ in $\L$, $T'$ is satisfiable in $\L$.
By the completeness of $\vdash$ in $\L$, $T'$ is semantically closed under (R1) and (R2).
By the first part of the proof, $T'$ is satisfiable in $\L^\cons$.
Since $T\subseteq T'$, $T$ is satisfiable in $\L^\cons$.  
\end{enumerate}
\end{proof}

\section{Omitting types and L\"owenheim-Skolem Theorems}\label{8}
Downwards and Upwards L\"owenheim-Skolem Theorems are consequences of the Omitting
Types Theorem. 
Throughout this section we assume that the fragment $\L$ is semantically closed under 
equality, 
retrieve, 
negation, disjunction,  
possibility over binary modalities, and 
quantifiers. 
An example of such fragment $\L$ is $\HDFOLR$ or $\HDPL$,
in which case $\L$ has $\omega$-OTP.
For cardinals greater than $\omega$, we need to drop the Kleene operator $*$
in order to have compactness and be able to apply our OTP (we will show in the next
section that compactness is necessary at least for certain strong fragments of $\L$). 
Some of the arguments in this and the next section are modelled after
the technique used by Lindstr\"om~\cite{lind78} for first-order logic without
equality.
\begin{theorem}[Downwards L\"owenheim-Skolem Theorem] \label{th:DLS}
Assume that $\L$ has $\alpha$-OTP.
Let $T$ be a satisfiable theory over a signature $\Delta$ of power at most $\alpha$.
Then $T$ has a Kripke structure $(W,M)$ such that 
$\card(W)\leq\alpha$ and $\card(M_{w,s})\leq\alpha$ for all rigid sorts $s\in S^\rigid$.
\end{theorem}

\begin{proof}
Let $C=\{C_s\}_{s\in S^\ext}$ be a sorted set of new constants for $\Delta$ such that $\card(C_s)=\alpha$ for all sorts $s\in S^\ext$.
Let $\Gamma^s\coloneqq\{ c \neq x \mid c \in C_s\}$ be a
type\footnote{Notice that for nominals, $c\neq x$ means $\neg\at{c}x$. Compare
Lemma~\ref{lemma:loc-ott} for a similar use.}
in one variable
$x$ of sort $s\in S^\ext$.
We show that $T$ $\alpha$-omits $\Gamma^s$:
\begin{proofsteps}{28em}
let $p$ be a set of sentences over $\Delta[C,x]$ such that \rlap{$\card(p)<\alpha$ and $T\cup p$ is satisfiable} & \\
$p\subseteq \Delta[C',x]$ for some $C'\subseteq C$ such that $\card(C'_s)< \alpha$ 
& since $\card(p)<\alpha$ \\
there exists $c\in C_s\setminus C_s'$ & since $\card(C_s)=\alpha$ and $\card(C_s')<\alpha$ \\
$T\cup p\cup \{ x = c\}$ is satisfiable & since $T\cup p$ is satisfiable and $c$ does not occur in $T\cup p$\\
$T$ $\alpha$-omits $\Gamma^s$ & since $p$ was arbitrarily chosen
\end{proofsteps}
Since $\L$ has $\alpha$-OTP, there exists a Kripke structure $(W,M)$ over $\Delta[C]$ which satisfies $T$ and omits $\Gamma^s$ for all $s\in S^\ext$.
\end{proof}

\begin{theorem}[Upwards L\"owenheim-Skolem Theorem] \label{th:ULS}
Assume that $\L$ has $\alpha$-OTP, where $\alpha$ is a regular cardinal.
Let $T$ be a satisfiable theory over a signature $\Delta$ of power at most $\alpha$.
For each model $(W,M)$ of $T$ there exists another model $(V,N)$ of $T$ such that $\card((V,N)_s)\geq\alpha$ for all sorts $s\in S^\ext$.

In fact, 
if $\Delta'$ is obtained from $\Delta$ by adding a rigid binary relation $\leq$ on each sort $s\in S^\ext$ interpreted by $(W,M)$ as infinite
then there exists an expansion $(V',N')$ of $(V,N)$ to $\Delta'$ such that $\pos{(W,M)_s,(W,M)_\leq}$ is a linear ordering of cofinality $\alpha$ for all sorts $s\in S^\ext$.
\end{theorem}

\begin{proof}
Let $\Omega\subseteq S^\ext$ be the set of all sorts interpreted by $(W,M)$ as infinite.
Let $C=\{C_s\}_{s\in\Omega}$ be a set of new rigid constants such that $C_s=\{c_i\mid i < \alpha \}$ for all $s\in\Omega$.
Let $T'$ be the theory over $\Delta'[C]$ obtained from $T$ by adding:
\[
\{ \leq\text{ is a linear order on } s \text{ without the greatest element}\}  \cup \{ c_i \leq c_j \mid i < j <\alpha \}
\text{ for each sort } s\in\Omega
\] 
The definition of $T'$ relies on the semantic closure of $\L$ under the relevant
sentence building operators. For example, for nominals, $c_i\leq c_j$ means
$\at{c_i}\pos{\leq}c_j$. 
There exists an expansion $(W',M')$ of $(W,M)$ to the signature $\Delta'[C]$ such that $(W',M')\models T'$. 
For each sort $s\in\Omega$ we define the following type in one variable $x$ of sort $s$:
\[\Gamma^s\coloneqq\{c_i\leq x \mid i<\alpha\} \] 
We show that $T'$ $\alpha$-omits $\Gamma^s$:

\begin{proofsteps}{28em}
let $p\subseteq \Sen(\Delta'[C,x])$ with $\card(p)<\alpha$ such that $T'\cup p$ is satisfiable & \\
$(V,N)\models T'\cup p$ for some Kripke structure $(V,N)$ over $\Delta'[C,x]$ & since $T'\cup p$ is satisfiable \\
$p\subseteq \Sen(\Delta'[C^\beta,x])$ for some $\beta<\alpha$, 
where $C^\beta$ is obtained from $C$ by restricting the constants of sort $s$ to
$C^\beta_{s}\coloneqq\{ c_i\in C_{s}\mid i < \beta\}$
& since $\alpha$ is regular \\
$(V^\beta,N^\beta)\models T\cup p$, where $(V^\beta,N^\beta)\coloneqq (V,N)\red_{\Delta'[C^\beta, x]}$ &
since $(V,N)\models T'\cup p$ and $T\subseteq T'$\\
there exists $w > max\{(V,N)_x,(V,N)_{c_\beta}\}$ & since $\pos{(V,N)_s,(V,N)_\leq}$ is a linear order without the greatest element  \\
$w > (V,N)_{c_i}$ for all $i<\beta$ 
& since $w > (V,N)_{c_\beta}$ and $(V,N)_{c_\beta} \geq (V,N)_{c_i}$ for all $i<\beta$ \\

\label{ps:ls-7}
$(V',N')\models T'\cup p$, where $(V',N')$ is the unique expansion of $(V^\beta,N^\beta)$ to the signature $\Delta'[C,x]$ such that $(V',N')_{c_i}= w$ for all $i\geq\beta$ &
since $(V^\beta,N^\beta)\models T\cup p$ and $w$ is greater than the interpretation of $c_\beta$ in $(V,N)$\\
\label{ps:ls-8} 
$(V',N')\not \models c_i \leq x $ for all $i\geq\beta$ 
& since $(V,N)_x < w $ and $w = (V',N')_{c_i}$ for all $i\geq\beta$\\
$T'$ $\alpha$-omits $\Gamma$ & from \ref{ps:ls-7} and \ref{ps:ls-8}, since $p$ was arbitrarily chosen 
\end{proofsteps}
By Theorem~\ref{th:eott}, there exists a model $(V',N')$ which satisfies $T'$
and omits $\Gamma^s$ for all $s\in\Omega$. 
It follows that $\pos{(V',N')_s,(V',N')_\leq}$ is a linear ordering
of cofinality\footnote{To be more precise, we can select a strictly increasing
  subsequence 
$(c_{i_j}: i_j<\alpha)$ which is unbounded. This sequence is order-isomorphic
with an ordinal $\gamma$, and since $\alpha$ is regular we have
$\gamma=\alpha$. In particular 
$\card(C_s)\geq \alpha$ for each $s\in\Omega$. } 
$\alpha$ for all sorts $s\in\Omega$. 
Let $(V,N)\coloneqq (V',N')\red_\Delta$, and notice that $(V,N)$ satisfies $T$
and its carrier sets corresponding to the sorts in $\Omega$ have cardinalities
greater than or equal to $\alpha$. 
\end{proof}
\section{Omitting types and compactness}\label{9}
In this section, we show that at least at some occasions, compactness is a
necessary condition for proving the Omitting Types Theorem for uncountable signatures.
We work within a fragment $\L$ with the following properties:
\begin{enumerate}[label=P\arabic*)]
\item $\L$ is semantically closed under 
(a)~possibility applied to nominal sentences, 
(b)~retrieve, 
(c)~negation, 
(d)~disjunction, and 
(e)~quantifiers.

\item Signatures have only one rigid sort and all function symbols (except variables) are flexible.
\end{enumerate}
Notice that $\L$ is semantically closed under possibility,
as $\pos{\lambda}\varphi\modelsm \Exists{x}\pos{\lambda}x\wedge
\at{x}\varphi$.

\subsection{Global substitutions} 
We begin by defining a notion of substitution which we then use to derive
compactness for infinite models from $\alpha$-OTP using a technique originally
developed by Lindstr\"om for first-order logic with only relational
symbols~\cite{lind78}. 
Consider a signature $\Delta=(\Sigma^\nom,\Sigma^\rigid\subseteq \Sigma)$ with
only one rigid sort and no rigid function symbols, that is, $S^\nom=\{s_1\}$,
$S^\rigid=S=\{s_2\}$ and $F^\rigid=\emptyset$.  
We define another signature $\Delta_+=(\Sigma^\nom_+,\Sigma^\rigid_+\subseteq \Sigma_+)$ as follows:
\begin{enumerate}
\item $\Sigma^\nom_+$ consists of only one sort, let us say, $s_0$, and
  $\Sigma^\rigid_+$ consists of two sorts $s_1$ and $s_2$. 
\item $\Sigma_+$ is obtained from $\Sigma^\nom$ by adding the following sets of flexible symbols:
\begin{enumerate}
\item $\{\sigma_+:s_1\underbrace{s_2\dots s_2}_{m-times}\to s_2\alt \sigma:\underbrace{s_2\dots s_2}_{m-times}\to s_2\in F\}$ and
\item $\{\pi_+:s_1\underbrace{s_2\dots s_2}_{m-times}\alt \pi:\underbrace{s_2\dots s_2}_{m-times}\in P\}$.
\end{enumerate}
\end{enumerate}
The signature $\Delta_+$ provides a local environment for encoding the Kripke structures over $\Delta$.
The following set of  sentences over $\Delta_+$ ensures that the interpretation of the rigid relation symbols in $\Delta$ is `locally rigid' in $\Delta_+$.
\[
\Gamma\coloneqq \{ \Forall{x_1,x_2,y_1,\dots,y_m}\pi_+(x_1,y_1,\dots,y_m) \Leftrightarrow \pi_+(x_2,y_1,\dots,y_m) \alt \pi:\underbrace{s_2\dots  s_2}_{m-times}\in P^\rigid\}
\]
Let $\mathtt{z}$ be a distinguished nominal variable for $\Delta_+$. 
We define a substitution $(\_)^+\colon\Delta\dra (\Delta_+[\z],\Gamma)$, that is,
\begin{enumerate}
\item a sentence function $(\_)^+\colon \Sen(\Delta)\to \Sen(\Delta_+[\z],\Gamma)$ and 
\item a reduct functor $(\_)^-\colon\Mod(\Delta_+[\z],\Gamma)\to \Mod(\Delta)$,
\end{enumerate}
such that the following global satisfaction condition holds:
$$(W^{\z\leftarrow w},M)\models \gamma^+\text{ iff } (W^{\z\leftarrow w},M)^-\models \gamma$$
for all Kripke structures $(W,M)\in|\Mod(\Delta_+,\Gamma)|$, all possible worlds $w\in |W|$ and all sentences $\gamma\in \Sen(\Delta)$.
\paragraph{Mapping on models}
Notice that a model in $|\Mod(\Delta_+,\Gamma)|$ can be regarded as a collection of Kripke structures over the signature $\Delta$.
Once $\z$ is assigned to a node,
the functor $(\_)^-$ extracts the Kripke structure corresponding to the node denoted by $\z$. 
Concretely,
the functor $(\_)^-\colon\Mod(\Delta_+[\z],\Gamma)\to\Mod(\Delta)$ maps each Kripke structure $(W^{\z\leftarrow w},M)\in|\Mod(\Delta_+[\z],\Gamma)|$ to $(W_w^-,M_w^-)\in|\Mod(\Delta)|$, where
\begin{enumerate}
\item~$W_w^-\coloneqq M_w\red_{\Sigma^\nom}$,
\footnote{Notice that $M_w\in|\Mod(\Sigma_+)|$ and $M_w\red_{\Sigma^\nom}$ is well-defined since $\Sigma^\nom\subseteq \Sigma_+$.}
\item the mapping $M_w^-:M_{w,s_1}\to |\Mod(\Sigma)|$ is defined as follows:
\begin{itemize}
\item For all $v\in M_{w,s_1}$, the carrier set $(M_w^-)_{v,s_2}$ is $M_{w,s_2}$.
\item For all $v\in M_{w,s_1}$ and all $\sigma:\underbrace{s_2\dots s_2}_{m-times}\to s_2\in F$, 
the function $(M_w^-)_{v,\sigma}\colon \underbrace{M_{w,s_2}\times\dots\times  M_{w,s_2}}_{m-times}\to M_{w,s_2}$ is defined by 
$(M_w^-)_{v,\sigma}(a_1,\dots,a_m)\coloneqq M_{w,\sigma_+}(v,a_1,\dots,a_m)$ for all $a_1,\dots,a_m\in M_{w,s_2}$.
\item For all $v\in M_{w,s_1}$ and all $\pi:\underbrace{s_2\dots s_2}_{m-times}\in P$, the relation $(M_w^-)_{v,\pi}\subseteq \underbrace{M_{w,s_2}\times\dots\times M_{w,s_2}}_{m-times}$ is defined by
$(M_w^-)_{v,\pi}\coloneqq \{(a_1,\dots,a_m)\alt (v,a_1,\dots,a_m)\in M_{w,\pi_+}\}$.
\end{itemize}
\end{enumerate}
Since $(W,M)\models \Gamma$, the Kripke structure $(W_w^-,M_w^-)$ interprets all rigid symbols in $P^\rigid$ uniformly across the worlds, which means it is well-defined.
\begin{fact}
The functor $(\_)^-\colon\Mod(\Delta_+[z],\Gamma)\to\Mod(\Delta)$ can be extended to 
$(\_)^-\colon\Mod(\Delta_+[z,X],\Gamma)\to\Mod(\Delta[X])$, where 
$X=\{X_s\}_{s\in S^\ext}$ is a set of variables for $\Delta$,
such that the interpretation of all variables in $X$ is preserved, that is, 
$(W^{z\leftarrow w},M)_x=(W^{z\leftarrow w},M)_x^-$ for all $x\in X$.
\end{fact}
\paragraph{Mapping on sentences}
 We define a mapping on sentences $(\_)^+\colon\Sen(\Delta[X])\to\Sen(\Delta_+[\z,X])$ in three steps, 
where $X=\{X_s\}_{s\in S^\ext}$ is any set of variables for $\Delta$.
\begin{enumerate}[label=S\arabic*)]
\item We define a mapping from the rigid terms over $\Delta[X]$ to the rigid terms over $\Delta_+[\z,X]$ by structural induction:
\begin{itemize}
\item 
$x^+\coloneqq x$, where $x$ is any variable of rigid sort from $X$ , and
\item $(\at{k}\sigma)(t_1,\dots,t_m)^+\coloneqq (\at{\z}\sigma_+)(\at{\z} k,t_1^+,\dots,t_m^+)$, 
where $k\in F^\nom\cup X_\any$, and 
$t_i$ are rigid terms over $\Delta_+[\z,X]$.
\end{itemize}
Notice that $(\_)^+$ is well-defined on rigid terms, as the set of rigid function symbols is empty.
\begin{lemma}
For all Kripke structures $(W,M)\in|\Mod(\Delta_+[X],\Gamma)|$, all possible worlds $w\in|W|$, and all rigid terms $t$ over $\Delta[X]$,
\begin{equation}
(W^{\z\leftarrow w},M)_{t^+}= (W^{\z\leftarrow w},M)^-_{t}.
\end{equation}
\end{lemma}
\begin{proof}
By structural induction on terms:
\begin{proofcases}
\item[$x\in \{X_s\}_{s\in S^\rigid}$] 
Obviously, $(W^{\z\leftarrow w},M)_x= (W^-_w,M^-_w)_x$.
\item[$\at{k}\sigma(t_1,\dots,t_m)$]  
Let $v\coloneqq (W^{\z\leftarrow w},M)_{\at{\z} k}=M_{w,k}$, and we have:

$(W^{\z\leftarrow w},M)_{(\at{\z}\sigma_+)(\at{\z} k,t_1^+,\dots,t_m^+)}=$
$M_{w,\sigma_+}(v,(W^{\z\leftarrow w},M)_{t_1^+},\dots, (W^{\z\leftarrow w},M)_{t_m^+})$.
By the induction hypothesis,

$M_{w,\sigma_+}(v,(W^{\z\leftarrow w},M)_{t_1^+},\dots, (W^{\z\leftarrow w},M)_{t_m^+})=$
$(M_w^-)_{v,\sigma}((W^-_w,M^-_w)_{t_1},\dots,(W^-_w,M^-_w)_{t_m})=$

$(W^-_w,M^-_w)_{(\at{k}\sigma)(t_1,\dots,t_m)}$.
\end{proofcases}
Since $F^\rigid=\emptyset$, the cases considered above cover all possibilities.
\end{proof}
\item We define the mapping $(\_)^+$ on rigid sentences of the form $\at{k}\varphi\in\Sen(\Delta[X])$ such that 
every rigid sentence will be mapped to a rigid sentence $(\at{k}\varphi)^+\in \Sen(\Delta_+[\z,X])$, which means that
\[
(W^{\z\leftarrow w},M)\models (\at{k}\varphi)^+ \text{ iff } (W^{\z\leftarrow w},M)\models^w (\at{k}\varphi)^+
\]
for all Kripke structures $(W,M)\in|\Mod(\Delta_+[X],\Gamma)|$ and all possible worlds $w\in|W|$.
We proceed by structural induction.

\begin{tabular}{l l}
\begin{minipage}{0.55\textwidth}
\begin{itemize}
\item[]
\item
$(\at{k} k')^+\coloneqq \at{\z} (k = k')$
\item
$(\at{k}\pos{\lambda}k')^+\coloneqq \at{\z}\lambda(k,k')$
\item
$(\at{k}(t_1 = t_2))^+\coloneqq  (\At{k} t_1)^+ = (\At{k} t_2)^+$
\item
$(\at{k}\pi(t_1,\dots,t_m))^+\coloneqq (\at{\z}\pi_+)(\at{\z} k,(\At{k}t_1)^+,\dots,(\At{k}t_m)^+)$
\end{itemize}
\end{minipage}
&
\begin{minipage}{0.45\textwidth}
\begin{itemize}
\item
$(\at{k}\vee \Phi)^+\coloneqq \vee_{\varphi\in\Phi} (\at{k}\varphi)^+$
\item
$(\at{k}\neg\varphi)^+\coloneqq \neg(\at{k}\varphi)^+$
\item
$(\at{k}\Exists{X'}\varphi)^+ \coloneqq \Exists{X'}(\at{k}\varphi)^+$
\item 
$(\at{k} \at{k'}\varphi)^+\coloneqq (\at{k'}\varphi)^+$
\item $(\at{k}\store{x}\varphi)^+\coloneqq (\at{k}\varphi(x\leftarrow k))^+$
\end{itemize}
\end{minipage}
\end{tabular}

\begin{lemma} [Rigid satisfaction condition] \label{lemma:R-SC}
For all sentences $\varphi\in\Sen(\Delta[X])$, 
all nominals $k\in F^\nom\cup X_\any$, 
all Kripke structures $(W,M)\in|\Mod(\Delta_+[X],\Gamma)|$ and 
all possible worlds $w\in|W|$, 
\begin{equation} \label{eq:P2}
(W^{\z\leftarrow w},M) \models (\at{k}\varphi)^+\text{ iff } (W^{\z\leftarrow w},M)^- \models \at{k}\varphi.
\end{equation}
\end{lemma}
\begin{proof}
Let $v\coloneqq M_{w,k}$.
We proceed by induction on the structure of $\varphi$:
\begin{proofcases}[itemsep=1ex]
\item[$k'\in {F^\nom}\cup X_\any$]
$(W^{\z\leftarrow w},M)\models(\at{k} k')^+$ iff 
$(W^{\z\leftarrow w},M)\models \at{\z} (k = k')$ iff 
$(W^{\z\leftarrow w},M)\models^w k = k'$ iff
$M_{w,k}=M_{w,k'}$ iff
$(M_w^-)_k=(M_w^-)_{k'}$ iff
$(W^-_w,M^-_w)\models \at{k}k'$.
\item[$\pos{\lambda}k'$]
$(W^{\z\leftarrow w},M)\models (\at{k}\pos{\lambda}k')^+$ iff 
$(W^{\z\leftarrow w},M)\models \at{\z}\lambda(k,k')$ iff 
$(W^{\z\leftarrow w},M)\models^w \lambda(k,k')$ iff
$(M_k,M_{k'})\in M_{w,\lambda}$ iff 
$(W^-_w,M^-_w)\models \at{k}\pos{\lambda}k'$.
\item[$t_1=t_2$] 
$(W^{\z\leftarrow w},M)\models (\at{k}(t_1 = t_2))^+$ iff
$(W^{\z\leftarrow w},M)\models (\At{k} t_1)^+ = (\At{k} t_2)^+$ iff

$(W^{\z\leftarrow w},M)_{(\At{k} t_1)^+} = (W^{\z\leftarrow w},M)_{(\At{k} t_2)^+}$ iff
$(W^-_w,M^-_w)_{\At{k} t_1}=(W^-_w,M^-_w)_{\At{k} t_2}$ iff
$(W^-_w,M^-_w)_{\at{k} t_1}=(W^-_w,M^-_w)_{\at{k} t_2}$
$(W^-_w,M^-_w)\models \at{k} (t_1= t_2)$.
\item[$\pi(t_1,\dots,t_m)$] 
$(W^{\z\leftarrow w},M)\models (\at{k}\pi(t_1,\dots,t_m))^+$ iff
$(W^{\z\leftarrow w},M)\models \at{\z}\pi_+(\at{\z} k,(\At{k}t_1)^+,\dots,(\At{k}t_m)^+)$ iff

$(v,(W^{\z\leftarrow w},M)_{(\At{k}t_1)^+},\dots,(W^{\z\leftarrow w},M)_{(\At{k}t_m)^+})\in M_{w,\pi_+}$ iff
$((W^-_w,M^-_w)_{\At{k} t_1},\dots,(W^-_w,M^-_w)_{\At{k}t_m})\in (M_w^-)_{v,\pi}$ iff
$((W^-_w,M^-_w)_{\at{k} t_1},\dots,(W^-_w,M^-_w)_{\at{k}t_m})\in (M_w^-)_{\at{k}\pi}$ iff
$(W^-_w,M^-_w)\models \at{k}\pi(t_1,\dots,t_m)$.
\item[$\neg\varphi$]
$(W^{\z\leftarrow w},M)\models (\at{k}\neg\varphi)^+$ iff 
$(W^{\z\leftarrow w},M)\models \neg (\at{k}\varphi)^+ $ iff 
$(W^{\z\leftarrow w},M)\models^w \neg (\at{k}\varphi)^+$ iff

$(W^{\z\leftarrow w},M) \not\models^w (\at{k}\varphi)^+$ iff
$(W^{\z\leftarrow w},M) \not\models (\at{k}\varphi)^+$ iff
$(W^-_w,M^-_w) \not\models \at{k}\varphi$ iff
$(W^-_w,M^-_w) \models \at{k}\neg\varphi$.
\item[$\vee\Phi$]
$(W^{\z\leftarrow w},M)\models (\at{k}\vee \Phi)^+$ iff
$(W^{\z\leftarrow w},M)\models \vee_{\varphi\in\Phi} (\at{k}\varphi)^+$ iff
$(W^{\z\leftarrow w},M)\models^w \vee_{\varphi\in\Phi} (\at{k}\varphi)^+$ iff
$(W^{\z\leftarrow w},M)\models^w (\at{k}\varphi)^+$ for some $\varphi\in\Phi$ iff
$(W^-_w,M^-_w)\models \at{k}\varphi$  for some $\varphi\in\Phi$ iff
$(W^-_w,M^-_w)\models \at{k}\vee_{\varphi\in\Phi}\varphi$.
\item[$\Exists{X'}\varphi$] 
Let $(V,N)\coloneqq (W^-_w,M^-_w)$.
Since $(\_)^-$ preserves the interpretation of variables, we have:
\begin{enumerate}
\item for any expansion $(W',M')$ of $(W,M)$ to $\Delta_+[X, X']$,
$(W'^{z\leftarrow w}, M')^-$ is an expansion of $(V,N)$ to $\Delta[X, X']$,
\item for any expansion $(V',N')$ of $(V,N)$ to $\Delta[X, X']$,
there exists an expansion $(W', M')$ of $(W,M)$ to $\Delta_+[X, X']$ such that $(W'^{z\leftarrow w}, M')^-=(V',N')$.
\end{enumerate}
$$\xymatrix{
(V',N') & \Delta[X, X'] \ar@{.>}[rr]^{(\_)^+} & & \Delta_+[\z, X, X'] & (W'^{z\leftarrow w},M')\\
(V,N) & \Delta[X, X'] \ar@{^{(}->}[u] \ar@{.>}[rr]^{(\_)^+} & & \Delta_+[\z,X] \ar@{^{(}->}[u] & (W^{z\leftarrow w},M)\\
& \Delta \ar@{^{(}->}[u] \ar@{.>}[rr]^{(\_)^+}  & & \Delta_+[\z] \ar@{^{(}->}[u] \\
}$$
Based on the remark above, the following are equivalent:
\begin{proofsteps}{32em}
$(W^{\z\leftarrow w},M)\models (\at{k}\Exists{X'}\varphi)^+$ & \\
$(W^{\z\leftarrow w},M)\models \Exists{X'}(\at{k}\varphi)^+$  & by the definition of $(\_)^+$\\
$(W'^{\z\leftarrow w},M')\models (\at{k}\varphi)^+$ for some expansion $(W',M')$ of $(W,M)$ to $\Delta_+[X, X']$ 
& since $(\at{k}\varphi)^+$ is rigid\\
$(V',N')\models \at{k}\varphi $ for some expansion $(V',N')$ of $(V,N)$ to $\Delta[X, X']$ & 
by the induction hypothesis  \\
$(V,N)\models \at{k} \Exists{X'}\varphi $ &
since $\at{k}\varphi^+$ is rigid
\end{proofsteps}
\item[$\at{k'}\varphi$] 
This case is straightforward, since $\at{k}\at{k'}\varphi\modelsm \at{k'}\varphi$.
\item[$\store{x}\varphi$] 
This case is straightforward, since $\at{k}\store{x}\varphi\modelsm \at{k}\varphi[x\leftarrow k]$.
\end{proofcases}
\end{proof}
\item The function $(\_)^+\colon\Sen(\Delta[X])\to\Sen(\Delta_+[\z,X])$ is defined by
$\varphi^+=\Forall{\x}(\at{\x}\varphi)^+$ for all $\varphi\in\Sen(\Delta[X])$, where 
$\x$ is a distinguished nominal variable for $\Delta[X]$.
\begin{proposition}[Global satisfaction condition] \label{prop:GSC}
For all sentences $\varphi\in\Sen(\Delta[X])$, 
all Kripke structures $(W,M)\in|\Mod(\Delta_+[X])|$, and 
all possible worlds $w\in|W|$,
\begin{equation} \label{eq:P3}
(W^{\z\leftarrow w},M)\models \varphi^+ \text{ iff }
(W^{\z\leftarrow w},M)^- \models \varphi.
\end{equation}
\end{proposition}
\begin{proof}
Let $(V,N)\coloneqq (W^-_w,M^-_w)$.
$$\xymatrix{
(V^{\x\leftarrow v},N) & 
\Delta[\x,X] \ar@{.>}[rr]^{(\_)^+} & &
\Delta_+[\z,\x,X] & (W^{\z\leftarrow w},M^{\x\leftarrow v}) \\
(V,N) & \Delta[X] \ar@{^{(}->}[u] \ar@{.>}[rr]^{(\_)^+} & & 
\Delta_+[\z,X] \ar@{^{(}->}[u]  & (W^{\z\leftarrow w},M)\\
& \Delta \ar@{^{(}->}[u] \ar@{.>}[rr]^{(\_)^+}  & & \Delta_+[\z] \ar@{^{(}->}[u] \\
}$$
The following are equivalent:
\begin{proofsteps}{27em}
$(W^{\z\leftarrow w},M)\models \varphi^+$ & \\
$(W^{\z\leftarrow w},M)\models \Forall{\x}(\at{\x}\varphi)^+$ & by the definition of $(\_)^+$ \\
$(W^{\z\leftarrow w},M^{\x\leftarrow v})\models (\at{\x}\varphi)^+$  for any expansion $(W^{\z\leftarrow w},M^{\x\leftarrow v})$ of $(W^{\z\leftarrow w},M)$ to $\Delta_+[\z,\x,X]$
& since $(W^{\z\leftarrow w},M)\models \Forall{\x}(\at{\x}\varphi)^+$ \\
$(V^{\x\leftarrow v},N)\models \at{\x}\varphi$ for any expansion $(V^{\x\leftarrow v},N)$ of $(V,N)$ to $\Delta[\x,X]$ 
 & by Lemma~\ref{lemma:R-SC}, since $(W^{\z\leftarrow w},M^{\x\leftarrow v})^-=(V^{\x\leftarrow v},N)$ \\
$(V,N)\models \Forall{\x}\at{\x}\varphi$ & by semantics\\
$(V,N)\models \varphi$ & by semantics
\end{proofsteps}
\end{proof}
\end{enumerate}
\subsection{Inf-compactness}
We say that $\L$ is \emph{inf-compact} if each set of sentences $\Gamma$ has an infinite model whenever each finite subset $\Gamma_f\subseteq \Gamma$ has an infinite model.
We say that $L$ is \emph{$\alpha$-inf-compact}, where $\alpha$ is an infinite cardinal, if each set of sentences $\Gamma$ of cardinality $\alpha$ has an infinite model whenever each finite subset $\Gamma_f\subseteq \Gamma$ has an infinite model.
We show that inf-compactness is a consequence of omitting type property.

\begin{theorem}
If $\L$ has $\alpha$-OTP, where $\alpha$ is a regular cardinal
then $\L$ is $\beta$-inf-compact for all cardinals $\beta<\alpha$.
\end{theorem}
\begin{proof}
Let $\Delta$ be a signature of power at most $\alpha$.
By induction, it suffices to prove that 
each sequence $\Phi_\beta=\{\varphi_i\in\Sen(\Delta) \alt i<\beta\}$ has an infinite model 
whenever each subsequence $\Phi_j\coloneqq \{\varphi_i\alt i< j\}$ has an infinite model for all $j<\beta$.
Let $\{(W^i,M^i)\in |\Mod(\Delta)|  \alt 0 <  i<\beta \}$ be a sequence of Kripke structures over $\Delta$ such that 
\begin{itemize}
\item the carrier sets of $(W^i,M^i)$ are infinite for  all indexes $j$ with $0<j<\beta$, and
\item $(W^j,M^j)\models \Phi_j$ for all indexes $j$ with $0<j<\beta$.
\end{itemize}
By L\"owenheim-Skolem properties, we can assume that all carrier sets of $(W^i,M^i)$ are of cardinality $\alpha$.
By renaming the elements, we assume furthermore that $|W^i|=|W^j|$ and $M^i_{w,s_2}=M^j_{w,s_2}$ for all $i<j<\beta$ and all possible worlds $w\in |W^i|$.
We define the following Kripke structure $(W^+,M^+)$ over $\Delta_+$:
\begin{itemize}
\item $|W^+|=\{w_i \alt  0< i < \beta \}$, where $\{w_i\alt 0< i < \beta\}$ is a sequence of pairwise distinct possible worlds. 
The carrier sets of $(W^+,M^+)$ for the sorts $s_1$ and $s_2$ are the carrier sets of $(W^i,M^i)$ for the sorts $s_1$ and $s_2$, where $0<i<\beta$. 
\item For all $k\in F^\nom$ and all $0<i<\beta$, we define $M^+_{w_i,k}\coloneqq W^i_k$.

\item For all $\sigma:\underbrace{s_2\dots s_2}_{m-times}\to s_2\in F$ and all $0<i<\beta$, 
the function $M^+_{w_i,\sigma^+}:M^+_{w_i,s_1}\times \underbrace{M^+_{w_i,s_2}\times \dots \times M^+_{w_i,s_2}}_{m-times}\to M^+_{w_i,s_2}$ is defined by $M^+_{w_i,\sigma_+}(a,b_1,\dots,b_m)=M^i_{a,\sigma}(b_1,\dots,b_n)$ for all $(a,b_1,\dots,b_m)\in M^+_{w_i,s_1}\times \underbrace{M^+_{w_i,s_2}\times \dots \times M^+_{w_i,s_2}}_{m-times}$.
\item For all $\pi:\underbrace{s_2\dots s_2}_{m-times}\in P$, we define $M^+_{w_i,\pi}\coloneqq \{ (a,b_1,\dots,b_m) \alt (b_1,\dots,b_m)\in M^i_{a,\pi}\}$. 
\end{itemize}
By the definition of $(W^+,M^+)$, we have   
\begin{equation} \label{eq:P4}
((W^+)^{z\leftarrow w_i},M^+)^-=(W^i,M^i)\text{ for all }i<\beta.
\end{equation}
Let $\Delta_\bullet$ be the signature obtained from $\Delta_+$ by adding a set of new nominals $C=\{k_i\alt 0<i<\beta\}$ and 
a new binary modality $\leq$.
Let $(W^\bullet,M^\bullet)$ be the expansion of $(W^+,M^+)$ to $\Delta_\bullet$ such that
\begin{enumerate}[label=(\alph*)]
\item $W^\bullet_{k_i}=w_i$ for all ordinals $i$ with $0<i<\beta$, and
\item $(w_i,w_j)\in W^\bullet_<$ iff $i<j$.
\end{enumerate} 
Let $T=\Gamma\cup \{\Forall{\z} \at{k_i} \pos{<}\z\Rightarrow \varphi^+_i(\z)  \alt i<\beta \}$.
We show that $(W^\bullet,M^\bullet)\models T$:
\begin{proofsteps}{26em}
\label{ps:ict-0} $(W^\bullet,M^\bullet)\models \Gamma$ & since $(W^+,M^+)\models \Gamma$\\
let $i$ be an ordinal such that $0<i<\beta$ & \\
\label{ps:ict-2} let $w_j\in |W^\bullet|$ such that $(w_i, w_j)\in W^\bullet_<$,
\rlap{meaning that $((W^{\bullet})^{\z\leftarrow w_j},M^\bullet)\models \at{k_i}\pos{<}z$} & \\
$i< j$ & by the definition of $(W^\bullet,M^\bullet)$, since $(w_i, w_j)\in W^\bullet_<$ \\
$((W^+)^{\z\leftarrow w_j},M^+)\models \Phi^+_j$ & 
by Proposition~\ref{prop:GSC} and statement~\ref{eq:P4}, since $(W^j,M^j)\models \Phi_j$ \\
$((W^\bullet)^{\z\leftarrow w_j},M^\bullet)\models \Phi^+_j$ & by the satisfaction condition\\
\label{ps:ict-6} $((W^\bullet)^{\z\leftarrow w_j},M^\bullet)\models \varphi^+_i$ & 
since $\varphi_i\in\Phi_j$\\
\label{ps:ict-7} $(W^\bullet,M^\bullet)\models \Forall{z}\at{k_i} \pos{<} \z \Rightarrow \varphi_i^+$ & 
from \ref{ps:ict-2} and \ref{ps:ict-6}\\
$(W^\bullet,M^\bullet)\models T$ & from \ref{ps:ict-0} and \ref{ps:ict-7}
\end{proofsteps}
By Theorem~\ref{th:ULS}, there exists a model $(V^\bullet,N^\bullet)$ of $T$ such that $(V^\bullet,V^\bullet_\leq)$ is of confinality $\alpha$.
We define $v_i\coloneqq V^\bullet_{k_i}$ for all $i<\beta$.
By cofinality, there exists $v\in V^\bullet_{s_0}$ such that $(v_i,v)\in V^\bullet_<$ for all $i<\beta$.
It follows that $((V^\bullet)^{\z\leftarrow v},N^\bullet) \models \varphi_i^+$ for all $i<\beta$.
Let $(V^+,N^+)\coloneqq (V^\bullet,N^\bullet)\red_{\Delta_+}$.
By the satisfaction condition, $((V^+)^{\z\leftarrow v},N^+)\models \varphi^+_i$ for all $i<\beta$.
By Proposition~\ref{prop:GSC}, $((V^+)^{\z\leftarrow v},N^+)^- \models \varphi_i$ for all $i<\beta$. 
\end{proof}

\section{Conclusion}

In this paper we established an omitting types theorem for first-order hybrid dynamic  logic and sufficiently expressive fragments. For countable signatures, the result followed without needing compactness whereas for uncountable signatures we had to restrict our attention to compact fragments of the logic.  It turns out that the latter restriction is actually necessary for some of these fragments, as compactness is a consequence of OTT for uncountable signatures. We also provided two applications of the OTT: (1) L\" owenheim-Skolem theorems and (2) a completeness theorem for the constructor-based version of  first-order hybrid dynamic  logic. In future work we intend to explore other interesting consequences of OTT in this setting, particularly the Robinson Joint Consistency theorem.

\section*{Acknowledgments}
This paper grew out of some lectures given by George Georgescu on forcing while the first author was a master student at the University of Bucharest. 
The work presented in this paper has been partially supported by the Japanese Contract Kakenhi 20K03718.

\bibliographystyle{ACM-Reference-Format}
\bibliography{ott}
\end{document}